\documentclass[12pt]{report}
\usepackage{graphicx}
\usepackage[utf8x]{inputenc}
\usepackage{amsmath,sectsty,amsthm,amssymb,fontenc,color,tensor,textcomp,amsfonts,relsize,graphicx,graphics,enumerate}
\usepackage{fancyhdr}
\usepackage[colorlinks=true,
            linkcolor=red,
            urlcolor=blue,
            citecolor=green]{hyperref}
\usepackage{mathtools}
\usepackage{amsmath,tikz}
\usepackage{tensor}
\usepackage[rmargin=3.5cm, lmargin=3.5cm]{geometry}
\newtheorem{theo}{Theorem}[section]
\newtheorem{lem}[theo]{Lemma}
\newtheorem{exa}[theo]{Example}
\newtheorem{prop}[theo]{Proposition}
\newtheorem{cor}[theo]{Corollary}
\newtheorem{df}[theo]{Definition}
\newtheorem{rmk}[theo]{Remark}
\newtheorem{cl}[theo]{Claim}
\newtheorem{prob}{Problem}

\newcommand\norm[1]{\left\lVert#1\right\rVert}

\begin{document}\sloppy
\vspace{1cm}

\begin{center}
\textbf{\LARGE Study of a four dimensional Willmore-type energy }
\vskip 8mm
\textbf{\large Peter Olamide Olanipekun }
\vskip3mm
B.Sc (Hons)(Mathematics)
\\
M.Sc (Mathematics)
\end{center}
\vskip15mm
\begin{center}\large
A thesis submitted  for the degree of \\{\bf Doctor of Philosophy} \\ 
at Monash University
\\
2021

\vskip15mm
School of
Mathematics
\\Faculty of Science\\ Monash University\\Melbourne, Australia
\end{center}
\pagenumbering{roman}
\thispagestyle{empty}
\newpage
\mbox{~}
\newpage
\noindent
{\bf Copyright notice}
\\
© Peter Olamide Olanipekun (2021) 




\newpage
\chapter*{Abstract}
In this thesis, a four dimensional conformally invariant energy is studied. This energy generalises the well known two-dimensional Willmore energy. Although not positive definite, it includes minimal hypersurfaces as critical points. We compute its first variation  and by applying Noether’s theorem to the invariances, we derive some conservation laws which are satisfied by its critical points and with good analytical dispositions. In particular, we show that critical points are smooth. We also investigate other possible four dimensional generalisations of the Willmore energy, and give strong evidence that critical points of such energies do not include minimal hypersurfaces.

\newpage
\chapter*{Decalaration}
This thesis is an original work of my research and contains no material which has been accepted for the award of any other degree or diploma at any university or equivalent institution and that, to the best of my knowledge and belief, this thesis contains no material previously published or written by another person, except where due reference is made in the text of the thesis.
\vskip3mm
\noindent

\vskip3mm
\noindent
Print Name: \hskip2mm Peter Olamide Olanipekun

\vskip3mm
\noindent
Date: \hskip1.4cm  19 September, 2021

\tableofcontents

\pagenumbering{arabic}
\chapter{Introduction}\label{chap1}
The goal of this Chapter is to present, in a simple (without technical details)  fashion, the problems studied in this thesis. Relevant technical details will be sorted later in the appropriate Chapters. The main object of study in this thesis is an  hypersurface critical for a certain four dimensional conformally invariant Willmore-type energy. 

\section{Brief historical context}

The Willmore energy, was originally defined for two dimensional manifolds (or simply surfaces). The Willmore energy (also known as the Willmore functional) of a surface $\Sigma$ immersed into the Euclidean space is given by the integral of the square of the mean curvature vector of $\Sigma$. That is,
\begin{align}
\frac{1}{4}\int_{\Sigma} |\vec{H}|^2 d\textnormal{vol}_g   \label{realwil}
\end{align}
where $d\textnormal{vol}_g$ is the area element.
In broad terms, it measures the sphericity of a given surface. 
\vskip3mm
\noindent
The origin of the Willmore energy can be traced to Sim\'eon-Denis Poisson and Sophie Germain who had engaged in a scientific contest to give a mathematical explanation of what is now known as the Chladni figures. 
The prize was won in 1816 by Sophie Germain. 
She had studied the elasticity properties of membranes and thin plates and her work brought into limelight an entirely new area of applied mathematics, namely, elasticity theory.   Her work  shows that the bending energy of a thin plate, in its simplest form, can be written (using our notation) as \eqref{realwil} (see \cite{ger}). Problems involving elastic surfaces can be studied via the energy \eqref{realwil} (see \cite{friesecke}). Since then, the Willmore energy has been widely studied in differential geometry and applied to some  areas of science such as General Relativity \cite {swhawking, huiskenil} (where the Willmore energy is the main contributor  to the Hawking mass), Cell Biology \cite{helfrichw} (where the Willmore energy plays a major role in the Helfrich model), String theory \cite{polyakov}, and Image Processing \cite{katzman}. The Willmore energy also relates  to the renormalised area functional for a minimal surfaces embedded in hyperbolic 3-manifold \cite{spyridon}.
Earlier works include those of G.\ Thomsen  who established the Euler-Lagrange equation for the energy \eqref{realwil} \cite{tho}, W.\ Blaschke  who observed minimal surfaces are minimizers of the energy  \eqref{realwil} \cite{bla}, J.\ Weiner who generalised the work of Thomsen in higher codimension \cite {wei} and T.\ J.\ Willmore who explored the lower bounds for the Willmore energy of a closed orientable surface \cite{willmoretom}. Actually, the Willmore energy was named after him. 
Of great importance is the Willmore conjecture (\cite{willmoretom}, see also \cite{lipeter}) which was proved in \cite{fcmarqu} by Fernando Cod\'a Marques and Andr\'e Neves.

\begin{theo}
Let $\Sigma\subset S^3$ be an embedded closed surface of genus $g\geq 1$. Then
$$\int_\Sigma (1+|\vec{H}|^2) \,d\textnormal{vol}_g \geq 2\pi^2$$
and the equality holds if and olnly if $\Sigma$ is the Clifford torus up to conformal transformations of $S^3$.
\end{theo}

\noindent
Other notable contributions are those of B.\ Y.\ Chen \cite{{chenny}, {chenny1}}, J.\ H.\ White \cite{white}, R.\ L.\ Bryant \cite{bry}, L.\ Simon \cite{simmmy}, U.\ Pinkall \cite{pinkyy}, R.\ Kusner \cite{kus}, B.\ Palmer \cite{{palmwi}, {palmwi1}}, just to mention afew.

\noindent
Recent efforts have been dedicated to the analysis of the critical points of the energy \eqref{realwil} under different conditions: in codimension one \cite{cdd,stru}, higher codimension \cite{noetherpaper}, in a fixed conformal class \cite{{boh}, \cite{sc}} and under other frameworks \cite{yber2, ybr, yber1,  br, riv}.

\section{The problem}
Let $g$ and $\vec{h}$ denote  the first and second fundamental form, respectively, on the immersed surface $\Sigma$. Let $\vec{h}_0:=\vec{h}-\vec{H}g$ \,\,be the trace-free second fundamental form. Let $\chi(\Sigma)$ be the Euler characteristic of the surface $\Sigma$. Owing to the Gauss-Bonnet theorem, the Willmore energy can be written as
\begin{align}
\int_\Sigma|\vec{H}|^2 d\textnormal{vol}_g = \frac{1}{4}\int_\Sigma|\vec{h}|^2 d\textnormal{vol}_g +\pi\chi(\Sigma) =\frac{1}{2}\int_\Sigma|\vec{h}_0|^2 d\textnormal{vol}_g +2\pi\chi(\Sigma)  \label{apps}
\end{align}
    so that the three energies appearing in \eqref{apps} have the same critical points. Although the energy $\int_\Sigma|\vec{h}_0|^2 d\textnormal{vol}_g$ is conformally invariant, the other two energies in \eqref{apps} are not.
\noindent
The critical points of the energy \eqref{realwil} are found via variation of the energy. They satisfy the following Euler-Lagrange equation which is known in literature as the Willmore equation: 
\begin{align}
\Delta_\perp\vec{H} +(\vec{H}\cdot\vec{h}_{ij})\vec{h}^{ij} -2|\vec{H}|^2\vec{H}=\vec{0} \label{eqya}
\end{align}
where $\Delta_\perp$ is the negative covariant Laplacian for the Levi-Civita connection in the normal bundle. We will refer to the left hand side of \eqref{eqya} as the {\it Willmore invariant}. In codimension one, the above equation reduces to:
\begin{align}
\Delta_g H+|h_0|^2 H=0  \label{eqya1}
\end{align}
where $\Delta_g$ is the negative Laplace-Beltrami operator on $\Sigma$.  Since the mean curvature depends on two derivatives of the immersion, the Willmore equation \eqref{eqya1} (or more generally \eqref{eqya}) is  a fourth order elliptic nonlinear partial differential equation. This leads us to the following interesting properties about the energy.
\begin{enumerate}[(a).] \label{ajaka}
\item\label{ajaka1} The Willmore energy is conformally invariant up to Gauss-Bonnet terms
\item \label{ajaku} The leading order operator in the Willmore equation is linear
\item \label{ajaki} Minimal surfaces are critical points of the Willmore energy
\item \label{ajaka2} The Willmore energy is non-negative.
\end{enumerate}

\subsection{Finding a four dimensional  Willmore-type energy}
Generalising the Willmore energy to higher dimensions has been a question of interest for a long time. It has been often asked what properties or conditions must be satisfied by an energy that would qualify as a generalised higher dimensional version of the Willmore energy. 
This essentially leads to the following problem in four dimensions:

\begin{prob}\label{probby}
Obtain a four dimensional Willmore-type energy preserving the properties \eqref{ajaka1}-\eqref{ajaka2} above.

\end{prob}
\noindent
By considering an energy involving a linear combination of the quantities
$$(\vec{h}_0)_{ij}\cdot (\vec{h}_0)^{jk}(\vec{h}_0)_{kl}\cdot (\vec{h}_0)^{li}\,, \quad (\vec{h}_0)_{ij}\cdot (\vec{h}_0)^{ij}(\vec{h}_0)_{kl}\cdot (\vec{h}_0)^{kl} \quad\mbox{and}\, ((\vec{h}_0)_{ij}\cdot (\vec{h}_0)^{jk}(\vec{h}_0)_{k}^i)^{4/3}$$
one can easily construct a four dimensional conformally invariant energy. However, preliminary computations show that such energies do not preserve the properties \eqref{ajaku} and \eqref{ajaki} above (see for instance those considered in \cite{rigoli} and in Section 5 of \cite{mondino}). Thus a different approach, which we shall now describe, is needed.

\vskip 3mm
\noindent
There is a relationship between Problem \ref{probby} and the singular Yamabe problem.  Let $(M^n, \bar g)$ be a smooth compact Riemannian manifold of dimension $n$ with boundary and let $\Sigma$ be an hypersurface, the singular Yamabe problem asks for a defining function $u$  for the boundary $\Sigma:=\partial M$ such that on the interior $ \mathring{M}$, the scalar curvature $R_g$ of the metric $g=u^{-2}\bar g$ satisfies   $R_g=-n(n-1)$. 
This problem was first considered by Loewner and Nirenberg in \cite{loewner}. Later, 
Andersson, Chr\'usciel and Friedrich \cite{acf} computed a conformal invariant which obstructs the smoothness of the function $u$. This obstruction to smooth boundary asymptotics for the Yamabe solution  was identified by Gover et.\ al as a Willmore surface invariant in $\mathbb{R}^3$ (see \cite{{GGHW}, {GoW2}, {GoW3}}). This identification led to the perception that higher dimensional generalisation of the Willmore invariant can be found  via obstructions to smooth boundary asymptotics of the singular Yamabe problem on conformally compact manifolds of higher dimension. They further asked if there is a corresponding energy for such invariants, that is, is there a higher dimensional generalisation of the Willmore energy whose Euler-Lagrange equation corresponds to the identified obstruction? It turns out that the answer is positive (see \cite{{govie}, {gragra}}). Moreover, the following four dimensional generalisation of the Willmore energy was identified
\begin{align}
\mathcal{E}(\Sigma):=\int_{\Sigma} (|\pi_{\vec{n}}\nabla\vec{H}|^2- |\vec{H}\cdot\vec{h}|^2 +7|\vec{H}|^4) \,\,\,d\textnormal{vol}_g  \label{jkas}
\end{align}
where notations have been slightly adjusted as follows: $\vec{h}_{ij}:=\nabla_i\nabla_j\vec{\Phi}$, \,\,\,$\vec{H}:= \frac{1}{4}g^{ij}\vec{h}_{ij}=\vec{h}^j_j$, \,\,\, $|\pi_{\vec{n}}\nabla\vec{H}|:=\pi_{\vec{n}}\nabla_i\vec{H}\cdot\pi_{\vec{n}}\nabla^i\vec{H}$, where $\pi_{\vec{n}}\nabla\vec{H} $ is the projection of the Levi-Civita connection, acting on $\vec{H}$, onto the normal bundle, \,\,\, $|\vec{H}\cdot\vec{h}|^2:= (\vec{H}\cdot\vec{h}_{ij})(\vec{H}\cdot\vec{h}^{ij})$ and $|\vec{H}|^4:= (\vec{H}\cdot\vec{H})^2$. This energy was first discussed  in codimension one\footnote{In codimension one, the energy $\mathcal{E}(\Sigma)$ reduces to $\int_{\Sigma} (|\nabla H|-|H|^2|h|^2+7|H|^4)\,\,\,d\textnormal{vol}_g$.} by Guven in \cite{jguv} and rediscovered by Robin-Graham and Reichert in \cite{robingraham}, and separately by Zhang in \cite{zhang}. The energy \eqref{jkas} satisfies the conditions \eqref{ajaka1}-\eqref{ajaki} specified above; infact, it satisfies \eqref{ajaka1} without Gauss-Bonnet. However, as noted in \cite{robingraham}, it is not bounded below, thus violating condition \eqref{ajaka2}. It is possible to construct a non-negative energy by adding, to the energy \eqref{jkas}, a multiple of the fourth power of the norm of the trace-free second fundamental form. In such situation, one considers the energy:
\begin{align}
\mathcal{E}^1(\Sigma):= \mathcal{E}(\Sigma)+\mu\int_\Sigma    |h_0|^4\,\,d\textnormal{vol}_g  \label{generalpo}
\end{align}
for some  $\mu \geq\frac{1}{12}$. But one discovers that the energy $\mathcal{E}^1(\Sigma)$ violates the very crucial condition \eqref{ajaki} (see Section \ref{othergen}). Thus, in this thesis, we will focus on studying the critical points of the energy \eqref{jkas}.



\subsection{Conservation laws}
The two dimensional Willmore energy is finite as soon as  the Willmore immersion $\vec{\Phi}:\Sigma \rightarrow \mathbb{R}^m$ belongs to the Sobolev space $W^{1,\infty}\cap W^{2,2}$.  This implies that $\vec{h}\in L^2$.  
 This fact, when incorporated  into the Willmore equation, gives no analytical meaning as the cubic curvature terms $(\vec{H}\cdot\vec{h})\vec{h}-2|\vec{H}|^2\vec{H}$ (or more importantly $\Delta_\perp\vec{H}$ as per \eqref{eqya}) would not even belong to $L^1$. However, this issue can be resolved with the help of conservation laws.
In the seminal article \cite{riv}, the author found such conservation laws. It was shown that the fourth-order Willmore equation, associated with the two-dimensional Willmore energy, can be expressed in divergence forms with several analytical implications. These divergence forms were elegantly obtained, in \cite{noetherpaper}, by applying Noether's theorem (a general result applicable in any dimension) to the invariances of the Willmore energy and several other pre-existing results were rediscovered. Indeed, there exist some  1-forms $\vec{V}_{\textnormal{tra}}$, ${V}_{\textnormal{dil}}$ and $\vec{V}_{\textnormal{rot}}$ such that

\noindent
\begin{align}
\vec{\mathcal{W}}=d^\star \vec{V}_{\textnormal{tra}}=\vec{0}\,\,,\quad d^\star {V}_{\textnormal{dil}}={0}\,\,,\quad d^\star \vec{V}_{\textnormal{rot}}=\vec{0}   \label{aaq12} \end{align}
where $\vec{\mathcal{W}}$ is the Willmore invariant. The quantity $\vec{V}_{\textnormal{tra}}$ is  geometric since it depends on $\vec{\mathcal{W}}$. The quantities ${V}_{\textnormal{dil}}$ and $\vec{V}_{\textnormal{rot}}$ are also geometric since they depend on $\vec{V}_{\textnormal{tra}}$.
\noindent
Applying Poincar\'e lemma to \eqref{aaq12}, one finds the primitives $\vec{L}$, $S$  and $\vec{R}$ satisfying
\[   \left\{
\begin{array}{ll}
      d^\star \vec{L}=\vec{V}_{\textnormal{tra}} \,\,,\quad d^\star {S}={V}_{\textnormal{dil}}\,\,,&d^\star\vec{R}= \vec{V}_{\textnormal{rot}} \\
      d \vec{L}=\vec{0} \,\,,\quad\quad\quad d {S}={0}\,\,,&d\vec{R}= \vec{0} \\
\end{array} 
\right. \]
where the Hodge star $\star$, the exterior differential $d$ and codifferential $d^\star$ are understood in terms of the induced metric $g$. 
Further computations \cite{noetherpaper} yield the following system\footnote{The geometric product $\bullet$ is defined in Section \ref{latty} of Chapter \ref{latty1}.} of second-order partial differential equations:
\begin{align}  \left\{
\begin{array}{ll}
      |g|^{1/2}\Delta_g (\star S)&=\star\big(d(\star\vec{n})\cdot d(\star\vec{R})\big) \\
      |g|^{1/2}\Delta_g(\star\vec{R})&=\star\big(d(\star\vec{n})\bullet d(\star\vec{R})\big) +\star\big(d(\star\vec{n})\,\, d(\star{S})\big)  \\
|g|^{1/2}\Delta_g\vec{\Phi}&=\star\big(d(\star\vec{R})\bullet d\vec{\Phi}\big) +\star\big(d(\star{S})\,\, d\vec{\Phi}\big)  \\
\end{array} 
\right. \label{sisio5} \end{align}
where $\vec{n}$ denotes the Gauss map on $\Sigma$. The system \eqref{sisio5} possesses a Jacobian structure suitable for analysis \cite{{ybr}, {noetherpaper},  {palais}, {riv}}. These formulations and results lead us to ask if analogous results can be obtained in four dimensions. In particular, we are interested in the following problem. 

\begin{prob}
What conservation laws are satisfied by the critical points of the energy \eqref{jkas}?
\end{prob}

\noindent
Understanding  Willmore-type equations  in higher dimensions, from analytical perspective, can be generally a difficult problem. 
For instance, the Willmore-type equation associated with \eqref{jkas} is a sixth-order nonlinear partial differential equation, the study of which is rare in literature. 

\vskip3mm
\noindent
In Chapter \ref{latty1}, we will use Noether's theorem and the fact that the energy \eqref{jkas} is invariant under coordinate transformation to derive some conservation laws satisfied by the critical points of \eqref{jkas}. These laws are particularly helpful in analysing the complicated sixth-order Willmore-type equation. 

\subsection{Regularity of critical points}
The question of regularity of the two dimensional Willmore surface is  already settled in literature. For instance, it was proved in \cite{palais} and \cite{riv} that  Willmore immersions associated with the two dimensional Willmore energy are smooth. The proof techniques heavily rely on the structure of the system \eqref{sisio5}. The regularity assumptions on $\vec{\Phi}$ implies that $\vec{n}$, $S$ and $\vec{R}$ belong to the Sobolev space $W^{1,2}.$ Thus it follows that $\Delta_g (\star S)$, $\Delta_g(\star\vec{R})$ and $\Delta_g\vec{\Phi}$ all belong to $L^1$; this information  is critical for regularity purpose. However, the Jacobian structure of system \eqref{sisio5} encodes an extra regularity. By using Wente estimates (see Chapter 3 of \cite{ helein}, \cite{palais}), one improves the regularity of  $(\star S)$ and $(\star\vec{R})$ by showing that they both belong to $W^{1,p}$ for some $p>2$.
The third equation in \eqref{sisio5} links the primitives $(\star S)$ and $(\star\vec{R})$ back to the Willmore immersion $\vec{\Phi}$. In fact, it follows that $\vec{\Phi}$ belongs to $W^{2,p}$. By applying a bootstrapping argument, one concludes that  $\vec{\Phi}$ is smooth.
This leads us to ask if such procedure can be applied in four dimensions. Essentially, we have:
\begin{prob}\label{pplk}
Show that the critical points of the energy \eqref{jkas} are smooth.
\end{prob}

\noindent
It is interesting to note that the proof techniques employed in resolving Problem \ref{pplk}  is remarkably different from the regularity proof for the two dimensional Willmore surfaces. Unlike the two dimensional case, the corresponding expressions for $\vec{L}$, $S$ and $\vec{R}$ are much more intricate. Also, the system of equations does not seem to possess Jacobian structure, hence we do not use Wente estimates in our proof. However with the help of a Bourgain-Brezis type result and classical results of Miranda \cite{mir} and Di Fazio \cite{faz}, we are able to obtain preliminary results needed to begin the regularity proof.

\noindent
In summary, we have made a progress towards the analytical study of the Willmore-type energy \eqref{jkas}. This energy is new in literature and has not been analytically explored.

\section{Outline of chapters}
In Chapter \ref{chap1}, we introduced the problems considered in this thesis. The main object of interest is the four dimensional Willmore-type energy.
\vskip 3mm

\noindent
The content of Chapter \ref{latty1} is computational, and heavily employs conventions from tensors and differential forms. We use variational techniques to compute the Euler-Lagrange equation (which we shall often refer to as the {\it Willmore-Euler-Lagrange equation} or simply the {\it Willmore-type equation}\footnote{More specifically, we shall often refer to analogous Willmore concepts in four dimensions as {\it Willmore-type}. }) for the energy \eqref{jkas}; this is basically Theorem \ref{WE}. Theorem \ref{cons-law} gives conservation laws for the energy \eqref{jkas} in higher codimension (or for Willmore-type submanifolds). Conformal invariance of the energy is established in Theorem \ref{kofoworola}. In Proposition \ref{coro}, we restrict our computations to codimension one and establish conservation laws satisfied by the Willmore-type hypersurfaces. In view of Problem \ref{probby}, in Section \ref{othergen} we also present other possible four dimensional generalisations of the energy \eqref{jkas} and argue that the only energy whose critical points include minimal hypersurfaces is the energy \eqref{jkas}; the other energies violate the condition \eqref{ajaki}. This explains the reason for the choice of the energy \eqref{jkas} in this study.

\vskip3mm
\noindent
In Chapter \ref{chap4}, we take the first step towards the analysis of the energy \eqref{jkas}. We prove (Theorem \ref{heart}) that the Willmore-type hypersurface is indeed smooth. This is done in a series of steps organised into Sections \ref{adzsq} - \ref{linkage}.   In Chapter \ref{chap5}, we discuss other analysis aspects of the energy \eqref{jkas}, all of which are open problems to the best of our knowledge. Thus,  Chapter \ref{chap5} presents suggestions for further research.  We include an Appendix where  some basic  notions in Riemannian geometry relevant to our study are recalled and some useful computations are elaborated. 

\vskip 3mm

\noindent
{\bf Notational conventions.} Unless otherwise specified, we will adopt the following notations and conventions in this work. 
\begin{enumerate}
\item The Greek letter $\vec{\Phi}$ will be reserved for immersions of $\Sigma$. In two dimensions, $\Sigma$ will denote a closed surface but in four dimensions it will denote a closed hypersurface or submanifold, depending on the specified codimension.
\item The induced metric or the first fundamental form will be denoted by $g$ with  volume element $d\textnormal{vol}_g:=|g|^{1/2}dx$. The component of the first and second fundamental forms are $g_{ij}$ and $\vec{h}_{ij}$ respectively. The mean curvature is denoted $\vec{H}$ and is defined $\vec{H}:=\frac{1}{n}g^{ij}\vec{h}_{ij}$ where $n$ is the dimension of $\Sigma$.

\item We will use the Greek letter $\delta$ to denote the variation of a quantity. The Greek letter $\delta$ with indices will be reserved for the component of the flat metric. For example, $\delta g_{ij}$ is the variation of the $(i,j)$ component of the metric $g$; while $\delta^{i}_j $ is the $(i,j)$ component of the flat metric.

\item We adopt Einstein's summation convention for our computations. 

\item We denote by $\pi_{\vec{n}} (\vec{V})$ and $\pi_{T}(\vec{V})$ the normal and tangential projections respectively of some vector $\vec{V}$. We have $\pi_{\vec{n}}+\pi_{T}=\textnormal{id}$.

\item Also, we use the metric tensor $g_{ij}$  (or the inverse metric  tensor $g^{ij}$) to lower (or raise)  indices on tensors. For example, $\tensor{T}{_j_l}= g_{ij}g_{kl} T^{ik}$ and $\tensor{T}{^j^l}=g^{ij}g^{kl}T_{ik}$.

\item The operators $\Delta$, $\Delta_g$ and $\Delta_0$ will denote the Hodge Laplacian, the Laplace Beltrami and the flat Laplacian respectively. The operators $\Delta$ and $\Delta_g$ coincide when they act on ``scalar" functions.

\end{enumerate}


\chapter{Variational Derivations and Conservation Laws}\label{latty1}

\section{Introduction} \label{latty}

In this chapter, we derive the Willmore-Euler-Lagrange equation for the Robin-Graham-Reichert  energy $\mathcal{E}(\Sigma)$ of \eqref{jkas}. This system of sixth-order, elliptic, nonlinear equations was first obtained in \cite{robingraham} and \cite{zhang} in arbitrary codimension\footnote{There is a mistake in \cite{robingraham}, in higher codimension, which we will correct in due time.}. In order to study this system of  equations \footnote{We will study the Willmore-Euler-Lagrange equation in codimension one. However, our method works in higher codimension at the cost of really cumbersome notations.}, we follow Bernard's formalism \cite{noetherpaper} to establish conservation laws satisfied by the critical points of $\mathcal{E}(\Sigma)$ under translation, dilation and rotation. Indeed, it is clear that the energy $\mathcal{E}(\Sigma)$ is invariant under translation, dilation and rotation. From this, we deduce that the Willmore-type operator (the left hand side of \eqref{abc}) is actually the divergence of a specific vector field and we write
$$\mathcal{\vec{W}}=\nabla_j\vec{V}^j$$
where the quantities $\mathcal{\vec{W}}$ and $\vec{V}$ (the stress-energy tensor for $\mathcal{E}(\Sigma)$) are made precise in Sections  \ref{sekiki} and \ref{sekaka}.

\noindent
Special conformal transformation can be obtained by performing an inversion, a translation and an inversion again (in that order). Hence, we prove, in Section \ref{sekuku}, that $\mathcal{E}(\Sigma)$ is conformally invariant by showing that $\mathcal{E}(\Sigma)$ is invariant under special conformal transformation,  following the method of Guven in \cite{jguv}. 
The need for the stress energy tensor (obtained in Section \ref{sekaka}), in the proof of conformal invariance, explains why the content of Section \ref{sekuku} comes after Section \ref{sekaka}. 
Finally in Section \ref{sekoko}, we obtain further conservation laws by proving the existence of some 2-forms needed to write some useful equations.

\vskip 3mm
\noindent
The major achievement of this Chapter is the reduction of the complicated sixth-order elliptic nonlinear Willmore-Euler-Lagrange equation for $\mathcal{E}(\Sigma)$ to a  manageable system of second order  partial differential equations.
Although this Chapter is largely computational, it prepares the ground for the local analysis of critical points of $\mathcal{E}(\Sigma)$.

\vskip 3mm

\noindent
We now briefly explain some of the notations and conventions adopted in this Chapter. The {\it small dot operator} $\cdot$ is used to denote the usual dot product between vectors. We also use the small dot to denote the natural extension of the dot product in $\mathbb{R}^m$ to multivectors (see \cite{fed}). We define the bilinear map 
$\mathlarger{\mathlarger{\mathlarger{\mathlarger{\llcorner}} }} :\Lambda ^p(\mathbb{R}^m)\times \Lambda^q(\mathbb{R}^m)\rightarrow \Lambda^{p-q}(\mathbb{R}^m)$, known as the {\it interior multiplication} of multivectors, by
$$
\langle \vec{A}\,\,\mathlarger{\mathlarger{\mathlarger{\mathlarger{\llcorner}} }} \vec{B},\, \vec{C}\rangle_{\mathbb{R}^m}=\langle \vec{A}, \vec{B}\wedge \vec{C}\rangle_{\mathbb{R}^m} \label{se123}$$
where $\vec{A}\in \Lambda^p(\mathbb{R}^m)$, $\vec{B}\in \Lambda^q(\mathbb{R}^m)$ and $\vec{C}\in \Lambda^{p-q}(\mathbb{R}^m)$ with $p\geq q$. If $p=1=q$ then we have the usual dot product. Here, we have used the notation $\langle \cdot, \cdot \rangle_{\mathbb{R}^m}$ to emphasise the inner product of multivectors in $\mathbb{R}^m$. We will reserve the notation $\langle\cdot, \cdot\rangle$ for the inner product of forms on patches of the manifold,  this  will be defined and used in Chapter \ref{chap4}.

\noindent
The {\it bullet dot operator}   $\bullet: \Lambda^p(\mathbb{R}^m)\times \Lambda^q(\mathbb{R}^m)\rightarrow \Lambda^{p+q-2}(\mathbb{R}^m)$ is well known in geometric algebra as the first-order contraction operator. For a $k$-vector $\vec{A}$ and a $1$-vector $\vec{B}$, we define 
$$\vec{A}\bullet \vec{B}=\vec{A}\mathlarger{\mathlarger{\mathlarger{\mathlarger{\llcorner}} }} \vec{B}.$$
For a $p$-vector $\vec{B}$ and a $q$-vector $\vec{C}$, we have
$$\vec{A}\bullet (\vec{B}\wedge\vec{C})=(\vec{A}\bullet \vec{B})\wedge \vec{C}+ (-1)^{pq}(\vec{A}\bullet\vec{C})\wedge\vec{B}.$$

\noindent
Recall that we are adopting Einstein's summation convention for our computations. Also, we will use the metric tensor  (or the metric inverse) to lower (or raise)  indices on tensors.

\section {Willmore-Euler-Lagrange type equation }\label{sekiki}
\begin{df}
Let  $\vec{\Phi}:\Sigma\rightarrow\mathbb{R}^m $ be an immersion of a 4-dimensional manifold $\Sigma$. Let $\mathcal{E}(\Sigma)$ be the Willmore-type energy
$$\mathcal{E}(\Sigma):=\int_{\Sigma}\left(|\pi_{\vec{n}} \nabla\vec{H}|^2-|\vec{H}\cdot\vec{h}|^2+ 7|\vec{H}|^4\right) \, d\textnormal{vol}_g.$$ 
The manifold $\Sigma$ is Willmore-type for $\mathcal{E}(\Sigma)$ if it is a critical point of $\mathcal{E}(\Sigma).$
\end{df}
\noindent
An important characterisation of Willmore-type immersions can be deduced from the Willmore-Euler-Langrange equation associated with $\mathcal{E}(\Sigma)$ which we now state as the main result of this section. 
\begin{theo} \label{WE}
Let  $\vec{\Phi}:\Sigma\rightarrow\mathbb{R}^m $ be an immersion of a 4-dimensional manifold $\Sigma$.  
\noindent
  $\vec{\Phi}$ is Willmore-type for $\mathcal{E}(\Sigma)$ if and only if $\vec{\Phi}$ satisfies the following Willmore-Euler-Lagrange equation  
\begin{align}
\vec{\mathcal{W}}=\vec{0}            \label{abc}
\end{align}
where
\begin{align}
\vec{\mathcal{W}}&:=-\frac{1}{2}\Delta^2_{\perp}\vec{H}-\frac{1}{2}(\vec{h}_{ik}\cdot\Delta_{\perp}\vec{H})\vec{h}^{ik}-4|\pi_{\vec{n}}\nabla\vec{H}|^2\vec{H} +2\pi_{\vec{n}}\nabla_j((\vec{h}^j_i\cdot\nabla^i\vec{H})\vec{H})
\nonumber
\\& \quad
-2\pi_{\vec{n}}\nabla_j((\vec{H}\cdot\vec{h}^j_i)\pi_{\vec{n}}\nabla^i\vec{H})
+2(\pi_{\vec{n}}\nabla_i\vec{H}\cdot\pi_{\vec{n}}\nabla_j\vec{H})\vec{h}^{ij}
\nonumber
                                \\&\quad-\frac{1}{2}\Delta_{\perp}((\vec{H}\cdot\vec{h}^{ij})\vec{h}_{ij})-2\nabla_i\nabla_k((\vec{H}\cdot\vec{h}^{ik})\vec{H}) -28|\vec{H}|^4\vec{H}                                                 \nonumber
\\&\quad      -\frac{1}{2}(\vec{H}\cdot\vec{h}^{ij})(\vec{h}_{ij}\cdot \vec{h}_{pq})\vec{h}^{pq}     -4(\vec{H}\cdot\vec{h}_{ij})(\vec{H}\cdot\vec{h}^i_k)\vec{h}^{jk}  +4|\vec{H}\cdot\vec{h}|^2\vec{H}  \nonumber
\\&\quad   +7 \Delta_\perp(|\vec{H}|^2\vec{H})  +7|\vec{H}|^2(\vec{H}\cdot\vec{h}_{ij})\vec{h}^{ij} \nonumber
\end{align}
where $\Delta_\perp$ is the negative covariant Laplacian for the Levi-Civita connection in the normal bundle.

\end{theo}

\begin{rmk}
As already pointed out by Zhang \cite{zhang} and Robin-Graham-Reichert \cite{robingraham}, examples of critical points of $\mathcal{E}(\Sigma)$ include  minimal submanifolds, totally geodesic submanifolds $\Sigma$ of $\mathbb{R}^m$, round spheres $\mathbb{S}^4$ of $\mathbb{R}^m$ and products of spheres of radius $r$ such as:  $\mathbb{S}^2(r_1)\times \mathbb{S}^2(r_2)$, $\mathbb{S}^1(r_1)\times \mathbb{S}^3(r_2)$, $\mathbb{S}^1(r_1)\times \mathbb{S}^1(r_2)\times \mathbb{S}^2(r_3)$ and $\mathbb{S}^1(r_1)\times \mathbb{S}^1(r_2)\times \mathbb{S}^1(r_3)\times \mathbb{S}^1(r_4)$  where $\sum r_i^2=1.$
 Indeed, these are all solutions of  the sixth order non-linear elliptic partial differential equation \eqref{abc}.
\end{rmk}

The proof of Theorem \ref{WE} relies on the following crucial lemma.
\begin{lem}
Let $\vec{\Phi}:\Sigma\rightarrow \mathbb{R}^m$ be a smooth immersion of  a 4-dimensional manifold $\Sigma$ and $\vec{\Phi}_s$  a variation of the form $\vec{\Phi}_s=\vec{\Phi}+s(A^j\nabla_j\vec{\Phi}+\vec{B})$, where $\vec{B}$ is a normal vector. Denote by $g$ and $\vec{h}$ the first and second fundamental forms respectively. Then it holds that
\begin{align}
&\delta g_{ij}= \nabla_iA_j +\nabla_j A_i -2\vec{B}\cdot\vec{h}_{ij}  \label{metric}
\\
&\delta g^{ij}=-\nabla^i A^j-\nabla^jA^i+2\vec{B}\cdot\vec{h}^{ij}  \label{a}
\\  
&\delta|g|^{1/2}= |g|^{1/2}(\nabla_jA^j-4\vec{B}\cdot\vec{H})  \label{metric det}
\\
&\pi_{\vec{n}}\delta \vec{h}_{ij}= A^k\pi_{\vec{n}} \nabla_k\vec{h}_{ij} + \nabla_iA^k\vec{h}_{ik} +\pi_{\vec{n}} \nabla_i\nabla_j \vec{B}  \label{lklkl}
\\
& \pi_{\vec{n}}\delta\vec{H}=\frac{1}{4}(4A^s\pi_{\vec{n}}\nabla_s\vec{H}+(\vec{B}\cdot\vec{h}^{ij})\vec{h}_{ij}+\Delta_{\perp}\vec{B})  \label{c}
\end{align}
\end{lem}
\begin{proof}
 Consider a variation of the type $$\delta\vec{\Phi}=\lim_{t\rightarrow 0}\frac{\vec{\Phi}_t-\vec{\Phi}}{t}:=A^s\nabla_s\vec{\Phi}+\vec{B}$$
where $\vec{B}$ is a normal vector.\\
The metric $g$ varies as follows.
\begin{align}
\delta g_{ij}&= \nabla _i\delta\vec{\Phi}\cdot \nabla_j\vec{\Phi} + \nabla_i\vec{\Phi} \cdot\nabla_j\delta\vec{\Phi}  \nonumber
\\&= \nabla_i (A^s\nabla_s\vec{\Phi} +\vec{B}) \cdot\nabla_j\vec{\Phi} +\nabla_j(A^s\nabla_s\vec{\Phi}+ \vec{B})\cdot \nabla_i\vec{B}  \nonumber
\\&= \nabla_i A_j + \nabla_j A_i +\nabla_i \vec{B}\cdot\nabla_j\vec{\Phi} + \nabla_j\vec{B}\cdot\nabla_i\vec{\Phi}  \nonumber
\\&=\nabla_iA_j +\nabla_j A_i -2\vec{B}\cdot\vec{h}_{ij}  .\nonumber
\end{align}

\noindent
The inverse metric $g^{-1}$ varies as follows. Observe that
\begin{align}
\delta g^{ij} &=\delta (g^{ik} g^{jl} g_{kl})  \nonumber
\\&= \delta g^{ik} \delta^j_k + \delta^i_l \delta g^{jl} +g^{ik} g^{jl}  \delta g_{kl}   \nonumber
\\&= 2 \delta g^{ij} +g^{ik} g^{jl} \delta g_{kl}   \nonumber
\end{align}
so that
\begin{align}
\delta g^{ij}&=- g^{ik} g^{jl}\delta g_{kl}  \nonumber
\\&= -\nabla^iA^j-\nabla^jA^i+2\vec{B}\cdot\vec{h}^{ij}.  \nonumber
\end{align}

\noindent
The variation of the square root of the determinant of the metric is found via the well-known Jacobi formula.
\begin{align}
\delta |g|^{1/2} &=\frac{1}{2} |g|^{-1/2} \delta |g|  \nonumber
\\&= \frac{1}{2}|g|^{-1/2} |g|\textnormal{Tr}(g^{-1} \delta g)  \nonumber
\\&= \frac{1}{2} |g|^{1/2} g^{ij} (\nabla_iA_j +\nabla_j A_i -2\vec{B}\cdot\vec{h}_{ij} )  \nonumber
\\&= |g|^{1/2}(\nabla_jA^j-4\vec{B}\cdot\vec{H}) .  \nonumber
\end{align}

\noindent
Next, we obtain the normal projection of the variation of the second fundamental form. 
Note that
$$\delta\nabla_j\vec{\Phi}=\nabla_j\delta\vec{\Phi}=\nabla_j(A^s\nabla_s\vec{\Phi}+\vec{B})=A^s\vec{h}_{sj}+\nabla_jA^s\nabla_s\vec{\Phi}+\nabla_j\vec{B}.$$
Using the well-known Codazzi-Mainardi equations
\begin{align}
\pi_{\vec{n}}\nabla_i\vec{h}_{jk}=\pi_{\vec{n}}\nabla_j\vec{h}_{ik}=\pi_{\vec{n}}\nabla_k\vec{h}_{ij}  \label{codazzi}
\end{align}

\noindent
we have
\begin{align}\pi_{\vec{n}}\delta\vec{h}_{ij}&=\pi_{\vec{n}}\nabla_{i}\delta\nabla_{j}\vec{\Phi}=
\pi_{\vec{n}}\nabla_i\left(\nabla_jA^s\nabla_s\vec{\Phi}+A^s\vec{h}_{js}+\nabla_j\vec{B}  \right)  \nonumber
\\&=A^s\pi_{\vec{n}}\nabla_{i}\vec{h}_{sj}+   \nabla_{i}A^s\vec{h}_{sj}+\nabla_{j}A^s\vec{h}_{is} +  \pi_{\vec{n}}\nabla_i\nabla_j\vec{B}     \nonumber \\&=A^s\pi_{\vec{n}}\nabla_{s}\vec{h}_{ij}+\nabla_{i}A^s\vec{h}_{sj}+\nabla_{j}A^s\vec{h}_{is}
+\pi_{\vec{n}}\nabla_i\nabla_j\vec{B}. \nonumber
\end{align}

\noindent
Finally,  the normal projection of the variation of the mean curvature vector is obtained via the Codazzi-Mainardi equation \eqref{codazzi}. Note that
\begin{align}
\pi_{\vec{n}}\nabla^j\delta\nabla
_j\vec{\Phi}&=\vec{h}_{sj}\nabla^jA^s+A^s\pi_{\vec{n}}\nabla^j\vec{h}_{sj}+\vec{h}^j_s\nabla_jA^s+\pi_{\vec{n}}\Delta_g\vec{B}
\nonumber\\&= 2\vec{h}_{sj}\nabla^jA^s+4A^s\pi_{\vec{n}}\nabla_s\vec{H}+\pi_{\vec{n}}\Delta_g\vec{B}.\label{b}
\end{align}

\noindent
By using \eqref{a} and \eqref{b}, we have
\begin{align}
\pi_{\vec{n}}\delta\vec{H}&= \frac{1}{4}\pi_{\vec{n}}\delta\nabla^j\nabla_j\vec{\Phi}=\frac{1}{4}\pi_{\vec{n}}\delta(g^{ij}\nabla_i\nabla_j\vec{\Phi})\nonumber
\\ &= \frac{1}{4}(\vec{h}_{ij}\delta g^{ij}+\pi_{\vec{n}}\nabla^j\delta\nabla_j\vec{\Phi})\nonumber
\\ &= \frac{1}{4} \vec{h}_{ij}(-\nabla^iA^j-\nabla^jA^i+2\vec{B}\cdot\vec{h}^{ij})  \nonumber 
\\&\quad+\frac{1}{4}(2\vec{h}_{sj}\nabla^jA^s+4A^s\pi_{\vec{n}}\nabla_s\vec{\vec{H}}+\pi_{\vec{n}}\Delta_g\vec{B})\nonumber
\\&= \frac{1}{4} (4A^s\pi_{\vec{n}}\nabla_s\vec{H}+2(\vec{B}\cdot\vec{h}^{ij})\vec{h}_{ij}+\pi_{\vec{n}}\Delta_g\vec{B})\nonumber
\\ &= \frac{1}{4}(4A^s\pi_{\vec{n}}\nabla_s\vec{H}+(\vec{B}\cdot\vec{h}^{ij})\vec{h}_{ij}+\Delta_{\perp}\vec{B}).  \nonumber
\end{align}

\end{proof}

\vskip 5mm
\begin{proof}[Proof of Theorem \ref{WE}]

Our approach is to compute the variation of each term of $\mathcal{E}(\Sigma)$ separately on a small patch $\Sigma_0$ of the manifold $\Sigma$. First, we consider the variation of $|\pi_{\vec{n}}\nabla\vec{H}|^2$.
 Given a tensor field $\vec{T}_j$, the variation of $|\vec{T}|^2$ is found from the variation of $\vec{T}_j$ as follows:
\begin{align}
\delta|\vec{T}|^2 &= \delta(g^{ij}\vec{T}_i\cdot\vec{T}_j)= (\vec{T}_i\cdot\vec{T}_j)\delta g^{ij}+2\vec{T}^j\cdot\delta\vec{T}_j \nonumber
\\ &= 2\vec{T}_i\cdot \vec{T}_j(-\nabla^iA^j+\vec{B}\cdot \vec{h}^{ij})+ 2\vec{T}^{j}\cdot\delta \vec{T}_j\nonumber
\end{align}
where we have used \eqref{a}. Using  \eqref{metric det}, we have 
\begin{align}
|g|^{-1/2}\delta(|\vec{T}|^2|g|^{1/2})&= 2\vec{T}_i\cdot\vec{T}_j(-\nabla^iA^j+\vec{B}\cdot\vec{h}^{ij})+2\vec{T}
^j\cdot\delta\vec{T}_j+|\vec{T}|^2(\nabla_jA^j-4\vec{B}\cdot\vec{H})\label{d}
\end{align}

\noindent
We will now vary $|\vec{T}|^2|g|^{1/2}$ for the choice $\vec{T}_j=\pi_{\vec{n}}\nabla_j\vec{H}$. We will need \eqref{c}. First note that
\begin{align}
&\quad\,\,\pi_{\vec{n}}\nabla^j\vec{H}\cdot \delta\pi_{\vec{n}}\nabla_j\vec{H}=\pi_{\vec{n}}\nabla^j\vec{H}\cdot\delta\nabla_j\vec{H}-\pi_{\vec{n}}\nabla^j\vec{H}\cdot \delta\pi_{T}\nabla_j\vec{H}  \nonumber \\ &= \pi_{\vec{n}}\nabla^j\vec{H}\cdot\nabla_j\delta\vec{H}-\pi_{\vec{n}}\nabla^j\vec{H}\cdot \delta\pi_{T}\nabla_j\vec{H}  \nonumber
\\ &= \pi_{\vec{n}}\nabla^j\vec{H}\cdot\nabla_j\pi_{\vec{n}}\delta\vec{H}+  \pi_{\vec{n}}\nabla^j\vec{H}\cdot\nabla_j\pi_T\delta\vec{H}-    \pi_{\vec{n}}\nabla^j\vec{H}\cdot\delta\pi_T\nabla_j\vec{H}   \nonumber
\\ &= \pi_{\vec{n}}\nabla^j\vec{H}\cdot\nabla_j\pi_{\vec{n}}\delta\vec{H}+ \pi_{\vec{n}}\nabla^j\vec{H}\cdot  \left[ \pi_{\vec{n}}\nabla_j  [ (\nabla_i\vec{\Phi}\cdot \delta\vec{H})\nabla^i\vec{\Phi} ]-\delta   [ (\nabla^i\vec{\Phi}\cdot \nabla_j\vec{H})\nabla_i\vec{\Phi}   ]               \right]           \nonumber
\\ &= \pi_{\vec{n}}\nabla^j\vec{H}\cdot\nabla_j\pi_{\vec{n}}\delta\vec{H}+ \nabla^j\vec{H}\cdot \left[ -\pi_{\vec{n}}\nabla_j[(\vec{H}\cdot\nabla_i\delta\vec{\Phi})\nabla^i\vec{\Phi}]  +\pi_{\vec{n}}\delta [(\vec{H}\cdot \vec{h}^i_j)\nabla_i\vec{\Phi}]      \right] \nonumber
\\ &= \pi_{\vec{n}}\nabla^j\vec{H}\cdot\nabla_j\pi_{\vec{n}}\delta\vec{H}-\nabla^j\vec{H}\cdot \left [ (\vec{H}\cdot \nabla_i\delta\vec{\Phi})\vec{h}^i_j-(\vec{H}\cdot\vec{h}_j^i)\pi_{\vec{n}}\nabla_i\delta\vec{\Phi}  \right]  \nonumber
\\ &= \pi_{\vec{n}}\nabla^j\vec{H}\cdot\nabla_j\pi_{\vec{n}}\delta\vec{H}-\nabla^j\vec{H} \cdot\left[ [\vec{H}\cdot \nabla_i(A^s\nabla_s\vec{\Phi}+\vec{B})]\vec{h}^i_j-(\vec{H}\cdot \vec{h}^i_j)\pi_{\vec{n}}\nabla_i(A^s\nabla_s\vec{\Phi}+\vec{B})   \right]  \nonumber
\\ &= \pi_{\vec{n}}\nabla^j\vec{H}\cdot\nabla_j\pi_{\vec{n}}\delta\vec{H}- \nabla^j \vec{H}\cdot \left[ (A^s\vec{H}\cdot\vec{h}_{is}+\vec{H}\cdot \nabla_i\vec{B})\vec{h}^i_j-(\vec{H}\cdot\vec{h}^i_j)(A^s\vec{h}_{is}+\pi_{\vec{n}}\nabla_i\vec{B})    \right]  \nonumber
\\&= \pi_{\vec{n}}\nabla^j\vec{H}\cdot\nabla_j\pi_{\vec{n}}\delta\vec{H}- A^i \left[ (\vec{H}\cdot\vec{h}_{is})(\vec{h}^s_j\cdot\nabla^j\vec{H})-(\vec{H}\cdot \vec{h}^s_j)(\vec{h}_{is}\cdot\nabla^j\vec{H})   \right]\nonumber
\\& \quad\quad\quad -(\vec{H}\cdot\nabla_i\vec{B})(\vec{h}^i_j\cdot\nabla^j\vec{H})+ (\vec{H}\cdot \vec{h}^i_j)(\nabla^j\vec{H}\cdot\pi_{\vec{n}}\nabla_i\vec{B}).\label{e}
\end{align}

\noindent
By using Simon's identity
$$\Delta\vec{h}_{ij}=\nabla_i\nabla_j\vec{H}+\vec{H}\cdot\vec{h}_{ik}g^{kl}\vec{h}_{lj}-|\vec{h}|^2\vec{h}_{ij}$$
we see that
$$\nabla_i\nabla_j\vec{H}=\nabla_j\nabla_i\vec{H}.$$
Hence, we find that
\begin{align}
&\quad\pi_{\vec{n}}\nabla^j\vec{H}\cdot\nabla_j\pi_{\vec{n}}\nabla_i\vec{H}=\pi_{\vec{n}} \nabla^j\vec{H}\cdot \nabla_j\nabla_i\vec{H}-\pi_{\vec{n}}\nabla^j\vec{H}\cdot\nabla_j\pi_T\nabla_i\vec{H} \nonumber
\\&= \pi_{\vec{n}}\nabla^j\vec{H}\cdot\nabla_i\nabla_j\vec{H}+\pi_{\vec{n}}\nabla^j\vec{H}\cdot\nabla_j((\vec{H}\cdot \vec{h}_{is})\nabla^s\vec{\Phi})\nonumber
\\ &= \pi_{\vec{n}}\nabla^j\vec{H}\cdot\nabla_i\pi_{\vec{n}}\nabla_j\vec{H}+\pi_{\vec{n}}\nabla^j\vec{H}\cdot\nabla_i\pi_T\nabla_j\vec{H}+(\vec{H}\cdot\vec{h}_{is})(\vec{h}^s_j\cdot\nabla^j\vec{H})\nonumber
\\ &= \frac{1}{2}\nabla_i|\pi_{\vec{n}}\nabla \vec{H}|^2-\pi_{\vec{n}}\nabla^j\vec{H}\cdot \nabla_i((\vec{H}\cdot\vec{h}_{sj})\nabla^s\vec{\Phi})+(\vec{H}\cdot\vec{h}_{is})(\vec{h}^s_j\cdot\nabla^j\vec{H})\nonumber
\\ &= \frac{1}{2}\nabla_i|\pi_{\vec{n}}\nabla\vec{H}|^2-(\vec{H}\cdot\vec{h}_{sj})(\vec{h}^s_i\cdot\nabla^j\vec{H})+ (\vec{H}\cdot\vec{h}_{is})(\vec{h}^s_j\cdot\nabla^j\vec{H}) \nonumber
\\&= \frac{1}{2}\nabla_i|\pi_{\vec{n}}\nabla\vec{H}|^2-(\vec{H}\cdot\vec{h}^s_j)(\vec{h}_{is}\cdot\nabla^j\vec{H})+(\vec{H}\cdot\vec{h}_{is})(\vec{h}^s_j\cdot\nabla^j\vec{H})\label{ooio}
\end{align}

\noindent
Using \eqref{ooio} in \eqref{e} yields
\begin{align}
\pi_{\vec{n}}\nabla^j\vec{H}\cdot\delta\pi_{\vec{n}}\nabla_j\vec{H}&= \pi_{\vec{n}}\nabla^j\vec{H}\cdot\left[\nabla_j\pi_{\vec{n}}\delta\vec{H}-A^i\nabla_j\pi_{\vec{n}}\nabla_i\vec{H}     \right] +\frac{A^i}{2}\nabla_i|\pi_{\vec{n}}\nabla\vec{H}|^2   \nonumber
\\&\quad-  (\vec{H}\cdot\nabla_i\vec{B})(\vec{h}^i_j\cdot\nabla^j\vec{H})+ (\vec{H}\cdot\vec{h}^i_j)(\nabla^j\vec{H}\cdot\pi_{\vec{n}}\nabla_i\vec{B}).
\end{align}

\noindent
Calling upon \eqref{c} now gives
\begin{align}
&\pi_{\vec{n}}\nabla^j\vec{H}\cdot\delta\pi_{\vec{n}}\nabla_j\vec{H}\nonumber
\\&= \pi_{\vec{n}}\nabla^j\vec{H}\cdot \left[ \frac{1}{4}\nabla_j(4A^s\pi_{\vec{n}}\nabla_s\vec{H}+(\vec{B}\cdot\vec{h}^{ik})\vec{h}_{ik}+\Delta_{\perp}\vec{B})-A^i\nabla_j\pi_{\vec{n}}\nabla_i\vec{H}  \right] \nonumber
\\&\quad -(\vec{H}\cdot\nabla_i\vec{B})(\vec{h}^i_j\cdot\nabla^j\vec{H})+(\vec{H}\cdot\vec{h}^i_j)(\nabla^j\vec{H}\cdot\pi_{\vec{n}}\nabla_i\vec{B})+\frac{A^i}{2}\nabla_i|\pi_{\vec{n}}\nabla\vec{H}|^2 \nonumber
\\&= (\pi_{\vec{n}}\nabla^j\vec{H}\cdot\pi_{\vec{n}}\nabla_i\vec{H})\nabla_jA^i
+\frac{1}{4}\pi_{\vec{n}}\nabla^j\vec{H}\cdot\nabla_j((\vec{B}\cdot\vec{h}^{ik})\vec{h}_{ik})\nonumber
\\&\quad +\frac{1}{4}\pi_{\vec{n}}\nabla^j\vec{H}\cdot\nabla_j\Delta_{\perp}\vec{B} -(\vec{H}\cdot\nabla_i\vec{B})(\vec{h}^i_j\cdot\nabla^j\vec{H})+(\vec{H}\cdot\vec{h}^i_j)(\nabla^j\vec{H}\cdot\pi_{\vec{n}}\nabla_i\vec{B}) \nonumber
\\ &\quad +\frac{A^i}{2}\nabla_i|\pi_{\vec{n}}\nabla\vec{H}|^2         \nonumber
\\&= (\vec{T}_i\cdot\vec{T}_j)\nabla^jA^i+\frac{A^i}{2}\nabla_i|\vec{T}|^2+\frac{1}{4}\vec{T}^j\cdot\nabla_j((\vec{B}\cdot\vec{h}^{ik})\vec{h}_{ik})+\frac{1}{4}\vec{T}^j\cdot\nabla_j\Delta_{\perp}\vec{B}\nonumber
\\&\quad -(\vec{H}\cdot\nabla_i\vec{B})(\vec{h}^i_j\cdot\vec{T}^j)+(\vec{H}\cdot\vec{h}^i_j)(\vec{T}^j\cdot\nabla_i\vec{B}).\label{jjak}
\end{align}
Substituting \eqref{jjak} into \eqref{d} gives
\begin{align}
&|g|^{-1/2}\delta(|\vec{T}|^2|g|^{1/2}) \nonumber
\\&= 2(\vec{T}_i\cdot \vec{T}_j)(\vec{B}\cdot\vec{h}^{ij})+\frac{1}{2}\vec{T}^j\cdot\nabla_j((\vec{B}\cdot \vec{h}^{ik})\vec{h}_{ik})+\frac{1}{2}\vec{T}^j\cdot\nabla_j\Delta_{\perp}\vec{B}  \nonumber
\\&\quad -2(\vec{H}\cdot \nabla_j\vec{B})(\vec{h}^j_i\cdot\vec{T}^i)+ 2(\vec{H}\cdot\vec{h}^j_i)(\vec{T}^i\cdot\nabla_j\vec{B})- 4|\vec{T}|^2\vec{H}\cdot\vec{B}+\nabla^j(|\vec{T}|^2A_j)  \nonumber
\\& = \vec{B}\cdot\left[ 2(\vec{T}_i\cdot\vec{T}_j)\vec{h}^{ij}-4|\vec{T}|^2\vec{H}-\frac{1}{2}(\vec{h}_{ik}\cdot \nabla_j\vec{T}^j)\vec{h}^{ik}+2\nabla_j((\vec{T}^i\cdot\vec{h}^j_i)\vec{H}) \right.  \nonumber
\\&\quad\quad\quad\left.-2\nabla_j((\vec{H}\cdot\vec{h}^j_i)\vec{T}^i) -\frac{1}{2}\Delta_{\perp}\pi_{\vec{n}}\nabla_i\vec{T}^i     \right] \nonumber
\\ &\quad+\nabla^j \left[ |\vec{T}|^2A_j+\frac{1}{2}(\vec{B}\cdot\vec{h}^{ik})(\vec{T}_j\cdot\vec{h}_{ik})-2(\vec{B}\cdot\vec{H})(\vec{T}^i\cdot\vec{h}_{ij}) \right.  \nonumber
\\&\quad\quad\quad\left.  +2(\vec{B}\cdot\vec{T}^i)(\vec{H}\cdot\vec{h}_{ij})+\frac{1}{2}\vec{T}_j\cdot\Delta_{\perp}\vec{B}-\frac{1}{2}\nabla_j\vec{B}\cdot\pi_{\vec{n}}\nabla_i\vec{T}^i  
+\frac{1}{2}\vec{B}\cdot\nabla_j\pi_{\vec{n}}\nabla_i\vec{T}^i         \right]. \label{kakl}
\end{align}
Equation \eqref{kakl} can be simplified by noting that
$$\pi_{\vec{n}}\nabla_i\vec{T}^i=\Delta_{\perp}\vec{H}.$$
Then, we have 
\begin{align}|g|^{-1/2}\delta(|\vec{T}|^2|g|^{1/2})=\vec{B}\cdot\mathcal{\vec{W}}_1+\nabla^j(V_1)_j\label{wil}\end{align}
where
\begin{align}
\mathcal{\vec{W}}_1&= 2(\pi_{\vec{n}}\nabla_i\vec{H}\cdot\pi_{\vec{n}}\nabla_j\vec{H})\vec{h}^{ij}-4|\pi_{\vec{n}}\nabla\vec{H}|^2\vec{H}-\frac{1}{2}(\vec{h}_{ik}\cdot\Delta_{\perp}\vec{H})\vec{h}^{ik}\nonumber
\\& +2\pi_{\vec{n}}\nabla_j((\vec{h}^j_i\cdot\nabla^i\vec{H})\vec{H})-2\pi_{\vec{n}}\nabla_j((\vec{H}\cdot\vec{h}^j_i)\pi_{\vec{n}}\nabla^i\vec{H})-\frac{1}{2}\Delta^2_{\perp}\vec{H} \nonumber
\end{align}
and 
\begin{align}
&(V_1)_j= |\pi_{\vec{n}}\nabla\vec{H}|^2A_j+\frac{1}{2}(\vec{B}\cdot\vec{h}^{ik})(\vec{h}_{ik}\cdot\nabla_j\vec{H})-2(\vec{B}\cdot\vec{H})(\vec{h}_{ij}\cdot\nabla^i\vec{H}) \nonumber
\\&+ 2(\vec{B}\cdot\nabla^i\vec{H})(\vec{H}\cdot\vec{h}_{ij})+\frac{1}{2}\pi_{\vec{n}}\nabla_j\vec{H}\cdot \Delta_{\perp}\vec{B}-\frac{1}{2}\nabla_j\vec{B}\cdot\Delta_{\perp}\vec{H}+\frac{1}{2}\vec{B}\cdot\nabla_j\Delta_{\perp}\vec{H}.  \nonumber
\end{align}

\noindent
Next is to compute the variation of $|\vec{H}\cdot\vec{h}|^2$. We consider an energy involving $|T|^2$ for a symmetric 2-tensor $T_{ij}$. We have
\begin{eqnarray*}
\delta|T|^2&=&\delta(g^{ik}g^{jl}T_{kl}T_{ij})=2\left(T^{ij}\delta T_{ij}+ T_{ij} T_{k}^{i}\delta g^{jk}  \right)\\& =&2 T^{ij}\delta T_{ij}+ 2T_{ij} T^i_k\left(-\nabla^jA^k-\nabla^kA^j+2\vec{B}\cdot \vec{h}^{jk}\right).
\end{eqnarray*}

\noindent
By using  \eqref{lklkl} and \eqref{c}, the variation of $T_{ij}$ for the choice  $T_{ij}:=\vec{H}\cdot \vec{h}_{ij}$ is 
\begin{eqnarray}
\delta{T}_{ij}&=&\delta(\vec{H}\cdot\vec{h}_{ij}) \nonumber \\&=&\nabla_{i}A^s\vec{H}\cdot \vec{h}_{sj}+\nabla_j A^s \vec{H}\cdot\vec{h}_{is}+ A^s\vec{H}\cdot\nabla_s\vec{h}_{ij}+ A^s\vec{h}_{ij}\cdot\nabla_s\vec{H}\nonumber\\&&+\vec{H}\cdot\nabla_i\nabla_j\vec{B}+\frac{1}{4}(\vec{B}\cdot\vec{h}^{pq})(\vec{h}_{ij}\cdot \vec{h}_{pq})+\frac{1}{4}\vec{h}_{ij}\cdot \Delta_{\perp}\vec{B}\nonumber\\&=& T_{sj}\nabla_i A^s+ T_{is}\nabla_j A^s+A^s\nabla_sT_{ij}\nonumber\\&&+\vec{H}\cdot \nabla_i\nabla_j\vec{B}+\frac{1}{4}(\vec{B}\cdot\vec{h}^{pq})(\vec{h}_{ij}\cdot \vec{h}_{pq})+\frac{1}{4}\vec{h}_{ij}\cdot \Delta_{\perp}\vec{B}.  \label{ikio}
\end{eqnarray}

\noindent
Using \eqref{metric det} and \eqref{ikio}, we find 
\begin{eqnarray*}
|g|^{-1/2}\delta\left(|T|^2|g|^{1/2} \right)&=& 2\left( T^{ij}\delta T_{ij}+2T_{ij}T^i_k \vec{B}\cdot \vec{h}^{jk}-T_{ij}T^i_k \nabla^jA^k-T_{ij} T^i_k\nabla^kA^j\right)   \\&& + |T|^2\nabla_jA^j-4|T|^2\vec{B}\cdot \vec{H}\\&=& \nabla_s\left(A^s |T|^2\right)+2 T^{ij}\vec{H}\cdot \nabla_i\nabla_j\vec{B}+\frac{1}{2} T^{ij}(\vec{B}\cdot \vec{h}^{pq})(\vec{h}_{ij}\cdot\vec{h}_{pq})\\ && +\frac{1}{2}T^{ij}\vec{h}_{ij}\cdot\Delta_{\perp}\vec{B}+4T_{ij}(\vec{H}\cdot\vec{h}_{k}^{i})(\vec{B}\cdot\vec{h}^{jk})-4|T|^2\vec{B}\cdot\vec{H}\\&=& \vec{B}\cdot \vec{\mathcal{W}}_2+\nabla_k (V_2)^k
\end{eqnarray*}
where
$$\vec{\mathcal{W}}_2:=2\nabla_i\nabla_k(T^{ik}\vec{H})+\frac{1}{2}\Delta_{\perp}(T^{ij}\vec{h}_{ij})+\frac{1}{2}T^{ij}(\vec{h}_{ij}\cdot \vec{h}_{pq})\vec{h}^{pq}+4T_{ij}T^i_k\vec{h}^{jk}-4|T|^2\vec{H}$$
and 
$$(V_2)^k:=2T^{ik}\vec{H}\cdot \nabla_i \vec{B}-2\vec{B}\cdot \nabla_i(T^{ik}\vec{H})+\frac{1}{2}T^{ij}\vec{h}_{ij}\cdot \nabla^k\vec{B}-\frac{1}{2}\vec{B}\cdot\nabla^k(T^{ij}\vec{h}_{ij})+A^k|T|^2.$$

\noindent
Finally, we consider the variation of  $|\vec{H}|^4$. Using \eqref{c} we have
\begin{eqnarray*}
\delta|\vec{H}|^4&=&4|\vec{H}|^2\vec{H}\cdot\delta \vec{H}
\\
&=&|\vec{H}|^2\vec{H}\cdot \left( 4A^s\nabla_s\vec{H}+(\vec{B}\cdot\vec{h}^{ij})\vec{h}_{ij}+\Delta_\perp\vec{B} \right).
\end{eqnarray*}

\noindent
Hence, using \eqref{metric det}, we find  
\begin{eqnarray*}
|g|^{-1/2}\delta\left( |\vec{H}|^4 |g|^{1/2} \right)&=& 4|\vec{H}|^{2}A^s\vec{H}\cdot \nabla_s\vec{H}+|\vec{H}|^2(\vec{B}\cdot\vec{h}^{ij})(\vec{H}\cdot\vec{h}_{ij}) \\&&\quad+|\vec{H}|^2\vec{H}\cdot\Delta_\perp\vec{B}+|\vec{H}|^4\nabla_jA^j-4|\vec{H}|^4\vec{B}\cdot\vec{H}
\\&=& \vec{B}\cdot\mathcal{\vec{W}}_3+ \nabla_j (V_3)^j.
\end{eqnarray*}
where 
\begin{gather*}
\vec{\mathcal{W}}_3= \Delta_\perp(|\vec{H}|^2\vec{H})  +|\vec{H}|^2(\vec{H}\cdot\vec{h}_{ij})\vec{h}^{ij}-4|\vec{H}|^4\vec{H}  
\end{gather*}
and
\begin{gather*}
(V_3)^j=|\vec{H}|^4A^j+ |\vec{H}|^2\vec{H}\cdot\nabla^j\vec{B}-\vec{B}\cdot\nabla^j(|\vec{H}|^2 \vec{H})   .
\end{gather*}

\noindent
By setting
 $$\vec{\mathcal{W}}:=\vec{\mathcal{W}_1}-\vec{\mathcal{W}_2}+7\vec{\mathcal{W}_3} \quad\mbox{and}\quad {V}^j:=({V}_1)^j-({V}_2)^j+7({V}_3)^j$$ we see that
\begin{align}
\vec{\mathcal{W}}&=2(\pi_{\vec{n}}\nabla_i\vec{H}\cdot\pi_{\vec{n}}\nabla_j\vec{H})\vec{h}^{ij}-4|\pi_{\vec{n}}\nabla\vec{H}|^2\vec{H}-\frac{1}{2}(\vec{h}_{ik}\cdot\Delta_{\perp}\vec{H})\vec{h}^{ik}+2\pi_{\vec{n}}\nabla_j((\vec{h}^j_i\cdot\nabla^i\vec{H})\vec{H})\nonumber
\\& \quad-2\pi_{\vec{n}}\nabla_j((\vec{H}\cdot\vec{h}^j_i)\pi_{\vec{n}}\nabla^i\vec{H})-\frac{1}{2}\Delta^2_{\perp}\vec{H}  -2\nabla_i\nabla_k((\vec{H}\cdot\vec{h}^{ik})\vec{H})-\frac{1}{2}\Delta_{\perp}((\vec{H}\cdot\vec{h}^{ij})\vec{h}_{ij}) \nonumber
                                \\&\quad -\frac{1}{2}(\vec{H}\cdot\vec{h}^{ij})(\vec{h}_{ij}\cdot \vec{h}_{pq})\vec{h}^{pq}-4(\vec{H}\cdot\vec{h}_{ij})(\vec{H}\cdot\vec{h}^i_k)\vec{h}^{jk}+4|\vec{H}\cdot\vec{h}|^2\vec{H}    \nonumber
\\&\quad               +7 \Delta_\perp(|\vec{H}|^2\vec{H})  +7|\vec{H}|^2(\vec{H}\cdot\vec{h}_{ij})\vec{h}^{ij}-28|\vec{H}|^4\vec{H}                                              \nonumber
\end{align}
\noindent
and

\begin{align}
V^j&= \left(|\pi_{\vec{n}}\nabla\vec{H}|^2 -|\vec{H}\cdot\vec{h}|^2+7|\vec{H}|^4\right)A^j      +\frac{1}{2}(\vec{B}\cdot\vec{h}^{ik})(\vec{h}_{ik}\cdot\nabla^j\vec{H})  -2(\vec{B}\cdot\vec{H})(\vec{h}^j_i\cdot\nabla^i\vec{H})             \nonumber
\\&\quad+2(\vec{B}\cdot\nabla^i\vec{H}) (\vec{H}\cdot\vec{h}^j_i)+\frac{1}{2}\pi_{\vec{n}}\nabla^j\vec{H}\cdot\Delta_{\perp}\vec{B}-\frac{1}{2}\nabla^j\vec{B}\cdot\Delta_{\perp}\vec{H}+\frac{1}{2}\vec{B}\cdot\nabla^j\Delta_{\perp}\vec{H}                        \nonumber
\\& \quad-2(\vec{H}\cdot\vec{h}^{ij})(\vec{H}\cdot\nabla_i\vec{B})+2\vec{B}\cdot\nabla_i((\vec{H}\cdot\vec{h}^{ij})\vec{H})-\frac{1}{2}(\vec{H}\cdot\vec{h}^{ik})\vec{h}_{ik}\cdot\nabla^j\vec{B}       \nonumber
\\&\quad  +\frac{1}{2}\vec{B}\cdot\nabla^j((\vec{H}\cdot\vec{h}^{ik})\vec{h}_{ik})                                        +7 |\vec{H}|^2\vec{H}\cdot\nabla^j\vec{B}-7\vec{B}\cdot\nabla^j(|\vec{H}|^2 \vec{H}).            \nonumber
\end{align}

\noindent
Moreover, it follows that
\begin{align}|g|^{-1/2}\delta\left((|\pi_{\vec{n}}\nabla\vec{H}|^2-|\vec{H}\cdot\vec{h}|^2+7|\vec{H}|^4)|g|^{1/2}\right)=\vec{B}\cdot\mathcal{\vec{W}}+\nabla_jV^j\label{wil}\end{align}

\noindent
or more specifically
\begin{align}
\delta\int_{\Sigma_0}(|\pi_{\vec{n}}\nabla\vec{H}|^2-|\vec{H}\cdot\vec{h}|^2+7|\vec{H}|^4)\,\, d\textnormal{vol}_g=\int_{\Sigma_0}(\vec{B}\cdot\mathcal{\vec{W}}+\nabla_jV^j)\,\, d\textnormal{vol}_g.  
\end{align}
Since $\vec{\Phi}$ is a critical point of the energy $\mathcal{E}(\Sigma)$, it holds
\begin{align}
\int_{\Sigma_0}(\vec{B}\cdot\mathcal{\vec{W}}+\nabla_jV^j)\,\, d\textnormal{vol}_g=0 .\label{kuj1}
\end{align}
As \eqref{kuj1} holds on every small patch $\Sigma_0$ of the manifold $\Sigma$, we can write
$$\vec{B}\cdot\mathcal{\vec{W}}+\nabla_jV^j=0.$$
The Euler-Lagrange equation for $\mathcal{E}(\Sigma)$ corresponds to the normal variation of $\mathcal{E}(\Sigma)$, that is $\pi_{\vec{n}}\delta\vec{\Phi}=\vec{B}$.
This gives
$$\mathcal{\vec{W}}=\vec{0}.$$
\noindent

\end{proof}
\begin{rmk}
We point out an important mistake of \cite{robingraham} (Proposition 5.6) where Robin-Graham and Reichert wrote the Willmore-Euler-Lagrange equation for $\mathcal{E}(\Sigma)$ in an arbitrary codimension. They did not show their computation but followed a formalism which was used by Guven \cite{jguv} to derive the energy $\mathcal{E}(\Sigma)$ in codimension one.
Proposition 5.6 of \cite{robingraham}  differs from \eqref{abc} by two terms: $2\pi_{\vec{n}}\nabla_j((\vec{h}^j_i\cdot\nabla^i\vec{H})\vec{H})$ and $-2\pi_{\vec{n}}\nabla_j((\vec{H}\cdot\vec{h}^j_i)\pi_{\vec{n}}\nabla^i\vec{H})$.
\\\noindent
Note that in codimension one, these two terms cancel out and Proposition 5.6 of \cite{robingraham} would then be correct. 
\end{rmk}

\noindent
We now prove an identity which holds on any 4-dimensional manifold $\Sigma$. This identity is obtained by varying 
\begin{align}
\mathcal{E}_{H^2}:= \int_{\Sigma_0\subset \Sigma}|\vec{H}|^2 d\textnormal{vol}_g.   \label{2DD}
\end{align}
on every small patch $\Sigma_0$ of $\Sigma$ and using the translation invariance of \eqref{2DD}. Note that we are not considering the critical points of the energy \eqref{2DD}. Hence the Willmore-type operator $\vec{\mathcal{W}}_{H^2}$ for \eqref{2DD} does not satisfy 
$$\vec{\mathcal{W}}_{H^2}=\vec{0}.$$
   
\begin{lem}  \label{rivvy}
Let $\vec{\Phi}:\Sigma\rightarrow \mathbb{R}^m$ be a smooth immersion of a 4-dimensional manifold $\Sigma$. Define 
\begin{align}
\vec{A}^s:= \nabla^s\vec{H}-2(|\vec{H}|^2g^{sk}-\vec{H}\cdot\vec{h}^{sk})\nabla_k\vec{\Phi}. \label{defA}
\end{align}
Then the following identity holds
\begin{align}
\nabla_s \vec{A}^s= \Delta_\perp\vec{H} +(\vec{H}\cdot \vec{h}_{ij})\vec{h}^{ij} -8|\vec{H}|^2\vec{H}.  \nonumber
\end{align}
\end{lem}

\begin{proof}
Using \eqref{c}, we have
\begin{align}
\delta|\vec{H}|^2&=  2\vec{H}\cdot \delta\vec{H}                  \nonumber
\\&= \frac{1}{2}\vec{H}\cdot \left( 4A^s\pi_{\vec{n}}\nabla_s\vec{H}+(\vec{B}\cdot\vec{h}^{ij})\vec{h}_{ij}
+\Delta_{\perp}\vec{B})  
  \right)   .       \nonumber
\end{align}

\noindent
Using \eqref{metric det},  we have
\begin{align}
&\quad\quad|g|^{-1/2} \delta\left( |\vec{H}|^2 |g|^{1/2}  \right)  \nonumber
\\&= 2A^s \vec{H}\cdot \nabla_s\vec{H} +\frac{1}{2} (\vec{B}\cdot\vec{h}^{ij})(\vec{H}\cdot\vec{h}_{ij}) +\frac{1}{2}\vec{H}\cdot\Delta_\perp\vec{B} +|\vec{H}|^2\nabla_s A^s -4(\vec{B}\cdot\vec{H}) |\vec{H}|^2      \nonumber
\\&= \nabla_s(A^s |\vec{H}|^2) +\frac{1}{2} (\vec{B}\cdot\vec{h}^{ij})(\vec{H}\cdot\vec{h}_{ij})-4(\vec{B}\cdot\vec{H}) |\vec{H}|^2   \nonumber +\frac{1}{2}\left[ \nabla_i(\vec{H}\cdot \nabla^i\vec{B})  -\nabla_i\vec{H} \cdot\pi_{\vec{n}} \nabla^i\vec{B}  \right]    \nonumber
\\&= \nabla_s(A^s |\vec{H}|^2)  +\frac{1}{2} (\vec{B}\cdot\vec{h}^{ij})(\vec{H}\cdot\vec{h}_{ij})   -4(\vec{B}\cdot\vec{H}) |\vec{H}|^2         \nonumber
\\& \quad\quad\quad+\frac{1}{2}\left[\nabla_i(\vec{H}\cdot \nabla^i\vec{B})  -\nabla^i (\pi_{\vec{n}}\nabla_i\vec{H}\cdot\vec{B})  +\vec{B}\cdot \nabla^i\pi_{\vec{n}}\nabla_i\vec{H}      \right]   \nonumber
\\&= \nabla_s\left[ A^s|\vec{H}|^2 +\frac{1}{2} \vec{H}\cdot \nabla^s\vec{B} -\frac{1}{2} \vec{B}\cdot\pi_{\vec{n}}\nabla^s\vec{H}     \right]  \nonumber
\\&\quad\quad\quad + \vec{B}\cdot   \left[\frac{1}{2} (\vec{H}\cdot\vec{h}_{ij})\vec{h}^{ij} -4|\vec{H}|^2\vec{H} +\frac{1}{2} \Delta_\perp \vec{H}          \right]  \nonumber
\\&= \nabla_s V_{H^2}^s +\vec{B}\cdot\vec{\mathcal{W}}_{H^2} \nonumber
\end{align}
where
$$V_{H^2}^s:=A^s|\vec{H}|^2 +\frac{1}{2} \vec{H}\cdot \nabla^s\vec{B} -\frac{1}{2} \vec{B}\cdot\pi_{\vec{n}}\nabla^s\vec{H} $$
and
$$\vec{\mathcal{W}}_{H^2}:=\frac{1}{2} (\vec{H}\cdot\vec{h}_{ij})\vec{h}^{ij} -4|\vec{H}|^2\vec{H} +\frac{1}{2} \Delta_\perp \vec{H}      .$$

\noindent
We have on every small patch $\Sigma_0$ the identity
\begin{align}
\delta\int_{\Sigma_0}|\vec{H}|^2 d\textnormal{vol}_g= \int_{\Sigma_0} (\nabla_s V_{H^2}^s +\vec{B}\cdot\vec{\mathcal{W}}_{H^2}) \,\,d\textnormal{vol}_g . \label{chinedui}
\end{align}

\noindent
Recall $\delta\vec{\Phi}=A^s\nabla_s\vec{\Phi}+\vec{B}$. Using invariance by translation of the energy \eqref{2DD}: $$\delta\vec{\Phi}=\vec{a}, \,\,\mbox{where}\,\, \vec{a}\in\mathbb{R}^m \,\,\mbox{is a constant vector}$$
yields
\begin{align}
A^s= \vec{a}\cdot \nabla^s\vec{\Phi}\,\quad\mbox{and}\,\quad \vec{B}=\pi_{\vec{n}}\vec{a}. \label{yies}
\end{align}

Now, using \eqref{yies} we have
\begin{align}
V_{H^2}^s&=A^s|\vec{H}|^2 +\frac{1}{2} \vec{H}\cdot \nabla^s\vec{B} -\frac{1}{2} \vec{B}\cdot\pi_{\vec{n}}\nabla^s\vec{H}  \nonumber
\\&= \vec{a}\cdot\left[\nabla^s\vec{\Phi}|\vec{H}|^2  +\frac{1}{2}\nabla^s\vec{H} -\pi_{\vec{n}}\nabla^s\vec{H} \right] \nonumber
\\&=\vec{a}\cdot\left[\nabla^s\vec{\Phi}|\vec{H}|^2  -\frac{1}{2} \nabla^s\vec{H} - (\vec{H}\cdot \vec{h}^{sk})\nabla_k\vec{\Phi} \right]
\end{align}
so that equation \eqref{chinedui} under translation invariance  becomes
\begin{align}
\delta\int_{\Sigma_0}|\vec{H}|^2 d\textnormal{vol}_g=\vec{a}\cdot \int_{\Sigma_0} (\nabla_s \vec{V}_{H^2}^s +\vec{\mathcal{W}}_{H^2} )\,\,d\textnormal{vol}_g=0 \label{dengu}
\end{align}
where $\vec{V}_{H^2}^s:=\nabla^s\vec{\Phi}|\vec{H}|^2  -\frac{1}{2} \nabla^s\vec{H} - (\vec{H}\cdot \vec{h}^{sk})\nabla_k\vec{\Phi}$. As \eqref{dengu} holds on every small patch $\Sigma_0$, we have $-\nabla_s \vec{V}_{H^2}^s =\vec{\mathcal{W}}_{H^2} $. That is,
\begin{align}
-\nabla_s\left[\nabla^s\vec{\Phi}|\vec{H}|^2  -\frac{1}{2} \nabla^s\vec{H} - (\vec{H}\cdot \vec{h}^{sk})\nabla_k\vec{\Phi} \right]
&= \left[ (\vec{H}\cdot\vec{h}_{ij})\vec{h}^{ij} -4|\vec{H}|^2\vec{H} +\frac{1}{2} \Delta_\perp \vec{H}   \right]  \nonumber
\end{align}
and then the desired identity follows.
\end{proof}

\section{The other energies} \label{othergen}

The energy \eqref{jkas} is unbounded from below as observed by Robin-Graham and Reichert in \cite{robingraham} but we can add   the energy $\mu\int_{\Sigma}|h_0|^4\,\, d\textnormal{vol}_g$ to it in order to obtain a nonnegative energy. This holds for $\mu\geq\frac{1}{12}$. In this section, we will investigate these other nonnegative energies in codimension one. 
\begin{prop}[See Proposition 6.6 of \cite{robingraham}]
Let $\vec{\Phi}:\Sigma\rightarrow\mathbb{R}^5$ be an immersion. If $\mu\geq\frac{1}{12}$, then the energy $\mathcal{E}(\Sigma)+\mu\int_\Sigma |h_0|^4\,\, d\textnormal{vol}_g$ is nonnegative.
\end{prop}
\begin{proof}
By using the formula $(h_0)_{ij}=h_{ij}-Hg_{ij}$, we have
\begin{align}
\mathcal{E}(\Sigma)+\mu\int_{\Sigma}|h_0|^4&= \int_{\Sigma} |\nabla H|^2 -|H|^2(|h_0|^2 +4|H|^2) +7|H|^2 +\mu |h_0|^4  \nonumber
\\&= \int_{\Sigma} |\nabla H|^2  +3|H|^4 -|H|^2|h_0|^2 +\mu |h_0|^4 .  \nonumber
\end{align}
Since $\int_{\Sigma} |\nabla H|^2\geq 0$, a sufficient condition for 
$\mathcal{E}(\Sigma)+\mu\int_{\Sigma}|h_0|^4$ to be nonnegative is that the quadratic form $3x^2-xy+\mu y^2$ be nonnegative definite. This gives that $3\mu-\frac{1}{4}\geq 0$ which implies $\mu\geq\frac{1}{12}$.
\end{proof}

\begin{rmk}\label{idf}
Among all energies of the form 
$$\mathcal{E}(\Sigma)+\mu\int_\Sigma |h_0|^4\,\, d\textnormal{vol}_g\,,$$we strongly suspect that the only energy whose critical points include minimal hypersurfaces is $\mathcal{E}(\Sigma)$.
\end{rmk}
We now give an argument to substantiate this remark. Since  minimal hypersurfaces are critical points of $\mathcal{E}(\Sigma)$, we only need to vary  the energy $\int_{\Sigma}|h_0|^4 \,\, d\textnormal{vol}_g$ and show that its critical points do not include minimal hypersurfaces.

\vskip3mm
\noindent
By using the decomposition $(\vec{h}_0)_{ij}=\vec{h}_{ij}-\vec{H}g_{ij}$ we find that
\begin{align}
|\vec{h}_0|^4=|\vec{h}|^4 -8|\vec{H}\cdot\vec{h}|^2+16|\vec{H}|^4 . \nonumber
\end{align}
The variations of $\int_\Sigma|\vec{H}\cdot\vec{h}|^2 \,\,d\textnormal{vol}_g$ and $\int_\Sigma|\vec{H}|^4\,\,d\textnormal{vol}_g$ have already been done in the proof of Theorem \eqref{WE} and clearly their critical points include minimal hypersurfaces. So, we only need to vary $\int_\Sigma |h|^4\,\,d\textnormal{vol}_g$.

\vskip2mm
\noindent
We find
\begin{align}
\delta |\vec{h}|^2=\delta(\vec{h}_{ij}\cdot\vec{h}^{ij}) = A^s \nabla_s(\vec{h}_{ij}\cdot\vec{h}^{ij}) + 4(\vec{B}\cdot\vec{h}^{is})(\vec{h}_{ij}\cdot\vec{h}^j_s) + 2\vec{h}^{ij}\cdot\nabla_i\nabla_j\vec{B}                 . \nonumber
\end{align}
\noindent
Hence
\begin{align}
\delta|\vec{h}|^4 &= 2|\vec{h}|^2\delta|\vec{h}|^2                                    \nonumber
\\& = A^s\nabla_s|h|^4 + 8(\vec{B}\cdot\vec{h}^{is})(\vec{h}_{ij}\cdot\vec{h}^j_s) |h|^2+ 4(\vec{h}^{ij}\cdot\nabla_i\nabla_j\vec{B}) |\vec{h}|^2                   \nonumber
\end{align}
and
\begin{align}
|g|^{-1/2} \delta(|\vec{h}|^4 |g|^{1/2})&=       \nabla_s(A^s |\vec{h}|^4) + 8(\vec{B}\cdot \vec{h}^{is}) (\vec{h}_{ij}\cdot\vec{h}^j_s) |\vec{h}|^2                                   \nonumber
\\&\quad\quad\quad + 4(\vec{h}^{ij}\cdot\nabla_i\nabla_j\vec{B})|\vec{h}|^2-4|\vec{h}|^4\vec{B}\cdot\vec{H} \nonumber
\\&= \nabla_s V_3^s+ \vec{B}\cdot\vec{\mathcal{W}}_3 \nonumber
\end{align}
where
\begin{align}
V_3^s&:= A^s|\vec{h}|^4 +4(\vec{h}^{sj}\cdot\nabla_j\vec{B})|\vec{h}|^2 -4\vec{B}\cdot\nabla_j(\vec{h}^{js}|\vec{h}|^2) \nonumber
\end{align}
and
\begin{align}
\vec{\mathcal{W}}_3 := 8\vec{h}^{is}(\vec{h}_{ij}\cdot\vec{h}^j_s)|\vec{h}|^2 -4|\vec{h}|^4\vec{H} +4\pi_{\vec{n}}\nabla_j\nabla_i(\vec{h}^{ij}|\vec{h}|^2)       . \label{jagaban78}
\end{align}

The critical point of the energy $\int_\Sigma |h|^4\,\, d\textnormal{vol}_g$ satisfy
\begin{align}
\int_{\Sigma_0}(\vec{B}\cdot\mathcal{\vec{W}}_3+\nabla_jV^j_3)\,\, d\textnormal{vol}_g=0 .\label{kuj11hjk}
\end{align}
As \eqref{kuj11hjk} holds on every small patch $\Sigma_0$ of the manifold $\Sigma$, we can write
$$\vec{B}\cdot\mathcal{\vec{W}}_3+\nabla_jV^j_3=0.$$
The Euler-Lagrange equation for $\mathcal{E}(\Sigma)$ corresponds to the normal variation of $\int_\Sigma |h|^4\,\, d\textnormal{vol}_g$, that is $\pi_{\vec{n}}\delta\vec{\Phi}=\vec{B}$.
This gives
\begin{align}
\mathcal{\vec{W}}_3=\vec{0}.  \label{zihuaguoyann}
\end{align}
We will now argue that minimal surfaces do not satisfy \eqref{zihuaguoyann}.
Assume that $\Sigma$ is a minimal hypersurface. 
The minimality assumption implies that 
\begin{align}
\vec{H}=\frac{1}{4}\vec{h}_j^j=\vec{0}.  \label{minim420}
\end{align}  
By using \eqref{minim420} and the contracted Codazzi-Mainardi equation $(\pi_{\vec{n}}\nabla_i \vec{h}_{j}^j=\pi_{\vec{n}}\nabla_j \vec{h}_{i}^j)$ in \eqref{jagaban78}, we find
\begin{align}
\mathcal{\vec{W}}_3&=8\vec{h}^{is}(\vec{h}_{ij}\cdot\vec{h}^j_s)|\vec{h}|^2  +4\pi_{\vec{n}}\nabla_j\nabla_i(\vec{h}^{ij}|\vec{h}|^2)  \nonumber
\\&= 4\vec{h}^{is}(\vec{h}_{ij}\cdot\vec{h}^j_s)|\vec{h}|^2  +4\vec{h}^{ij}\nabla_i\nabla_j|\vec{h}|^2 . \label{hard}
\end{align}
It is hard to imagine that the assumption $\vec{H}=\vec{0}$ would make the right hand side of \eqref{hard} vanish. However, we do not rule out the possibility of having a minimal hypersurface for which the right hand side of \eqref{hard} would vanish. Thus, we strongly suspect that for any $\mu$, the critical points of the energy \begin{align}
\mathcal{E}(\Sigma)+\mu\int_\Sigma|h_0|^4\,\,d\textnormal{vol}_g \label{NO00}
\end{align}
 do not include minimal hypersurfaces.

\begin{rmk}
One could also add another curvature term to have
\begin{align}
\mathcal{E}(\Sigma)+\int_\Sigma \sigma\textnormal {Tr}(h_0^4) +\mu|h_0|^4\,\,d\textnormal{vol}_g\quad\quad \sigma,\mu\in \mathbb{R}. \label{NO01}
 \end{align}
However, computations show a strong evidence that minimal hypersurfaces are not critical points of the energies  \eqref{NO00} and \eqref{NO01}.
\end{rmk}

\section{Conservation laws}  \label{sekaka}
In this section, the left hand side of the Willmore-type equation \eqref{abc} is reformulated in terms of some divergence free quantities with the help of Noether's theorem.

\vskip3mm
\noindent
Let $\Omega$ be an open subset of $D\subset \mathbb{R}^n$ and let $E\subset \mathbb{R}^m$. Let the Lagrangian
$$L:\{ (x,y,z): (x,y)\in D\times E,\, z\in T_y E\otimes T^* _x D \}\mapsto \mathbb{R}$$
be continuously differentiable. By choosing a $C^1$ density measure $d\mu (x)$ on $\Omega$, we can define a functional $\mathcal{L}$ on the set of maps $C^1(\Omega, E)$ via 
$$\mathcal{L}(u):= \int_{\Omega} L(x, u(x), du(x))\, d\mu(x).$$
A tangent vector $X$ is an {\it infinitesimal symmetry} for $\mathcal{L}$ if and only if 
$$\frac{\partial L}{\partial y^i}(x, y, z) X^i(y) +\frac{\partial L}{\partial z^i_\alpha}(x, y, z) \frac{\partial X^i}{\partial y^j}(y)z^j_\alpha =0.$$
The vector field $J$ defined by  
$$J^\alpha:= \rho(x)\frac{\partial L}{\partial z^i_\alpha}(x, u, du) X^i(u)$$
where $\rho$ satisfies $d\mu(x)= \rho(x) dx^1 \cdots dx^n$, is often called {\it Noether current} in  Physics. We are now ready to state a version of Noether's theorem relevant for our purpose. 

\begin{theo}\label{thirstyy}
Let $X$ be a Lipschitz tangent vector field on $E$, which is an infinitesimal symmetry for $\mathcal{L}$. If $u:\Omega \rightarrow  E$ is a critical point of $\mathcal{L}$, then 
$$\sum_{\alpha=1}^n \frac{\partial J^\alpha}{\partial x^\alpha}  =0$$
where $\{ x^\alpha\}_{\alpha=1,...,n}$ are coordinates on $\Omega$ such that $d\mu(x)= \rho(x) dx^1 \cdots dx^n$.
\end{theo}

\noindent
The following result holds by applying Noether's theorem to the invariances of the energy $\mathcal{E}(\Sigma).$
\begin{theo}\label{cons-law}
Let $\vec{\Phi}:\Sigma\rightarrow \mathbb{R}^m$ be a smooth immersion of a 4-dimensional manifold $\Sigma$ and let $\mathcal{\vec{W}}$ be as defined in Proposition \ref{WE} with $\mathcal{\vec{W}}=\vec{0}$. Define the following quantity
\begin{align}
\vec{V}^j:&=  |\pi_{\vec{n}}\nabla \vec{H}|^2\nabla^j\vec{\Phi}-2(\pi_{\vec{n}}\nabla^j\vec{H}\cdot \pi_{\vec{n}}\nabla_i\vec{H})\nabla^i\vec{\Phi}+(\vec{h}_{i}^j\cdot\Delta_{\perp}\vec{H})\nabla^i\vec{\Phi}  \nonumber
\\&\quad+\frac{1}{2} \nabla^j\Delta_{\perp}\vec{H}  -2(\vec{h}_{i}^j\cdot\nabla^i\vec{H})\vec{H}+2(\vec{H}\cdot\vec{h}_{i}^j)\pi_{\vec{n}}\nabla^i\vec{H} -2\nabla_i((\vec{H}\cdot\vec{h}^{ij})\vec{H})     \nonumber
\\&\quad+4\pi_{\vec{n}}\nabla_i((\vec{H}\cdot\vec{h}^{ij})\vec{H})-\frac{1}{2}\nabla^j((\vec{H}\cdot\vec{h}^{ik})\vec{h}_{ik}) -|\vec{H}\cdot\vec{h}|^2\nabla^j\vec{\Phi}      \nonumber \\&\quad+\pi_{\vec{n}}\nabla^j((\vec{H}\cdot\vec{h}^{ik})\vec{h}_{ik})+7|\vec{H}|^4\nabla^j\vec{\Phi}+7\nabla^j(|\vec{H}|^2\vec{H})-14\pi_{\vec{n}}\nabla^j(|\vec{H}|^2\vec{H}). \nonumber
\end{align}

\noindent
Then the following conservation laws hold by the invariance of $\mathcal{E}(\Sigma)$ under translation, dilation and rotation respectively.
 \[   \left\{
\begin{array}{ll}
    \vec{0}&=\nabla_j\vec{V}^j\\
      0&=\nabla_j(\vec{\Phi}\cdot\vec{V}^j-\nabla^j|\vec{H}|^2) \\
    \vec{0}&=\nabla_j(\vec{\Phi}\wedge\vec{V}^j+ \vec{M}^j)\\
\end{array} 
\right. \]
where $\vec{M}^j$ denotes the two-vector
\begin{align}\vec{M}^j:&=\frac{1}{2}\Delta_{\perp}\vec{H}\wedge\nabla^j\vec{\Phi}+2\vec{H}\wedge\pi_{\vec{n}}\nabla^j\vec{H}+2(\vec{H}\cdot\vec{h}^{ij})\vec{H}\wedge \nabla_i\vec{\Phi}\nonumber
\\&\quad+\frac{1}{2}(\vec{H}\cdot\vec{h}^{ik})\vec{h}_{ik}\wedge\nabla^j\vec{\Phi}-7|\vec{H}|^2\vec{H}\wedge\nabla^j\vec{\Phi}.  \nonumber\end{align}

\end{theo}

\begin{proof}
$\empty$\\
\textbf{Translation.}
Consider a translation $\delta\vec{\Phi}=\vec{a}$, where $\vec{a}$ is a constant vector in $\mathbb{R}^m$. Then 
$$A_j=\vec{a}\cdot\nabla_j\vec{\Phi}\quad\mbox{and}\quad \vec{B}=\pi_{\vec{n}}\vec{a}.$$

\noindent
If $\vec{U}$ is any vector, then it holds
$$\vec{U}\cdot\vec{B}=\vec{a}\cdot\pi_{\vec{n}}\vec{U}$$
and
\begin{align} \pi_{\vec{n}}\vec{U}\cdot\nabla_j\vec{B}&=\nabla_j(\vec{B}\cdot\pi_{\vec{n}}\vec{U}) -\vec{B}\cdot \nabla_j\pi_{\vec{n}}\vec{U}   \nonumber
\\&= \nabla_j(\vec{a}\cdot \pi_{\vec{n}}\vec{U}) -\pi_{\vec{n}}\vec{a}\cdot\nabla_j\pi_{\vec{n}} \vec{U} \nonumber
\\&= \vec{a}\cdot \nabla_j\pi_{\vec{n}}\vec{U} -\vec{a}\cdot\pi_{\vec{n}}\nabla_j\pi_{\vec{n}} \vec{U} \nonumber
\\&=\vec{a}\cdot\pi_{T}\nabla_j\pi_{\vec{n}}\vec{U}. \nonumber
\end{align}
Hence
\begin{align}
\vec{U}\cdot\Delta_{\perp}\vec{B}&= \pi_{\vec{n}}\vec{U}\cdot\nabla^i\pi_{\vec{n}}\nabla_i\vec{B}=\nabla^i(\pi_{\vec{n}}\vec{U}\cdot\nabla_i\vec{B})-(\pi_{\vec{n}}\nabla^i\pi_{\vec{n}}\vec{U})\cdot \nabla_i\vec{B}   \nonumber
\\&= \nabla^i\left[\vec{a}\cdot\pi_{T}\nabla_i\pi_{\vec{n}}\vec{U}   \right]-\vec{a}\cdot \pi_{T}\nabla_i\pi_{\vec{n}}\nabla^i\pi_{\vec{n}}\vec{U}    \nonumber
\\&= \vec{a}\cdot \left[ \nabla^i\pi_{T}\nabla_i\pi_{\vec{n}}\vec{U}-\pi_{T}\nabla_i\pi_{\vec{n}}\nabla^i\pi_{\vec{n}}\vec{U}     \right].    \nonumber
\end{align}
Then we have,
$$V^j=\vec{a}\cdot\vec{V}^j,$$
where
\begin{align}
\vec{V}^j&= | \pi_{\vec{n}}\nabla \vec{H}  |^2\nabla^j\vec{\Phi}+\frac{1}{2} (\vec{h}_{ik}\cdot \nabla^j\vec{H})\vec{h}^{ik}-2(\vec{h}_{i}^j\cdot \nabla^i\vec{H})\vec{H}+ 2(\vec{H}\cdot\vec{h}_{i}^j)\pi_{\vec{n}}\nabla^i\vec{H}                \nonumber
\\ &\quad+ \frac{1}{2}\pi_{\vec{n}}\nabla^j\Delta_{\perp}\vec{H}-\frac{1}{2}\pi_{T}\nabla^j\Delta_{\perp}\vec{H}+\frac{1}{2} \nabla^i\pi_{T}\nabla_i\pi_{\vec{n}}\nabla^j\vec{H}-\frac{1}{2}\pi_{T}\nabla_i\pi_{\vec{n}}\nabla^i\pi_{\vec{n}}\nabla^j\vec{H}   \nonumber
\\&\quad -2\pi_{T}\nabla_i((\vec{H}\cdot\vec{h}^{ij})\vec{H})
+ 2\pi_{\vec{n}} \nabla_i((\vec{H}\cdot\vec{h}^{ij})\vec{H})-\frac{1}{2}\pi_{T}\nabla^j((\vec{H}\cdot\vec{h}^{ik})\vec{h}_{ik}) \nonumber
\\&\quad+\frac{1}{2}\pi_{\vec{n}}\nabla^j((\vec{H}\cdot\vec{h}^{ik})\vec{h}_{ik})-|\vec{H}\cdot\vec{h}|^2 \nabla^j\vec{\Phi}             +7|\vec{H}|^2\nabla^j\vec{\Phi}+7\pi_{T}\nabla^j(|\vec{H}|^2\vec{H})  \nonumber
\\&\quad-7\pi_{\vec{n}}\nabla^j(|\vec{H}|^2\vec{H})
 \nonumber
\\&= |\pi_{\vec{n}}\nabla \vec{H}|^2\nabla^j\vec{\Phi}-2(\pi_{\vec{n}}\nabla^j\vec{H}\cdot \pi_{\vec{n}}\nabla_i\vec{H})\nabla^i\vec{\Phi}+(\vec{h}_{i}^j\cdot\Delta_{\perp}\vec{H})\nabla^i\vec{\Phi}+\frac{1}{2} \nabla^j\Delta_{\perp}\vec{H}   \nonumber
\\&\quad -2(\vec{h}_{i}^j\cdot\nabla^i\vec{H})\vec{H}+2(\vec{H}\cdot\vec{h}_{i}^j)\pi_{\vec{n}}\nabla^i\vec{H} -2\nabla_i((\vec{H}\cdot\vec{h}^{ij})\vec{H}) +4\pi_{\vec{n}}\nabla_i((\vec{H}\cdot\vec{h}^{ij})\vec{H})    \nonumber
\\&\quad-\frac{1}{2}\nabla^j((\vec{H}\cdot\vec{h}^{ik})\vec{h}_{ik})   -|\vec{H}\cdot\vec{h}|^2\nabla^j\vec{\Phi}     +\pi_{\vec{n}}\nabla^j((\vec{H}\cdot\vec{h}^{ik})\vec{h}_{ik})+7|\vec{H}|^2\nabla^j\vec{\Phi} \nonumber
\\&\quad+7\nabla^j(|\vec{H}|^2\vec{H})-14\pi_{\vec{n}}\nabla^j(|\vec{H}|^2\vec{H}) .
\nonumber
\end{align}

\noindent
We note that 
$$\nabla_j\vec{\Phi}\cdot \vec{V}^j=\Delta_g|\vec{H}|^2.$$

\vskip 3mm
\textbf{Dilation.}
Next, we consider a dilation $\delta\vec{\Phi}=\lambda\vec{\Phi}$,  for some constant $\lambda\in\mathbb{R}$. Clearly,
$$A_j=\lambda\vec{\Phi}\cdot\nabla_j\vec{\Phi}\quad\mbox{and}\quad\vec{B}=\lambda\pi_{\vec{n}}\vec{\Phi}.$$

\noindent
If $\vec{U}$ is any vector, then it holds
$$\vec{U}\cdot\vec{B}=\lambda\vec{\Phi}\cdot\pi_{\vec{n}}\vec{U}$$
and
\begin{align}
\quad\pi_{\vec{n}}\vec{U}\cdot\nabla_j\vec{B}&=  \nabla_j(\vec{B}\cdot\pi_{\vec{n}}\vec{U}) -\vec{B}\cdot \nabla_j\pi_{\vec{n}}\vec{U}             \nonumber
\\&=  \nabla_j(\lambda\pi_{\vec{n}}\vec{\Phi}\cdot\vec{U})  - \lambda\pi_{\vec{n}}\vec{\Phi}\cdot\nabla_j\pi_{\vec{n}}\vec{U}        \nonumber
\\&= \lambda\left(\vec{\Phi}\cdot \nabla_j\pi_{\vec{n}}\vec{U}-\pi_{\vec{n}}\nabla_j\pi_{\vec{n}}\vec{U}  \right)   \nonumber
\\&=\lambda \vec{\Phi}\cdot\pi_{T}\nabla_j\pi_{\vec{n}}\vec{U}.  \nonumber
\end{align}
Hence
\begin{align}
\vec{U}\cdot \Delta_{\perp}\vec{B}&= \pi_{\vec{n}}\vec{U}\cdot\nabla^i\pi_{\vec{n}}\nabla_i\vec{B}=\nabla^i(\pi_{\vec{n}}\vec{U}\cdot \nabla_{i}\vec{B})-(\pi_{\vec{n}}\nabla^i\pi_{\vec{n}}\vec{U})\cdot \nabla_i\vec{B}   \nonumber
\\&= \nabla^i\left[\lambda \vec{\Phi}\cdot \pi_{T}\nabla_i\pi_{\vec{n}}\vec{U}  \right]-\lambda\vec{\Phi}\cdot \pi_{T}\nabla_{i}\pi_{\vec{n}}\nabla^i\pi_{\vec{n}}\vec{U}  \nonumber
\\&= \lambda\vec{\Phi}\cdot \left[ \nabla^i\pi_{T}\nabla_i\pi_{\vec{n}}\vec{U}-\pi_{T}\nabla_i\pi_{\vec{n}}\nabla^i\pi_{\vec{n}}\vec{U}  \right]-4\lambda\vec{H}\cdot \vec{U}   \nonumber
\end{align}
 From this and the computation above for  translation, we find
$$V^j=\lambda\vec{\Phi}\cdot\vec{V}^j-2\lambda\vec{H}\cdot\nabla^j\vec{H}=\lambda\left[ \vec{\Phi}\cdot\vec{V}^j-\nabla^j|\vec{H}|^2  \right].$$

\vskip 3mm
\noindent

\noindent\textbf{Rotation.} We now consider an infinitesimal rotation given by\footnote{Here, we use $\star$ to denote the Hodge star operator on the ambient space. } $$\delta\vec{\Phi}=\star(\vec{b}\wedge\vec{\Phi}) \quad\mbox{for some constant}\quad \vec{b}\in\Lambda^{m-2}(\mathbb{R}^m).$$
Clearly,
$$A_j=\vec{b}\cdot\star(\vec{\Phi}\wedge\nabla_j\vec{\Phi})\quad\mbox{and}\quad \vec{B}=(\vec{b}\cdot\star(\Phi\wedge\vec{n}_{\alpha}))\vec{n}_{\alpha}$$
where we have summed over $\alpha$.
If $\vec{U}$ is any vector, then it holds that
$$\vec{U}\cdot\vec{B}=\vec{b}\cdot\star(\vec{\Phi}\wedge\pi_{\vec{n}}\vec{U})$$
and
\begin{align}
\pi_{\vec{n}}\vec{U} \cdot \nabla_j\vec{B}&=\nabla_j(\vec{B}\cdot\pi_{\vec{n}}\vec{U}) -\vec{B}\cdot \nabla_j\pi_{\vec{n}} \vec{U}   \nonumber
\\ &=\vec{b}\cdot\star(\vec{\Phi}\wedge\pi_{T}\nabla_j\pi_{\vec{n}}\vec{U}+\nabla_j\vec{\Phi}\wedge\pi_{\vec{n}}\vec{U}). \nonumber\end{align}

\noindent
Hence 
\begin{align}
\vec{U}\cdot\Delta_{\perp}\vec{B}&=\pi_{\vec{n}}\vec{U}\cdot\nabla^i\pi_{\vec{n}}\nabla_i\vec{B}=\nabla^i(\pi_{\vec{n}}\vec{U} \cdot \nabla_i\vec{B})-(\pi_{\vec{n}}\nabla^i\pi_{\vec{n}}\vec{U} )\cdot \nabla_i\vec{B}    \nonumber
\\&=  \nabla^i\left(\vec{b}\cdot\star (\vec{\Phi}\wedge\pi_{T}\nabla_i\pi_{\vec{n}}\vec{U}+\nabla_i\vec{\Phi}\wedge\pi_{\vec{n}} \vec{U})  \right)         \nonumber \\&\quad\quad-\vec{b}\cdot\star\left(\vec{\Phi}\wedge\pi_{T}\nabla_i\pi_{\vec{n}}\nabla^i\pi_{\vec{n}}\vec{U}+\nabla_i\vec{\Phi}\wedge\pi_{\vec{n}}\nabla^i\pi_{\vec{n}}\vec{U}\right)   \nonumber
\\&= \vec{b}\cdot\star \left[\vec{\Phi}\wedge\left(\nabla^i\pi_{T}\nabla_i\pi_{\vec{n}}\vec{U}-\pi_{T}\nabla_i\pi_{\vec{n}} \nabla^i\pi_{\vec{n}}\vec{U} \right) +\nabla^i\vec{\Phi}\wedge\pi_{T}\nabla_i\pi_{\vec{n}}\vec{U}                    \right.\nonumber
\\&\quad\quad\left. +\nabla^i(\nabla_i\vec{\Phi}\wedge\pi_{\vec{n}}\vec{U})
-\nabla_i\vec{\Phi}\wedge\pi_{\vec{n}}\nabla^i\pi_{\vec{n}}\vec{U} \right]  \nonumber
\\&=\vec{b}\cdot\star\left[\vec{\Phi}\wedge\left( \nabla^i\pi_{T}\nabla_i\pi_{\vec{n}}\vec{U}-\pi_{T}\nabla_i\pi_{\vec{n}} \nabla^i\pi_{\vec{n}}\vec{U} \right) +2\nabla^i\vec{\Phi}\wedge \pi_{T}\nabla_i\pi_{\vec{n}}\vec{U}+4\vec{H}\wedge\pi_{\vec{n}}\vec{U} \right]   \nonumber
\\&=\vec{b}\cdot\star\left[ \vec{\Phi}\wedge\left( \nabla^i\pi_{T}\nabla_i\pi_{\vec{n}}\vec{U}-\pi_{T}\nabla_i\pi_{\vec{n}}\nabla^i\pi_{\vec{n}}\vec{U} \right)+4\vec{H}\wedge\pi_{\vec{n}}\vec{U} \right].   \nonumber
\end{align}
From this and the computation above for translation, we find
\begin{align}
V^j&=\vec{b}\cdot\star\left[ \vec{\Phi}\wedge\vec{V}^j+ \frac{1}{2}\Delta_{\perp} \vec{H}\wedge\nabla^j\vec{\Phi}      +2\vec{H}\wedge\pi_{\vec{n}}\nabla^j\vec{H} +2(\vec{H}\cdot\vec{h}^{ij})\vec{H}\wedge\nabla_i\vec{\Phi} \right. \nonumber
\\&\quad\quad\left.+\frac{1}{2}(\vec{H}\cdot\vec{h}^{ik})\vec{h}_{ik}\wedge\nabla^j\vec{\Phi} -7|\vec{H}|^2\vec{H}\wedge\nabla^j\vec{\Phi}\right] . \nonumber
\end{align}
\end{proof}

\section{Conformal invariance} \label{sekuku}
The goal of this Section is to prove the following Theorem. We adopt the techniques found in \cite {jguv} in our proof. We assume that $\Sigma$ is a closed manifold so that boundary terms vanish whenever we integrate by parts. 
\begin{theo} \label{kofoworola}
The energy $\mathcal{E}(\Sigma)$ is conformally invariant.
\end{theo}

\begin{proof}
Consider the energy
 $$\mathcal{E}_0:=\int_{\Sigma}|\pi_{\vec{n}}\nabla\vec{H}|^2 d\text{vol}_g.$$

\noindent
Since it is invariant under reparametrisation, only  normal variations matter (tangential variations contribute to terms that appear as exact derivatives and can be integrated away). Moreover, as we have seen by using Noether's theorem and the invariance under translation, critical points of $\mathcal{E}_0$ satisfy an equation of the type

$$\nabla_j\vec{F}^j=\vec{\mathcal{W}}_0,$$

\noindent
for some suitable vector field $\vec{F}^j$ and $\vec{\mathcal{W}}_0$ is a normal vector. Actually, we can identify $\vec{F}^j$ as the component of the stress-energy tensor associated with $\mathcal{E}_0$. 
Accordingly, for a variation $\delta\vec{\Phi}$, we have 

$$-\delta\mathcal{E}_0=-\int_{\Sigma}\vec{\mathcal{W}_0}\cdot\delta\vec{\Phi}=-\int_{\Sigma}\nabla_j\vec{F}^j\cdot\delta\vec{\Phi}=\int_{\Sigma}\vec{F}^j\cdot\nabla_j\delta\vec{\Phi}.$$

\noindent
Let us consider a special conformal variation of the type
$$\delta\vec{\Phi}=|\vec{\Phi}|^2\vec{c}-2(\vec{c}\cdot\vec{\Phi})\vec{\Phi}$$
where $\vec{c}\in \mathbb{R}^m$ is a constant vector.
\noindent
We see that $\vec{F}^j$ satisfies
$$\vec{F}^j=F^{ij}\nabla_i\vec{\Phi}+\pi_{\vec{n}}\nabla_i\vec{G}^{ij}+C^j\vec{H}$$
\noindent
where
\noindent
$$F^{ij}=|\pi_{\vec{n}}\nabla\vec{H}|^2g^{ij}-2\nabla^j\vec{H}\cdot\pi_{\vec{n}}\nabla^i\vec{H} +\frac{1}{2}\vec{h}^{ij}\cdot\Delta_{\perp}\vec{H},$$

$$\vec{G}^{ij}=2(\vec{H}\cdot\vec{h}^{ij})\vec{H} +\frac{1}{2}g^{ij}\Delta_{\perp} \vec{H}$$
\noindent
and
\noindent
$$C^j= -4\nabla^j|\vec{H}|^2-4\vec{h}^{jk}\cdot\nabla_k\vec{H}= 4\nabla_k(|\vec{H}|^2g^{jk}-\vec{H}\cdot\vec{h}^{jk}).$$

\noindent
Let us note that if $\vec{u}^j$ is a normal vector, then 
\begin{align}
&\int_{\Sigma}\pi_{\vec{n}}\nabla_i \vec{u}^{ij}\cdot \nabla_j (|\vec{\Phi}|^2\vec{c}-2(\vec{c}\cdot\vec{\Phi})\vec{\Phi})  \nonumber 
\\&=\int_{\Sigma}\nabla_i\vec{u}^{ij}\cdot\pi_{\vec{n}} \nabla_j (|\vec{\Phi}|^2\vec{c}-2(\vec{c}\cdot\vec{\Phi})\vec{\Phi})  \nonumber
\\&=\int_{\Sigma} \nabla_i\vec{u}^{ij}\cdot(\nabla_j|\vec{\Phi}|^2\pi_{\vec{n}}\vec{c} -2(\vec{c}\cdot\nabla_j\vec{\Phi})\pi_{\vec{n}}\vec{\Phi}  )  \nonumber
\\&= \int_{\Sigma} \nabla_i\vec{u}^{ij}\cdot (\nabla_j|\vec{\Phi}|^2\vec{c}-2(\vec{c}\cdot\nabla_j\vec{\Phi})\vec{\Phi})  - \int_{\Sigma} \nabla_i\vec{u}^{ij}\cdot (\nabla_j|\vec{\Phi}|^2\pi_{T}\vec{c} -2(\vec{c}\cdot\nabla_j\vec{\Phi})\pi_{T}\vec{\Phi})  \nonumber
\\&= \int_{\Sigma}\nabla_i\vec{u}^{ij} \cdot(\nabla_j|\vec{\Phi}|^2\vec{c}-2(\vec{c}\cdot\nabla_j\vec{\Phi})\vec{\Phi}) +  \int_{\Sigma} \vec{u}^{ij} \cdot\nabla_i(\nabla_j|\vec{\Phi}|^2\pi_{T}\vec{c} -2(\vec{c}\cdot\nabla_j\vec{\Phi})\pi_{T}\vec{\Phi})  \nonumber
\\&= \int_{\Sigma} \nabla_i\vec{u}^{ij}\cdot (\nabla_j|\vec{\Phi}|^2\vec{c} -2 (\vec{c}\cdot\nabla_j\vec{\Phi})\vec{\Phi})  +\int_{\Sigma} \vec{u}^{ij}\cdot (\nabla_j|\vec{\Phi}|^2\pi_{\vec{n}}\nabla_i\pi_{T}\vec{c} -2(\vec{c}\cdot\nabla_j\vec{\Phi})\pi_{\vec{n}}\nabla_i\pi_{T}\vec{\Phi}) .  \nonumber
\end{align}

\noindent
But 
\begin{align}
\pi_{\vec{n}}\nabla_i\pi_{T}\vec{c}=(\vec{c}\cdot\nabla_k\vec{\Phi})\vec{h}_i^k   \quad\mbox{and} \quad \pi_{\vec{n}}\nabla_i\pi_T\vec{\Phi} =(\vec{\Phi}\cdot\nabla_k\vec{\Phi})\vec{h}_i^k=\frac{1}{2}\nabla_k|\vec{\Phi}|^2\vec{h}_i^k.  \nonumber
\end{align}

\noindent
Hence, we find
\begin{align}
&\int_{\Sigma} \vec{u}^{ij}\cdot(\nabla_j|\vec{\Phi}|^2\pi_{\vec{n}}\nabla_i\pi_{T}\vec{c}-2(\vec{c}\cdot\nabla_j\vec{\Phi})\pi_{\vec{n}}\nabla_i\pi_{T}\vec{\Phi})  \nonumber
\\& = \int_{\Sigma}(\vec{h}^k_i\cdot\vec{u}^{ij})\left(\nabla_j|\vec{\Phi}|^2(\vec{c}\cdot\nabla_k\vec{\Phi})-\nabla_k|\vec{\Phi}|^2(\vec{c}\cdot\nabla_j\vec{\Phi})\right).   \nonumber
\end{align}

\noindent
This is in particular the case for $\vec{u}^{ij}=\vec{G}^{ij}$ which we will adopt later.
\\ 
Now, we have
$$\vec{F}^j\cdot\nabla_j\delta\vec{\Phi}=(F^{ij}\nabla_i\vec{\Phi} +\pi_{\vec{n}}\nabla_i\vec{G}^{ij} +C^j\vec{H})\cdot\nabla_j(|\vec{\Phi}|^2\vec{c}-2(\vec{c}\cdot\vec{\Phi})\vec{\Phi}).$$

\noindent
Upon integration by parts, the latter yields\footnote{See page \pageref{se123} for the definition of the interior multiplication $\mathlarger{\mathlarger{\mathlarger{\mathlarger{\llcorner}} }}$.}
\begin{align}
\delta\mathcal{E}_0&= -2\vec{c}\cdot \int_{\Sigma} \vec{G}^j_j- (\vec{G}^{ij}\wedge\vec{h}_{ij})\mathlarger{\mathlarger{\mathlarger{\mathlarger{\llcorner}} }} \vec{\Phi} + C^j (\vec{H}\wedge\nabla_j\vec{\Phi})\mathlarger{\mathlarger{\mathlarger{\mathlarger{\llcorner}} }}\vec{\Phi}   \nonumber
\\&\hspace{2cm}+\int_{\Sigma} F^{ij} \nabla_i\vec{\Phi}\cdot\nabla_j(|\vec{\Phi}|^2\vec{c}-2(\vec{c}\cdot\vec{\Phi})\vec{\Phi})  \nonumber
\\&\hspace{3cm}+  \int_{\Sigma}(\vec{h}^k_i\cdot\vec{G}^{ij})\left(\nabla_j|\vec{\Phi}|^2(\vec{c}\cdot\nabla_k\vec{\Phi})-\nabla_k|\vec{\Phi}|^2(\vec{c}\cdot\nabla_j\vec{\Phi})\right)  \nonumber
\\&= -2\vec{c}\cdot\int_{\Sigma} \vec{G}^j_j + F^j_j\vec{\Phi} -(\vec{G}^{ij}\wedge \vec{h}_{ij})\mathlarger{\mathlarger{\mathlarger{\mathlarger{\llcorner}} }} \vec{\Phi} + C^j (\vec{H}\wedge\nabla_j\vec{\Phi})\mathlarger{\mathlarger{\mathlarger{\mathlarger{\llcorner}} }}\vec{\Phi}   \nonumber
\\&\hspace{2cm} +\int_{\Sigma}  F^{ij} \nabla_j|\vec{\Phi}|^2(\vec{c}\cdot\nabla_i\vec{\Phi}) - F^{ij}\nabla_i|\vec{\Phi}|^2(\vec{c}\cdot\nabla_j\vec{\Phi})  \nonumber
\\&\hspace{3cm} +\int_{\Sigma}(\vec{h}^k_i\cdot\vec{G}^{ij})\left(\nabla_j|\vec{\Phi}|^2(\vec{c}\cdot\nabla_k\vec{\Phi})-\nabla_k|\vec{\Phi}|^2(\vec{c}\cdot\nabla_j\vec{\Phi})\right)  \nonumber
\\&= -2\vec{c}\cdot\int_{\Sigma} \vec{G}^j_j +F^j_j\vec{\Phi}-(\vec{G}^{ij}\wedge \vec{h}_{ij})\mathlarger{\mathlarger{\mathlarger{\mathlarger{\llcorner}} }} \vec{\Phi} + C^j (\vec{H}\wedge\nabla_j\vec{\Phi})\mathlarger{\mathlarger{\mathlarger{\mathlarger{\llcorner}} }}\vec{\Phi}   \nonumber
\\&\hspace{2cm} +\int_{\Sigma} (F^{kj}+\vec{h}^k_i\cdot\vec{G}^{ij}) (\nabla_j|\vec{\Phi}|^2(\vec{c}\cdot\nabla_k\vec{\Phi})-\nabla_k|\vec{\Phi}|^2(\vec{c}\cdot\nabla_j\vec{\Phi})). \label{u3s}
\end{align}

\noindent
On the other hand, we have
$$F^{kj}+\vec{h}^k_i\cdot\vec{G}^{ij}\quad\mbox{is symmetric}$$
\noindent
so that 
\begin{align}
(F^{kj}+\vec{h}^k_i\cdot\vec{G}^{ij}) (\nabla_j|\vec{\Phi}|^2(\vec{c}\cdot\nabla_k\vec{\Phi})-\nabla_k|\vec{\Phi}|^2(\vec{c}\cdot\nabla_j\vec{\Phi}))=0.\label{090}
\end{align}
Also,
$$\vec{G}^j_j=2\Delta_{\perp}\vec{H}+ 8|\vec{H}|^2\vec{H},\quad\quad F^j_j=2|\pi_{\vec{n}}\nabla\vec{H}|^2+2\vec{H}\cdot\Delta_\perp\vec{H}  =\Delta_g|\vec{H}|^2$$

\noindent
so that
\begin{align}
\int_\Sigma F^j_j\vec{\Phi} =4\int_\Sigma |\vec{H}|^2\vec{H}. \nonumber
\end{align}

\noindent
Hence
\begin{align}
\int_{\Sigma} \vec{G}^j_j +F^j_j\vec{\Phi}  =\int_{\Sigma} 28|\vec{H}|^2\vec{H}- 2(\vec{H}\cdot\vec{h}^{ij})\vec{h}_{ij},  \label{091}
\end{align}

\noindent
where we have used the following identity derived in Lemma \ref{rivvy}
\begin{align}
\nabla^s \vec{A}_s:= \nabla^s\left[\nabla_s\vec{H}-2(|\vec{H}|^2g_{sk}-\vec{H}\cdot\vec{h}_{sk})\nabla^k\vec{\Phi}  \right]=\Delta_\perp \vec{H} +(\vec{H}\cdot\vec{h}_{ij})\vec{h}^{ij} -8|\vec{H}|^2\vec{H}. \label {Riv}
\end{align}

\noindent
Next, integrating by parts, we find
\begin{align}
&\int_{\Sigma} C^j (\vec{H}\wedge\nabla_j\vec{\Phi}) \mathlarger{\mathlarger{\mathlarger{\mathlarger{\llcorner}} }} \vec{\Phi} =4\int_\Sigma \nabla_k (|\vec{H}|^2g^{jk} -\vec{H}\cdot\vec{h}^{jk}) (\vec{H}\wedge\nabla_j\vec{\Phi})\mathlarger{\mathlarger{\mathlarger{\mathlarger{\llcorner}} }} \vec{\Phi}  \nonumber
\\&= -4\int_{\Sigma}(|\vec{H}|^2g^{jk}-\vec{H}\cdot\vec{h}^{jk})\left[(\nabla_k\vec{H}\wedge\nabla_j\vec{\Phi})\mathlarger{\mathlarger{\mathlarger{\mathlarger{\llcorner}} }} \vec{\Phi} +(\vec{H}\wedge\vec{h}_{jk})\mathlarger{\mathlarger{\mathlarger{\mathlarger{\llcorner}} }} \vec{\Phi} +(\vec{H}\wedge\nabla_j\vec{\Phi})\mathlarger{\mathlarger{\mathlarger{\mathlarger{\llcorner}} }} \nabla_k\vec{\Phi}    \right] \nonumber
\\&= -4\int_{\Sigma}(|\vec{H}|^2g^{jk}-\vec{H}\cdot\vec{h}^{jk})\left[(\nabla_k\vec{H}\wedge\nabla_j\vec{\Phi})\mathlarger{\mathlarger{\mathlarger{\mathlarger{\llcorner}} }} \vec{\Phi}               +(\vec{H}\wedge\vec{h}_{jk})\mathlarger{\mathlarger{\mathlarger{\mathlarger{\llcorner}} }} \vec{\Phi}-\vec{H}g_{jk}   \right]  \nonumber
\\&= -4\int_{\Sigma}(|\vec{H}|^2g^{jk}-\vec{H}\cdot\vec{h}^{jk})\left[(\nabla_k\vec{H}\wedge\nabla_j\vec{\Phi})\mathlarger{\mathlarger{\mathlarger{\mathlarger{\llcorner}} }} \vec{\Phi}               +(\vec{H}\wedge\vec{h}_{jk})\mathlarger{\mathlarger{\mathlarger{\mathlarger{\llcorner}} }} \vec{\Phi} \right]  \nonumber
\\&=4\int_{\Sigma}(\vec{H}\cdot\vec{h}^{jk})\left[ (\nabla_k\vec{H}\wedge\nabla_j\vec{\Phi}) \mathlarger{\mathlarger{\mathlarger{\mathlarger{\llcorner}} }} \vec{\Phi} +(\vec{H}\wedge\vec{h}_{jk})\mathlarger{\mathlarger{\mathlarger{\mathlarger{\llcorner}} }} \vec{\Phi}  \right] -4\int_{\Sigma} |\vec{H}|^2(\nabla^j\vec{H}\wedge\nabla_j\vec{\Phi})\mathlarger{\mathlarger{\mathlarger{\mathlarger{\llcorner}} }} \vec{\Phi}.  \label{comb1}
\end{align}

\noindent
On the other hand, using again \eqref{Riv} yields
\begin{align}
&\int_{\Sigma} (\vec{G}^{ij}\wedge\vec{h}_{ij}) \mathlarger{\mathlarger{\mathlarger{\mathlarger{\llcorner}} }} \vec{\Phi}= 2\int_{\Sigma} (\vec{H}\cdot\vec{h}^{ij})(\vec{H}\wedge\vec{h}_{ij})\mathlarger{\mathlarger{\mathlarger{\mathlarger{\llcorner}} }} \vec{\Phi} -2\int_{\Sigma} (\vec{H}\wedge\Delta_{\perp}\vec{H})\mathlarger{\mathlarger{\mathlarger{\mathlarger{\llcorner}} }} \vec{\Phi}  \nonumber
\\&= 2 \int_{\Sigma} \left[(\Delta_\perp \vec{H}-(\vec{H}\cdot\vec{h}^{ij})\vec{h}_{ij})\wedge\vec{H}  \right]\mathlarger{\mathlarger{\mathlarger{\mathlarger{\llcorner}} }} \vec{\Phi}                    \nonumber
\\&= 2\int_\Sigma \left[(\nabla^s\vec{A}_s-2(\vec{H}\cdot\vec{h}_{ij})\vec{h}^{ij})  \wedge\vec{H} \right]\mathlarger{\mathlarger{\mathlarger{\mathlarger{\llcorner}} }} \vec{\Phi}   \nonumber
\\&=-4\int_{\Sigma} (\vec{H}\cdot\vec{h}_{ij})(\vec{h}^{ij}\wedge\vec{H})\mathlarger{\mathlarger{\mathlarger{\mathlarger{\llcorner}} }} \vec{\Phi} -2\int_\Sigma (\vec{A}^s\wedge\nabla_s\vec{H})  \mathlarger{\mathlarger{\mathlarger{\mathlarger{\llcorner}} }} \vec{\Phi}+ (\vec{A}^s\wedge\vec{H})\mathlarger{\mathlarger{\mathlarger{\mathlarger{\llcorner}} }} \nabla_s\vec{\Phi}          \nonumber
\\&= -4\int_{\Sigma} (\vec{H}\cdot\vec{h}_{ij})(\vec{h}^{ij}\wedge\vec{H})\mathlarger{\mathlarger{\mathlarger{\mathlarger{\llcorner}} }} \vec{\Phi} -2\int_\Sigma (\vec{A}^s\wedge\nabla_s\vec{H})  \mathlarger{\mathlarger{\mathlarger{\mathlarger{\llcorner}} }} \vec{\Phi} +(\vec{A}^s\cdot\nabla_s\vec{\Phi})\vec{H}  \nonumber
\\&=  -4\int_{\Sigma} (\vec{H}\cdot\vec{h}_{ij})(\vec{h}^{ij}\wedge\vec{H})\mathlarger{\mathlarger{\mathlarger{\mathlarger{\llcorner}} }} \vec{\Phi}   -2\int_\Sigma (\vec{A}^s\wedge\nabla_s\vec{H})  \mathlarger{\mathlarger{\mathlarger{\mathlarger{\llcorner}} }} \vec{\Phi} -4|\vec{H}|^2\vec{H}    \nonumber
\\&= \int_{\Sigma}-4 (\vec{H}\cdot\vec{h}_{ij})(\vec{h}^{ij}\wedge\vec{H})\mathlarger{\mathlarger{\mathlarger{\mathlarger{\llcorner}} }} \vec{\Phi}  -4|\vec{H}|^2(\nabla^j\vec{H}\wedge \nabla_j\vec{\Phi}) \mathlarger{\mathlarger{\mathlarger{\mathlarger{\llcorner}} }} \vec{\Phi} \nonumber
\\&\quad +\int_\Sigma 4(\vec{H}\cdot\vec{h}^{jk})(\nabla_k\vec{H}\wedge\nabla_j\vec{\Phi})\mathlarger{\mathlarger{\mathlarger{\mathlarger{\llcorner}} }} \vec{\Phi} +8|\vec{H}|^2\vec{H}.  \label{comb2}
\end{align}

\noindent
We find from \eqref{comb1} and \eqref{comb2} 
\begin{align}
\int_{\Sigma} C^j(\vec{H}\wedge\nabla_j\vec{\Phi})\mathlarger{\mathlarger{\mathlarger{\mathlarger{\llcorner}} }} \vec{\Phi} -(\vec{G}^{ij}\wedge \vec{h}_{ij})\mathlarger{\mathlarger{\mathlarger{\mathlarger{\llcorner}} }} \vec{\Phi} =-8\int_\Sigma |\vec{H}|^2\vec{H}. \label{broke}
\end{align}

\noindent
Using \eqref{090}, \eqref{091} and \eqref{broke} in \eqref{u3s} gives
\begin{align}
\delta\mathcal{E}_0=  -2\vec{c}\cdot\int_\Sigma 20|\vec{H}|^2\vec{H} -2(\vec{H}\cdot\vec{h}^{ij})\vec{h}_{ij}.              \nonumber
\end{align}

\noindent
Consider next the energy
\begin{align}
\mathcal{E}_1:=\int_\Sigma |\vec{H}\cdot\vec{h}|^2 d\textnormal{vol}_g  .  \nonumber      
\end{align}

\noindent
Using similar notation as above, we have this time
\begin{align}
F^{ij}=|\vec{H}\cdot\vec{h}|^2 g^{ij} -\frac{1}{2}(\vec{H}\cdot\vec{h}^{kl})(\vec{h}_{kl}\cdot\vec{h}^{ij}) -2(\vec{H}\cdot\vec{h}^{ik})(\vec{H}\cdot\vec{h}^j_k)        \nonumber
\end{align}

\noindent
and
\begin{align}
\vec{G}^{ij}=-2(\vec{H}\cdot\vec{h}^{ij})\vec{H} -\frac{1}{2}g^{ij} (\vec{H}\cdot \vec{h}^{kl})\vec{h}_{kl} \quad\quad\textnormal{and}\quad C^j=0.
\end{align}

\noindent
By same token as above
\begin{align}
\delta\mathcal{E}_1&= -2\vec{c}\cdot\int_{\Sigma} \vec{G}^j_j +F^j_j\vec{\Phi} - (\vec{G}^{ij}\wedge\vec{h}_{ij})\mathlarger{\mathlarger{\mathlarger{\mathlarger{\llcorner}} }} \vec{\Phi} = -2\vec{c}\cdot\int_\Sigma \vec{G}^j_j  \nonumber
\\&= -2\vec{c}\cdot\int_\Sigma -8|\vec{H}|^2\vec{H} -2(\vec{H}\cdot\vec{h}^{ij})\vec{h}_{ij}         \nonumber
\end{align}

\noindent
Finally, we consider the energy
\begin{align}
\mathcal{E}_2:= \int_\Sigma |\vec{H}|^4 d\textnormal{vol}_g.  \nonumber
\end{align}

\noindent
Using similar notation as above, we have
\begin{align}
F^{ij}=|\vec{H}|^4g^{ij} -|\vec{H}|^2\vec{H}\cdot\vec{h}^{ij}  ,  \nonumber
\end{align}
\begin{align}
\vec{G}^{ij}= -|\vec{H}^2|\vec{H} g^{ij}\quad\quad\textnormal{and}\quad C^j=0.  \nonumber
\end{align}

\noindent
By same token as above, we find
\begin{align}
\delta \mathcal{E}_2&= -2\vec{c}\cdot \int_{\Sigma} \vec{G}^j_j +F^j_j\vec{\Phi}  -(\vec{G}^{ij}\wedge\vec{h}_{ij})\mathlarger{\mathlarger{\mathlarger{\mathlarger{\llcorner}} }} \vec{\Phi} =-2\vec{c}\cdot\int_\Sigma \vec{G}^j_j   \nonumber
\\&= -2\vec{c}\cdot\int_\Sigma -4|\vec{H}|^2\vec{H}.  \nonumber
\end{align}

\noindent
Therefore 
\begin{align}
\delta\mathcal{E}:=\delta(\mathcal{E}_0+ a\mathcal{E}_1+b\mathcal{E}_2)=-2\vec{c}\cdot \int_\Sigma (-2a-2) (\vec{H}\cdot\vec{h}^{ij})\vec{h}_{ij} +(20-8a-4b) |\vec{H}|^2\vec{H}    \nonumber
\end{align}
so that only the values $a=-1$ and $b=7$ yields $\delta\mathcal{E}=0$.

\noindent
Observe that we have just shown that the only linear combination of $\mathcal{E}_0$, $\mathcal{E}_1$ and $\mathcal{E}_2$ that gives a conformally invariant energy is 
$$\int_{\Sigma}|\pi_{\vec{n}}\nabla \vec{H}|^2 -|\vec{H}\cdot \vec{h}|^2 +7|\vec{H}|^4.$$
This completes the proof.
\end{proof}

\noindent

\section{Existence of potential 2-forms} \label{sekoko}

In this section, we present further conservation laws. We briefly clarify some of the notations used in this Section.
Most of the quantities used are simultaneoulsy differential forms and multivectors. The space $\Lambda^k(\mathbb{R}^4, \Lambda^p(\mathbb{R}^m))$ is the collection of all $k$-forms that act as $p$-vectors from $\mathbb{R}^4$ to $\mathbb{R}^m$. Thus an element of this space is both a $p$-vector and a $k$-form, or simply a $p$-vector-valued $k$-form. For $\vec{A}\in\Lambda^k(\mathbb{R}^4, \Lambda^p(\mathbb{R}^m))$ and $\vec{B}\in\Lambda^\ell(\mathbb{R}^4, \Lambda^q(\mathbb{R}^m))$, the product $\vec{A}\overset{\bullet}\wedge_4 \vec{B} \in \Lambda^{k+\ell}(\mathbb{R}^4, \Lambda^{p+q-2}(\mathbb{R}^m))$ represents the contraction between multivectors and wedge product between forms. Similar meaning is given to the products $\vec{A}\overset{\cdot}\wedge_4 \vec{B}$ and $\vec{A}\overset{\wedge}\wedge_4 \vec{B}$ where the upper and lower symbols denote operations between vectors and forms respectively. We will simply write $\vec{A}\wedge \vec{B}$ whenever it is clear to do so. For instance, if $\vec{A}$ is a 1-vector-valued 0-form and $\vec{B}$ is a $q$-vector valued $\ell$-form.

We recall the following useful result due to Henri Poincar\'e (see \cite{jost}).

\noindent
\begin{lem}[Poincar\'e lemma]
Let $\Omega\subset \mathbb{R}^4$ be an open ball. Let $\omega\in \Lambda^k(\Omega, \Lambda^p(\mathbb{R}^m))$ with $1\leq k\leq 4$ satisfy
$$d\omega=0.$$
Then $\omega=d\alpha$ for some $\alpha\in \Lambda^{k-1}(\Omega, \Lambda^p(\mathbb{R}^m))$. Similarly, if $0\leq k\leq 3$ and
$$d^\star \omega=0,$$
then $\omega=d^\star\alpha$ for some $\alpha\in \Lambda^{k+1}(\Omega, \Lambda^p(\mathbb{R}^m))$.
\end{lem}

\begin{prop}\label{olaa}
Let $\vec{\Phi}$ be a critical point of $\mathcal{E}(\Sigma)$ and let $\vec{V}$ be as defined in Proposition \ref{cons-law}. Then there is a two-form $\vec{l}_0\in \Lambda ^2(\mathbb{R}^4, \Lambda^1(\mathbb{R}^m))$   such that 
\begin{align}
\vec{V}=\star d\vec{l}_0.\nonumber
\end{align}
Moreover, there are two-forms $X\in \Lambda ^2(\mathbb{R}^4, \Lambda^0(\mathbb{R}^m))$ and $\vec{Y}\in \Lambda ^2(\mathbb{R}^4, \Lambda^2(\mathbb{R}^m))$ such that

\begin{align}%
\begin{cases}dX&=\vec{l}_0\overset{\cdot}\wedge_4 d\vec{\Phi} -\star d|\vec{H}|^2           \\
d\vec{Y}&=\vec{l}_0\overset{\wedge}\wedge_4 d\vec{\Phi}-\star\vec{M}    \nonumber
\end{cases} 
\end{align}
where $\star$ is the usual Hodge star operator on differential forms.
\end{prop}

\begin{proof}
From Proposition \ref{cons-law}, it holds that $\nabla_j\vec{V}^j=\vec{0}.$ In other words, $d^\star\vec{V}=\vec{0}$ where $\vec{V}\in \Lambda^1(\mathbb{R}^4, \Lambda^1(\mathbb{R}^m))$. By the Poincar\'e lemma, there exists some $\vec{l}_0\in\Lambda^2(\mathbb{R}^4, \Lambda^1(\mathbb{R}^m))$ such that 
\begin{align}
\vec{V}=\star d\vec{l}_0\quad\quad \mbox{with}\quad d^\star\vec{l}_0=\vec{0}. \nonumber
\end{align}
Similarly,  we found in Proposition \ref{cons-law},
$$d^\star(\vec{\Phi}\cdot\vec{V}-d|\vec{H}|^2)=\nabla_j(\vec{\Phi}\cdot\vec{V}-\nabla^j|\vec{H}|^2)=0.$$
By the Poincar\'e lemma, there is some $X_0\in\Lambda^2(\mathbb{R}^4, \Lambda^0(\mathbb{R}^m))$ such that \begin{align}
d^\star X_0&=\vec{\Phi}\cdot\vec{V}-d|\vec{H}|^2=\star d\vec{l}_0\overset{\cdot}\wedge_4\vec{\Phi}-d|\vec{H}|^2.                  \nonumber
\end{align}
We are free to demand that $dX_0=-\star\vec{l}_0\overset{\cdot}\wedge_4 d\vec{\Phi}$. Thus we have
\begin{align}
d\star X_0= -d\vec{l}_0\overset{\cdot}\wedge_4\vec{\Phi}-\star d|\vec{H}|^2=-d(\vec{l}_0\overset{\cdot}\wedge_4\vec{\Phi})+\vec{l}_0\overset{\cdot}\wedge_4 d\vec{\Phi}-\star d|\vec{H}|^2   \nonumber
\end{align}
By setting $X:=\star X_0  +\vec{l}_0 \overset{\cdot}\wedge_4   \vec{\Phi}$ where $X\in\Lambda^2(\mathbb{R}^4, \Lambda^0(\mathbb{R}^m))$, we have that
\begin{align}
dX= \vec{l}_0\overset{\cdot}\wedge_4 d\vec{\Phi}-\star d|\vec{H}|^2\quad\quad\mbox{with}\quad d^\star X=0.  \nonumber
\end{align}

\noindent
Finally, recall that in Proposition \ref{cons-law} we  obtained $d^\star(\vec{\Phi}\wedge \vec{V}+\vec{M})=\vec{0}$. By the Poincar\'e lemma, there exists some $\vec{Y}_0 \in\Lambda^2(\mathbb{R}^4, \Lambda^2(\mathbb{R}^m))$ such that
\begin{align}
d^\star\vec{Y}_0&=\vec{\Phi}\wedge\vec{V} +\vec{M}=\vec{\Phi}\overset{\wedge}\wedge_4 \star d\vec{l}_0+\vec{M}.   \nonumber
\end{align}
We are free to demand that $d\vec{Y}_0=\star\vec{l}_0\overset{\wedge}\wedge_4 d\vec{\Phi}$. Thus we have 
\begin{align}
d\star \vec{Y}_0=-\vec{\Phi}\overset{\wedge}\wedge_4 d\vec{l}_0+\star \vec{M}= -d(\vec{\Phi}\overset{\wedge}\wedge_4\vec{l}_0) +d\vec{\Phi}\overset{\wedge}\wedge_4\vec{l}_0 +\star\vec{M}. \nonumber
\end{align}
By defining a two-vector $\vec{Y}\in\Lambda^2(\mathbb{R}^4, \Lambda^2(\mathbb{R}^m))$ by $\vec{Y}:=-\star\vec{Y}_0-\vec{\Phi}\overset{\wedge}\wedge_4\vec{l}_0,$  we have 
\begin{align}
d\vec{Y}=\vec{l}_0\overset{\wedge}\wedge_4 d\vec{\Phi}-\star\vec{M}\quad\quad\mbox{with}\quad d^\star\vec{Y}=\vec{0}.  \nonumber
\end{align}
\end{proof}
\noindent
For legibility and simplicity of notations, we  now present further conservation laws in codimension one. From now on, $\Sigma$ will be understood as a four dimensional closed hypersurface. 

\begin{prop}\label{coro}
Let $\vec{\Phi}:\Sigma \rightarrow \mathbb{R}^5$ be a critical point of $\mathcal{E}(\Sigma)$ and let $\vec{V}$ be as defined in Theorem \ref{cons-law}.
Then it holds that
\begin{gather}
d^{\star} \vec{V}=\vec{0}, \nonumber
\\ d^\star (\vec{\Phi}\cdot\vec{V}-d|\vec{H}|^2)=0  \nonumber
 \\ \mbox{and}  \quad d^\star(\vec{\Phi}\wedge\vec{V} +(\vec{J}+2d^\star(|\vec{H}|^2d\vec{\Phi}))\wedge d\vec{\Phi})=\vec{0}  \nonumber
\end{gather}
where $\vec{J}:=\frac{1}{2}\Delta_\perp \vec{H}+\frac{1}{2}\vec{H}|\vec{h}|^{2}-7|\vec{H}|^2\vec{H}$.
Also, there exist two-forms $\vec{L}_0,\vec{L}\in\Lambda^2(\mathbb{R}^4, \Lambda^1(\mathbb{R}^5))$, $S_0\in\Lambda^2(\mathbb{R}^4, \Lambda^0(\mathbb{R}^5))$ and $\vec{R}_0\in\Lambda^2(\mathbb{R}^4, \Lambda^2(\mathbb{R}^5))$ such that
\begin{gather}
\vec{V}= d^\star\vec{L}_0 \nonumber\\
dS_0= \vec{L}\overset{\cdot}\wedge_4 d\vec{\Phi} \nonumber\\
d\vec{R}_0=\vec{L}\overset{\wedge}\wedge_4 d\vec{\Phi} +\star d\vec{u} -(\vec{J} +8 |\vec{H}|^2\vec{H})\wedge \star d\vec{\Phi} \nonumber
\end{gather}
where where $\vec{L}_0$ and $\vec{L}$ are related by $\vec{L}=\star(\vec{L}_0-d|\vec{H}|^2\wedge_4 d\vec{\Phi})$ and $\vec{u}$ satisfies the Hodge decomposition
\begin{gather}
\frac{5}{3}|\vec{H}|^2d^\star \vec{\eta}= d\vec{u}+d^\star \vec{v},\nonumber\\
 \quad \vec{u}\in\Lambda^0(\mathbb{R}^4, \Lambda^2(\mathbb{R}^5)), \vec{v}\in\Lambda^2(\mathbb{R}^4, \Lambda^2(\mathbb{R}^5)), \vec{\eta}:=d\vec{\Phi}\overset{\wedge}\wedge_4 d\vec{\Phi}\in\Lambda^2(\mathbb{R}^4, \Lambda^2(\mathbb{R}^5)).  \nonumber
\end{gather}
\end{prop}

\begin{proof}
From Theorem \ref{cons-law}, we have $\nabla_j\vec{V}^j=\vec{0}.$ In other words, $d^\star\vec{V}=\vec{0}$ where $\vec{V}\in \Lambda^1(\mathbb{R}^4, \Lambda^1(\mathbb{R}^5))$. By the Poincar\'e lemma, there exists some $\vec{L}_0\in\Lambda^2(\mathbb{R}^4, \Lambda^1(\mathbb{R}^5))$ such that 
\begin{gather}
\vec{V}=d^\star \vec{L}_0          \label{lqw}
\end{gather}
with the choice
$$d\vec{L}_0=\vec{0}.$$
Also, from Theorem \ref{cons-law}, the invariance of $\mathcal{E}(\Sigma)$ under dilation implies that  
\begin{gather}
d^\star(\vec{\Phi}\cdot\vec{V}-d|\vec{H}|^2)=\nabla_j(\vec{V}^j\cdot\vec{\Phi}-\nabla^j|\vec{H}|^2)=0.        \label{lqwer}
\end{gather}
Note that
\begin{align}
\vec{V}\cdot\vec{\Phi}&= d^\star\vec{L}_0\cdot\vec{\Phi}= \star d\star \vec{L}_0\cdot\vec{\Phi}= \star\left( d(\star\vec{L}_0\cdot\vec{\Phi})-(\star\vec{L}_0)\overset{\cdot}\wedge_4 d\vec{\Phi}  \right)  \nonumber
\\&=d^\star(\vec{L}_0\cdot\vec{\Phi})  -\star((\star \vec{L}_0)\overset{\cdot}\wedge_4 d\vec{\Phi})   \nonumber
\end{align}
so that \eqref{lqwer} becomes
\begin{gather}
d^\star\left[ \star((\star \vec{L}_0)\overset{\cdot}\wedge_4 d\vec{\Phi})  +d|\vec{H}|^2  \right]=0   \nonumber
\end{gather}
or equivalently
\begin{gather}
d\left[(\star \vec{L}_0)\overset{\cdot}\wedge_4 d\vec{\Phi} -\star d|\vec{H}|^2     \right]=0.          \nonumber
\end{gather}

\noindent
Hence, there exists $S_0\in \Lambda^2(\mathbb{R}^4, \Lambda^0(\mathbb{R}^5))$ such that
\begin{gather}
dS_0=(\star \vec{L}_0)\overset{\cdot}\wedge_4 d\vec{\Phi} -\star d|\vec{H}|^2.  \nonumber
\end{gather}

\noindent
Further computations show that
\begin{align}
dS_0&=\frac{1}{6} (\star\vec{L}_0)_{[kl}\cdot \nabla_{m]}\vec{\Phi} \,\,\,dx^{k}\wedge_4 dx^l \wedge_4 dx^m -\frac{1}{6}\epsilon_{iklm}\nabla^i|\vec{H}|^2 \,\,\,dx^{k}\wedge_4 dx^l \wedge_4 dx^m   \nonumber
\\&=\frac{1}{6}\epsilon_{ij[kl}\nabla_{m]}\vec{\Phi}\cdot\left( \frac{1}{2}\vec{L}_0^{ij}-\frac{1}{2}\nabla^{[i}|\vec{H}|^2\nabla^{j]}\vec{\Phi} \right)  \,\,\,dx^{k}\wedge_4 dx^l \wedge_4 dx^m  \nonumber
\\&=\star(\vec{L}_0-d|\vec{H}|^2\wedge_4 d\vec{\Phi})\overset{\cdot}\wedge_4 d\vec{\Phi} .   \nonumber
\end{align}
By setting 
$$\vec{L}:=\star(\vec{L}_0-d|\vec{H}|^2\wedge_4 d\vec{\Phi})$$
we find that
\begin{gather}
dS_0=\vec{L}\overset{\cdot}\wedge_4d\vec{\Phi}.  \nonumber
\end{gather}
Note that $d^\star \vec{L}=\star d\vec{L}_0=\vec{0}.$

\noindent
Now, it holds in codimension one that
\begin{align}
\vec{V}^j\wedge\nabla_j\vec{\Phi}&=\nabla_j(\vec{J}\wedge\nabla^j\vec{\Phi} +2|\vec{H}|^2\vec{h}^{ij}\wedge\nabla_i\vec{\Phi})   \nonumber\\
&= \nabla_j(\vec{J}\wedge \nabla^j\vec{\Phi} +2\nabla^i(|\vec{H}|^2\nabla_i\vec{\Phi})\wedge\nabla^j\vec{\Phi})  \nonumber
\end{align}
so that
\begin{gather}
\nabla_j\left( \vec{V}^j\wedge\vec{\Phi}-(\vec{J}+2\nabla^i(|\vec{H}|^2\nabla_i\vec{\Phi})\wedge\nabla^j\vec{\Phi} )\right)=\vec{0}  \nonumber
\end{gather}
or equivalently
\begin{align}
d^\star\left( \vec{V}\wedge\vec{\Phi}-\vec{J}\wedge d\vec{\Phi} -2d^\star(|\vec{H}|^2d\vec{\Phi})\wedge d\vec{\Phi}  \right)=\vec{0} . \label{becc}
\end{align}
Observe that
\begin{align}
\vec{V}\wedge \vec{\Phi}&= d^\star\vec{L}_0\wedge \vec{\Phi}=\star d\star \vec{L}_0\wedge \vec{\Phi}    \nonumber
=       \star\left[d(\star \vec{L}_0\wedge \vec{\Phi})-(\star\vec{L}_0)\overset{\wedge}\wedge_4 d\vec{\Phi}  \right]    \nonumber
\\&=d^\star(\vec{L}_0\wedge\vec{\Phi})-\star\left[(\star\vec{L}_0)\overset{\wedge}\wedge_4 d\vec{\Phi}\right]    \nonumber
\\&=d^\star(\vec{L}_0\wedge\vec{\Phi})- \star(\vec{L}\overset{\wedge}\wedge_4 d\vec{\Phi})  -\star\left[\left(\star(d|\vec{H}|^2\wedge_4 d\vec{\Phi})  \right) \overset{\wedge}\wedge_4 d\vec{\Phi} \right]  .    \label{gracee}
\end{align}

\noindent
We will choose $\vec{\eta}\in\Lambda^2(\mathbb{R}^4, \Lambda^2(\mathbb{R}^5))$ as follows
\begin{align}
\vec{\eta}=\frac{1}{2}\vec{\eta}_{ij} \,\,\,dx^{i}\wedge_4 dx^j:= \frac{1}{2} \nabla_i\vec{\Phi}\wedge\nabla_j\vec{\Phi}\,\,\, dx^{i}\wedge_4 dx^j.
\end{align}

\noindent
Focusing on \eqref{gracee}, we see that
\begin{align}
\star\left[\left(\star(d|\vec{H}|^2\wedge_4 d\vec{\Phi})  \right) \overset{\wedge}\wedge_4 d\vec{\Phi} \right]  &= \frac{1}{6}\epsilon^{klmr}\left( \frac{1}{6}\delta^{\alpha\beta\gamma}_{klm}\epsilon_{ij\alpha\beta}\frac{1}{2} \nabla^{[i}|\vec{H}|^2\nabla^{j]}\vec{\Phi} \right)\wedge \nabla_\gamma\vec{\Phi} \,\,g_{rs}\,\,\, dx^s  \nonumber
\\&= \frac{1}{6}\delta_{ij}^{mr} \nabla^{[i}|\vec{H}|^2\nabla^{j]}\vec{\Phi}\wedge \nabla_m\vec{\Phi}  g_{rs} dx^s     \nonumber
  \\&=-\frac{1}{3} \nabla_m|\vec{H}|^2 \nabla^m\vec{\Phi}\wedge\nabla_r\vec{\Phi} dx^r  \nonumber
\\&=-\frac{1}{3}\nabla_m(|\vec{H}|^2\nabla^m\vec{\Phi}\wedge \nabla_r\vec{\Phi}) dx^r +\frac{1}{3} |\vec{H}|^2 \nabla_m(\nabla^m\vec{\Phi}\wedge\nabla_r\vec{\Phi}) dx^r  \nonumber
\\&=-\frac{1}{3} d^\star(|\vec{H}|^2\vec{\eta})+\frac{1}{3}|\vec{H}|^2 d^\star\vec{\eta}  .  \nonumber
\end{align}
Thus, it follows that
\begin{align}
\vec{V}\wedge\vec{\Phi}=  d^\star(\vec{L}_0\wedge\vec{\Phi})- \star(\vec{L}\overset{\wedge}\wedge_4 d\vec{\Phi})     +\frac{1}{3} d^\star(|\vec{H}|^2\vec{\eta})-\frac{1}{3}|\vec{H}|^2 d^\star\vec{\eta}   \nonumber
\end{align}
and so \eqref{becc} becomes
\begin{align}
d^\star\left[-\star(\vec{L}\overset{\wedge}\wedge_4 d\vec{\Phi})  -\frac{1}{3}|\vec{H}|^2 d^\star\vec{\eta} -\vec{J}\wedge d\vec{\Phi} -2d^\star(|\vec{H}|^2d\vec{\Phi})\wedge d\vec{\Phi}      \right]=\vec{0}.   \label{sct}
\end{align}

\noindent
Observe that
\begin{align}
d^\star (|\vec{H}|^2 d\vec{\Phi})\wedge d\vec{\Phi}&= \nabla_j(|\vec{H}|^2\nabla^j\vec{\Phi})\wedge\nabla_i\vec{\Phi}  dx^i       \nonumber
\\&= 4|\vec{H}|^2\vec{H}\wedge\nabla_i\vec{\Phi} dx^i +\nabla_j|\vec{H}|^2\nabla^j\vec{\Phi}\wedge \nabla_i\vec{\Phi}  dx^i  \nonumber
\\&= 4|\vec{H}|^2\vec{H}\wedge\nabla_i\vec{\Phi} dx^i  +\nabla^j(|\vec{H}|^2\nabla_j\vec{\Phi}\wedge\nabla_i\vec{\Phi}) dx^i \nonumber
\\&\quad\quad-|\vec{H}|^2\nabla^j(\nabla_j\vec{\Phi}\wedge\nabla_i\vec{\Phi}) dx^i   \nonumber
\\&= 4|\vec{H}|^2\vec{H}\wedge d\vec{\Phi} +d^\star(|\vec{H}|^2\vec{\eta})  -|\vec{H}|^2d^\star\vec{\eta}   \label{scv}
\end{align}

\noindent
Now, substitute \eqref{scv} into \eqref{sct} to find
\begin{align}
d^\star\left[-\star(\vec{L}\overset{\wedge}\wedge_4 d\vec{\Phi})  +\frac{5}{3}|\vec{H}|^2 d^\star\vec{\eta} -(\vec{J} +8 |\vec{H}|^2\vec{H})\wedge d\vec{\Phi}     \right]=\vec{0}  .\nonumber
\end{align}

\noindent
We introduce the Hodge decomposition 
\begin{align}
\frac{5}{3} |\vec{H}|^2d^\star\vec{\eta}= d\vec{u} +d^\star \vec{v}  \nonumber
\end{align}
so that 
\begin{align}
d^\star\left[ -\star(\vec{L}\overset{\wedge}\wedge_4 d\vec{\Phi})  +d\vec{u}  -(\vec{J} +8 |\vec{H}|^2\vec{H})\wedge d\vec{\Phi}    \right] =\vec{0}   \nonumber
\end{align}
or equivalently
\begin{align}
d\left[\vec{L}\overset{\wedge}\wedge_4 d\vec{\Phi} +\star d\vec{u} -(\vec{J} +8 |\vec{H}|^2\vec{H})\wedge \star d\vec{\Phi}      \right]=\vec{0}.
\end{align}
Therefore there exists $\vec{R}_0\in\Lambda^2(\mathbb{R}^4, \Lambda^2(\mathbb{R}^5))$ such that
\begin{align}
d\vec{R}_0=\vec{L}\overset{\wedge}\wedge_4 d\vec{\Phi} +\star d\vec{u} -(\vec{J} +8 |\vec{H}|^2\vec{H})\wedge \star d\vec{\Phi}  .   \nonumber
\end{align}

\noindent
This completes the proof.

\end{proof}

\begin{rmk}
We have proved in Theorem \ref{coro} the existence of 2-forms $\vec{L}$, $S_0$ and $\vec{R}_0$. Observe that $S_0$ and $\vec{R}_0$ are defined in terms of $\vec{L}$ which depends on some geometric quantity $\vec{V}$. The following computations show that $S_0$ and $\vec{R}_0$ can be directly linked back to some geometric quantities via the ``return equation".
\end{rmk}

\textbf{Return equation}
\begin{align}
&\quad\quad\epsilon^{mijk}\nabla_i(\vec{R}_0)_{jk}\bullet \nabla_m\vec{\Phi}  \nonumber
\\&= \epsilon^{mijk}\left[   \vec{L}_{jk}\wedge \nabla_i\vec{\Phi}+\epsilon_{rijk}\nabla^r\vec{u}-\epsilon_{rijk}(\vec{J}+8|\vec{H}|^2\vec{H})\wedge\nabla^r\vec{\Phi}                  \right]\bullet \nabla_m\vec{\Phi}    \nonumber
\\&=  \epsilon^{mijk}(\vec{L}_{jk}\cdot\nabla_m\vec{\Phi}) \nabla_i\vec{\Phi} +6\delta^m_r \nabla^r\vec{u}\bullet \nabla_m\vec{\Phi}-6\delta^m_r(\vec{J}+8|\vec{H}|^2\vec{H})(-\delta^r_m)         \nonumber
\\&= \epsilon^{mijk}\nabla_m (S_0)_{jk}\nabla_i\vec{\Phi}+6\nabla_i\vec{u}\bullet\nabla^i\vec{\Phi}+24(\vec{J}+8|\vec{H}|^2\vec{H}).  \nonumber
\end{align}

We have found
\begin{align}
\epsilon^{mijk}(\nabla_i(\vec{R}_0)_{jk}\bullet \nabla_m\vec{\Phi}+ \nabla_i (S_0)_{jk}\nabla_m\vec{\Phi})-6\nabla_i\vec{u}\bullet\nabla^i\vec{\Phi}&=  24(\vec{J}+8|\vec{H}|^2\vec{H}) \label{stella}
\end{align}

\noindent
This last equation \eqref{stella} can be further simplified in order to re-introduce $v$.

Observe that
\begin{align}
\nabla_i\vec{u}\bullet \nabla^i\vec{\Phi}&= \left[\frac{5}{3} |\vec{H}|^2\nabla_j(\nabla^j\vec{\Phi}\wedge \nabla_i\vec{\Phi}) -\nabla^j\vec{v}_{ji}    \right]\bullet\nabla^i\vec{\Phi}         \nonumber
\\&= \frac{5}{3}|\vec{H}|^2\left( 4\vec{H}\wedge\nabla_i\vec{\Phi}+\nabla^j\vec{\Phi}\wedge\vec{h}_{ji}  \right)\bullet\nabla^i\vec{\Phi}    -\nabla_j\vec{v}^{ji}\bullet\nabla_i\vec{\Phi}   \nonumber
\\&= \frac{5}{3}|\vec{H}|^2\left( -16\vec{H}+4\vec{H} \right)   -\nabla_j\vec{v}^{ji}\bullet\nabla_i\vec{\Phi}   \nonumber 
\\&= -20|\vec{H}|^2\vec{H} -\nabla_j\vec{v}^{ji}\bullet\nabla_i\vec{\Phi} . \nonumber
\end{align}
Thus we find that
\begin{align}
\epsilon^{mijk}\nabla_i(\vec{R}_0)_{jk}\bullet \nabla_m\vec{\Phi}+ \epsilon^{mijk} \nabla_i (S_0)_{jk} \nabla_m\vec{\Phi}   +6\nabla_j\vec{v}^{ji}\bullet\nabla_i\vec{\Phi} = 24(\vec{J}+8|\vec{H}|^2\vec{H}) -120|\vec{H}|^2\vec{H}.   \nonumber
\end{align}
Clearly, 
\begin{align}
24(\vec{J}+8|\vec{H}|^2\vec{H}) -120|\vec{H}|^2\vec{H}&=24\left( \frac{1}{2}\Delta_{\perp}\vec{H}+\frac{1}{2}(\vec{H}\cdot\vec{h}^{ij})\vec{h}_{ij} +|\vec{H}|^2\vec{H} \right)-120|\vec{H}|^2\vec{H}  \nonumber
\\& =12\Delta_{\perp}\vec{H}+12(\vec{H}\cdot\vec{h}^{ij})\vec{h}_{ij} -96|\vec{H}|^2\vec{H}   \nonumber
\\&= 12\nabla_j\left[\nabla^j\vec{H}-2|\vec{H}|^2\nabla^j\vec{\Phi} +2(\vec{H}\cdot\vec{h}^{jk})\nabla_k\vec{\Phi}  \right]   \label{zea}
\end{align}
where we have used Lemma \ref{rivvy} to write \eqref{zea}.
Hence, we have found
\begin{gather}
\epsilon^{mijk}\nabla_i(\vec{R}_0)_{jk}\bullet \nabla_m\vec{\Phi}+ \epsilon^{mijk} \nabla_i (S_0)_{jk} \nabla_m\vec{\Phi}   +6\nabla_j\vec{v}^{ji}\bullet\nabla_i\vec{\Phi}   \nonumber
\\=12\nabla_j\left[\nabla^j\vec{H}-2|\vec{H}|^2\nabla^j\vec{\Phi} +2(\vec{H}\cdot\vec{h}^{jk})\nabla_k\vec{\Phi}  \right] .   \label{recca}
\end{gather}

\noindent
Lastly, applying divergence to the Hodge decomposition
$$\nabla^j\vec{u}+\nabla_i\vec{v}^{ij}=\frac{5}{3}|\vec{H}|^2\nabla_i\vec{\eta}^{ij}$$
yields
\begin{align}\Delta_g\vec{u}=\frac{5}{3}\nabla_j|\vec{H}|^2\nabla_i\vec{\eta}^{ij}. \label{hads8}\end{align}

\noindent
Also, $\vec{v}$ satisfies the equation
$$\Delta\vec{v}=\frac{5}{3} d\left( |\vec{H}|^2d^\star\vec{\eta}  \right).$$

\chapter{Regularity Results} \label{chap4}
We prove that the critical point of the 4-dimensional Willmore-type energy \eqref{sjda} is smooth. In particular, we have
\begin{theo}\label{heart}
Let $\vec{\Phi}\in W^{3,2}\cap W^{1,\infty}$ be a non-degenerate critical point of the energy
\begin{gather}
\mathcal{E}(\Sigma):= \int_\Sigma \left(|\pi_{\vec{n}}\nabla \vec{H}|^2-|\vec{H}\cdot\vec{h}|^2 +7|\vec{H}|^2    \right) d\textnormal{vol}_g.  \label{sjda}
\end{gather}
Then $\vec{\Phi}$ is smooth with the estimate
\begin{align}
\norm{DH}_{L^\infty(B_r)} +\frac{1}{r} \norm{H}_{L^\infty(B_r)}  \lesssim \frac{1}{r^2} \left(\norm{DH}_{L^2(B_1)} +\norm{H}_{L^4(B_1)}   \right) \quad\quad\forall \, r<1   \nonumber
\end{align}
where $B_r\subset\mathbb{R}^4$ is any ball of radius $r$.
\end{theo}

\section{Introduction and outline}


\noindent
The proof of Theorem \ref{heart} is divided into different sections of this Chapter. In Section \ref{adzsq}, we develop generic regularity results applicable to our problem. 
The goal of Section \ref{amena} is to understand the regularity of the quantities $\vec{V}, \vec{L}_0, \vec{L}, S_0, \vec{R}_0$ and $\vec{u}$ obtained in Chapter \ref{latty1}. A major achievement of this section is the existence of an invertible operator $\vec{\mathcal{P}}$ used in creating new two forms: $\vec{L}_1, S$ and $\vec{R}$ which are estimated in Section \ref{toute}. Section \ref{linkage} answers the important question of how to link the regularity of the primitives $S$ and $\vec{R}$ back to geometric quantities like $\vec{H}$ or $\vec{\Phi}$. This is done via the ``return equation" obtained in Chapter \ref{latty1}. By obtaining a Morrey-type estimate, the integrability of $d\vec{H}$ is improved, the bootstrap argument ensues and the proof is finished.

\noindent
It is important to clarify some of the notations used in subsequent sections. The operators $d$ and $d^\star$ are the well known exterior differential and codifferential, respectively, of differential forms. They are defined in terms of the metric induced Levi-Civita connection $\nabla$.\,\,\,The operator $D$ involves the flat partial derivatives. In the sequel, the sign $\lesssim$ will indicate the presence of an unimportant multiplicative constant depending on $||\vec{\Phi}||_{W^{3,2}\cap W^{1,\infty}}$.

\section{Preliminary results} \label{adzsq}
\begin{df}
Let $\Omega\subset \mathbb{R}^4$ be a ball. The Sobolev space $W^{k,p}(\Omega)$  of measurable functions from $\Omega$ into $\mathbb{R}^m$ is defined by
$$W^{k,p}(\Omega):=\left\{ f\in L^p(\Omega) \,:\, D^\alpha f\in L^p(\Omega)\,\, \forall\,\,\alpha \,\,\mbox{with} \,\,|\alpha|\leq k          \right\}   $$
with the norm
$$\norm{f}_{W^{k,p}(\Omega)}:= \sum_{|\alpha|\leq k}\norm{D^\alpha f}_{L^p(\Omega)}.$$
The dual of $W^{k,p}(\Omega)$ is $W^{-k,p}(\Omega)$. 

\noindent
The homogeneous Sobolev space $\dot W^{k,p}(\Omega)$ is the space of all $k$-weakly differentiable functions $f\in L^p(\Omega)$ such that $D^\alpha f\in L^p(\Omega)$ for all $|\alpha|=k$.

\noindent
The closure of $C_c^\infty(\Omega)$ in $W^{k,p}(\Omega)$ is denoted by $W_0^{k,p}(\Omega)$. 
\end{df}
\noindent
Suppose that $\vec{\Phi}\in W^{3,2}(\Omega)\cap W^{1,\infty}(\Omega)$ is a critical point of the energy $\mathcal{E}(\Sigma).$

\noindent
By hypothesis, the coefficients of the metric tensor given by $g_{ij}:=\nabla_i\vec{\Phi}\cdot\nabla_j\vec{\Phi}$ lie in the space $W^{2,2}\cap L^{\infty}$. We are interested in the regularity of Willmore-type manifolds with no branch point, that is the metric does not degenerate. The metric coefficients are uniformly bounded from above and below and satisfy, in a local coordinate chart $\Omega$ of $\Sigma$:
\begin{gather}
g_{ij} \simeq \delta_{ij} \quad\quad \textnormal{as (4 $\times$ 4)-matrices.}  \nonumber
\end{gather}
In other words, $g_{ij}$ satisfies on $\Omega$
\begin{gather}
c^{-1} \delta_{ij} u^i v^j\leq g_{ij} u^i v^j \leq c\delta_{ij} u^i v^j     \nonumber
\end{gather}
for some $c>0$ and for all $u,v \in\mathbb{R}^4$. We have used $\delta_{ij}$ to denote the $(i,j)-$component of the flat metric.

\noindent
For $A\in \Lambda^k(\Omega)$ and $B\in \Lambda^k(\Omega)$ with local representations
$$A=\frac{1}{k!}A_{[i_1....i_k]} dx^{i_1}\wedge_4\cdots\wedge_4 dx^{i_k}\quad\mbox{and}\quad B=\frac{1}{k!}B_{[j_1....j_k]} dx^{j_1}\wedge_4\cdots\wedge_4 dx^{j_k} $$
we define the usual inner product on $\Omega\in\mathbb{R}^4 $
\begin{gather}
\langle A, B\rangle :=\frac{1}{(2k)!}\int_\Omega g^{i_1 j_1} \cdots g^{i_kj_k} A_{[i_1...i_k} B_{j_1...j_k]} dx^{1}\wedge_4 \cdots \wedge_4 dx^4.   \nonumber
\end{gather}
Since the metric is controlled by the Euclidean metric on $\Omega$, we have
\begin{gather}
\langle A, B\rangle \simeq  \langle A, B \rangle_0:=  \frac{1}{(2k)!}\int_\Omega \delta^{i_1 j_1} \cdots \delta^{i_kj_k} A_{[i_1...i_k} B_{j_1...j_k]} dx^{1}\wedge_4 \cdots \wedge_4 dx^4  \nonumber
\end{gather}
which is the usual inner-product on forms for the flat metric $\delta$. Accordingly, we have
\begin{align}
\norm{A}_{L^p}\equiv \sup_{\norm{B}_{L^{p'}}=1}\langle A, B\rangle _0 = \sup_{\norm{B}_{L^{p'}}=1}\langle A, B\rangle  \quad\quad\forall \,\,p\in[1,\infty)\,,\,\, \frac{1}{p} +\frac{1}{p'}=1.   \nonumber
\end{align}
\noindent
This fact will be used recurrently and tacitly in subsequent sections, in  particular, when calling upon duality arguments.  
\\
First, we discuss the regularising properties of an equation of the type

\begin{align} \begin{cases} 
      \Delta u=f & \mbox{in}\,\, \Omega \\
      \,\,\,\,\,u=0 & \mbox{on} \,\, \partial \Omega.
   \end{cases} \label{take}
\end{align}

\noindent
Equivalently, we can write 
\begin{align}
\mathcal{L}[u]:= \partial_i(|g|^{1/2} g^{ij} \partial_j u)= |g|^{1/2} f  \nonumber
\end{align}
where $\mathcal{L}$ is a second-order partial differential operator. The coefficient $a^{ij}:\Omega\rightarrow \mathbb{R}$ defined by $a^{ij}:= |g|^{1/2} g^{ij}$ is clearly symmetric and 
satisfies the uniform ellipticity condition: there exists some $c > 0$ such that
\begin{align}
c^{-1}\delta^{ij} u_i v_j\leq a^{ij}u_iv_j \leq c \delta^{ij}u_iv_j\,, \,\,\forall\,\, u,v\in\mathbb{R}^4.\nonumber
\end{align}
\\
\noindent
Unfortunately, the coefficients $a^{ij}$ are not H\"older continuous and standard regularity theory tools for elliptic equations cannot be applied but thanks to  Claim \ref{cll}  and Proposition \ref{callit}, we are still able to obtain the desired results.

\begin{cl}\label{cll}
The coefficients $a^{ij}$ belong to the space $W^{2,2}\cap L^\infty$. 
\end{cl}
\begin{proof}
First, we show that $W^{2,2}\cap L^\infty$ is an algebra under pointwise multiplication. Let $\alpha,\beta \in W^{2,2}\cap L^\infty$. Clearly, $\alpha \beta \in L^\infty$. It remains to show that $\alpha\beta \in W^{2,2}$. By Liebnitz rule we write
\begin{align}
 D ^2(\alpha \beta)=  \alpha D ^2 \beta+ \beta D ^2 \alpha+ 2 D  \alpha D  \beta . \nonumber
\end{align}
Now, $\alpha\in L^\infty,  D ^2 \beta\in L^2$ implies that $\alpha D ^2 \beta\in L^2$. By the same token $\beta D ^2\alpha\in L^2$. By Sobolev embedding, the functions $ D  \alpha, D  \beta\in W^{1,2}\subset L^4$ so that the product $ D  \alpha D  \beta\in L^2$. We arrive at $\alpha\beta\in W^{2,2}$. Thus $\alpha\beta\in W^{2,2}\cap L^\infty$.

\noindent
Next, using the hypothesis $\vec{\Phi}\in W^{3,2}\cap W^{1,\infty}$ and the fact that $W^{2,2}\cap L^\infty$ is an algebra under pointwise multiplication, we have
$g^{ij}:=\nabla^i\vec{\Phi}\cdot \nabla^j\vec{\Phi}\in W^{2,2}\cap L^\infty$. Hence $a^{ij}\in W^{2,2}\cap L^\infty$.
\end{proof}
\noindent
By the Sobolev embedding theorem, $W^{2,2}\subset W^{1,4}$. Thus $a^{ij}$ belong to $L^\infty\cap W^{1,4}$.
It is important to note the role of the fact that $g^{ij}\in L^\infty$ in the proof of Claim \ref{cll} (as $W^{2,2}$ is not an algebra under pointwise multiplication). Otherwise, it would be difficult to now arrive at $a^{ij}\in W^{1,4}$. This is much needed because uniformly elliptic operators with coefficients in $W^{1,4}$ have fairly good dispositions towards regularity theory.

\begin{rmk}
Although it is true that the product of a  function in $L^p$ and any metric component or its determinant remains in $L^p$, the same cannot be said of Sobolev spaces (or of their duals). This is because the metric coefficients are not sufficiently differentiable (not Lipschitz). But the coefficients belong to $L^\infty\cap W^{2,2}$ by Claim \ref{cll} and the following holds.
\end{rmk}

\begin{cl} \label{claimclaim}
Let $\alpha$ be a function in $(L^\infty \cap W^{2,2})$. Let $\beta$ be a function in  $\dot{W}^{-1,2}$ Then the product $\alpha\beta$ belongs to $  \dot{W}^{-1,2} \oplus L^{4/3} \subset \dot{W}^{-1,2}$.
\end{cl}
\begin{proof}
Since  $\beta \in \dot{W}^{-1,2}$, there exists $\gamma\in L^2$ such that $ D  \gamma=\beta$. We have
\begin{align}
\alpha\beta=  D (\alpha\gamma)-\gamma  D  \alpha.  \nonumber
\end{align}
Now, $\alpha\in L^\infty$ and $\gamma\in L^2$ implies that $ D (\alpha\gamma)\in \dot{W}^{-1,2}$. By the Sobolev embedding theorem, $\alpha\in W^{2,2}\subset W^{1,4}$ so that $ D  \alpha\in L^4$ and the product $\gamma D \alpha \in L^{4/3}$. By the dual of the Sobolev embedding theorem, $\alpha\beta\in  \dot{W}^{-1,2}\oplus L^{4/3}  \subset \dot{W}^{-1,2}.$
\end{proof}

\noindent
As a consequence, we have  the norm
$$\norm{A}_{W_0^{1,2}}\equiv \sup_{\norm{B}_{\dot W^{-1,2}}=1} \langle A, B\rangle _0 \simeq \sup_{\norm{B}_{\dot W^{-1,2}}=1} \langle A, B\rangle .$$

\begin{prop}\label{callit}
Let $\Omega\subset\mathbb{R}^4$ be a ball. Suppose that  $f:\Omega\rightarrow \mathbb{R}$  with $f\in L^2(\Omega)$.  Let $u$ be a weak solution of the problem 
\[ \begin{cases} 
      \Delta u=f & \mbox{in}\,\, \Omega \\
      \,\,\,\,\,u=0 & \mbox{on} \,\, \partial \Omega.
   \end{cases}
\]
Then $D^2 u\in L^2$ and 
$$||D^2 u||_{L^2(\Omega)}\lesssim ||f||_{L^2(\Omega)}.$$
\end{prop}

\begin{proof}
Clearly $\mathcal{L}[u]=|g|^{1/2}f$ belongs to $L^2$ since $f\in L^2(\Omega)$. A classical result of Miranda \cite{mir} (see also Corollary 1.4 in \cite{cruz}) then yields the desired estimate
$$\norm{D^2 u}_{L^2(\Omega)}\lesssim \norm{|g|^{1/2} f}_{L^2(\Omega)}\lesssim \norm{f}_{L^2(\Omega)}.$$
\end{proof}

\noindent
Suppose that  $u\in \Lambda^k(\mathbb{R}^4)$  and $f=d^\star B$, where $B\in \Lambda^{k+1}(\mathbb{R}^4)$, then    \eqref{take}  has the special structure   
$$\Delta u=d^\star B$$
which is equivalent to
\begin{align}
\mathcal{L}[u_{i_1\cdots i_k}] \,\,dx^{i_1}\wedge_4\cdots\wedge_4 dx^{i_k}&=|g|^{1/2} d^\star B  \nonumber
\\&= \partial_a(|g|^{1/2} g^{ai_{k+1}}B_{i_1\cdots i_{k+1}})\,\,dx^{i_1}\wedge_4\cdots\wedge_4 dx^{i_k}.  \nonumber
\end{align}
\noindent
This is a decoupled system of $k$-uniformly elliptic equations in flat divergence form with coefficients in $W^{2,2}\cap L^\infty$ which embeds in $VMO\cap L^\infty$. Such equations are studied by Di Fazio in  \cite{faz} (see also \cite{auscher}).


\begin{prop}  \label{adz}
Let $\Omega\subset\mathbb{R}^4$ be a ball. Suppose that $B\in \Lambda^{k+1}\otimes L^p(\Omega)$ for $p\in(1,\infty)$. 
Let $u\in \Lambda^k(\Omega)$ be the $k$-form satisfying
\[ \begin{cases} 
      \Delta u=d^\star B& \mbox{in}\,\, \Omega \\
      \,\,\,\,\,u=0 & \mbox{on} \,\, \partial \Omega.
   \end{cases}
\]
Then 
$$\norm{Du}_{L^p(\Omega)}\lesssim \norm{B}_{L^p(\Omega)}.$$
\end{prop}

\begin{proof}
Using Theorem 2.1 in \cite{faz}, each component of $u$ satisfies
$$\norm{Du_{i_1\cdots i_k}}_{L^p(\Omega)}\lesssim \norm{|g|^{1/2} g^{ai_{k+1}}B_{i_1\cdots i_{k+1}}}_{L^p(\Omega)}  \lesssim \norm{B}_{L^p(\Omega)}.$$
\end{proof}

\noindent
We will require the following corollaries of Theorem 2.1 in \cite{faz}.
\begin{prop} \label{111}
Let $\Omega\in\mathbb{R}^4$ be a ball. Suppose that $u\in\Lambda^2(\Omega)$ satisfies the conditions
$$d^\star u=0\quad\quad\mbox{and}\quad\quad du\in \dot{W}^{-1,2}(\Omega).$$
Then
$$\norm{u}_{L^2(\Omega)}\lesssim \norm{du}_{\dot{W}^{-1,2}(\Omega)}.$$
\end{prop}

\begin{proof}
Let $\Psi \in \Lambda^2\otimes L^2(\Omega)$ be arbitrary. Standard Hodge theory guarantees  there exist $\psi\in \Lambda^3\otimes W^{1,2}(\Omega)$ and $\theta\in\Lambda^1\otimes W^{1,2}(\Omega)$ such that
$$\Psi =d\theta +d^\star \psi  \quad\quad\mbox{with}\quad\quad d \psi=0\quad\mbox{and}\quad d^\star\theta=0.$$ 
We have, per Claim \ref{claimclaim},
$$\mathcal{L}[\psi]=|g|^{1/2}\Delta\psi=|g|^{1/2} d\Psi \,\,\in \,\dot{W}^{-1,2}(\Omega).       $$
Reasoning componentwise (set $\psi|_{\partial \Omega}=0$), we use \cite{faz} to arrive at
\begin{align}
\norm{D\psi}_{L^2(\Omega)}\lesssim \norm{|g|^{1/2} d\Psi}_{\dot{W}^{-1,2}(\Omega)}\lesssim \norm{\Psi}_{L^2(\Omega)}.  \label{cklh}
\end{align}
Observe that
\begin{align}
\langle u, \Psi\rangle &=\langle u, d\theta +d^\star \psi\rangle     \nonumber
\\&= \langle du, \psi\rangle     \nonumber
\\&\lesssim \norm{du}_{\dot{W}^{-1,2}(\Omega)} \norm {\psi}_{W_0^{1,2}(\Omega)}   \nonumber
\\ &\overset{\eqref{cklh}}\lesssim \norm{du}_{\dot{W}^{-1,2}(\Omega)} \norm {\Psi}_{L^2(\Omega)}   .  \nonumber
\end{align}
Hence the result
$$\norm{u}_{L^2(\Omega)}\lesssim \norm{du}_{\dot{W}^{-1,2}(\Omega)}.$$

\end{proof}

\begin{prop} \label{nadaa}
Let $\Omega \subset \mathbb{R}^4$ be a ball. Any $Q\in\Lambda^3(\Omega, \Lambda^2(\mathbb{R}^5))\otimes \dot{W}^{-1,2}(\Omega)$ can be decomposed in the form
$$Q=dA+d^\star B \quad\quad\mbox{with}\quad\quad d^\star A=0\,,\,dB=0\,\,,\,\, A\in\Lambda^2\,\,,\,\, B\in\Lambda^4\,$$
and
$$\norm{A}_{L^2(\Omega)}+\norm{B}_{L^2(\Omega)}\lesssim \norm{Q}_{\dot{W}^{-1,2}(\Omega)}.$$
\end{prop}
\begin{proof}

Let $\Psi\in L^2(\Omega)$ be arbitrary. As in the proof of Proposition \ref{111}, we have
$$\Psi=d\theta+d^\star \psi\quad\mbox{with} \quad d^\star\theta=0\,,\,d\psi=0,$$
and
\begin{align}\norm{D\psi}_{L^2(\Omega)}\lesssim \norm{|g|^{1/2} d\Psi}_{\dot{W}^{-1,2}(\Omega)}\lesssim \norm{\Psi}_{L^2(\Omega)}.  \label{neededd}
\end{align}
Since $\star B$ is a 0-form and $d\psi=0$, we have  
\begin{align}
\langle \star B,\Psi \rangle&=\langle \star B, d\theta+d^\star\psi \rangle  \nonumber
\\&=\langle \star d^\star B, \psi \rangle \nonumber
\\&= \langle \star Q- d^\star(\star A), \psi \rangle  \nonumber
\\&= \langle \star Q, \psi \rangle   \nonumber
\\&\lesssim \norm{\star Q}_{\dot{W}^{-1,2}(\Omega)} \norm{\psi}_{W_0^{1,2}(\Omega)}   \nonumber
\\&\overset{\eqref{neededd}} \lesssim \norm{Q}_{\dot{W}^{-1,2}(\Omega)} \norm{\Psi}_{L^2(\Omega)} \nonumber
\end{align}
where Claim \ref{claimclaim} has been used to obtain the last line. Hence
\begin{align}\norm{B}_{L^2(\Omega)}\simeq \norm{\star B}_{L^2(\Omega)}\lesssim \norm{Q}_{\dot{W}^{-1,2}(\Omega)}.\label{jknba}\end{align}
 Similarly, we have
\begin{align}
\langle A, \Psi \rangle&=   \langle A, d\theta+ d^\star \psi \rangle   \nonumber
\\&= \langle dA, \psi \rangle  \nonumber
\\&= \langle Q-d^\star B, \psi \rangle  \nonumber
\\&= \langle Q, \psi \rangle  \nonumber
\\&\lesssim \norm{Q}_{\dot{W}^{-1,2}(\Omega)} \norm{\psi}_{W_0^{1,2}}  \nonumber
\\& \lesssim \norm{Q}_{\dot{W}^{-1,2}(\Omega)} \norm{\Psi}_{L^2(\Omega)}. \nonumber  
\end{align}
Hence 
\begin{gather}
\norm{A}_{L^2(\Omega)}\lesssim \norm{Q}_{\dot{W}^{-1,2}(\Omega)}.\label{jknb} 
\end{gather}
Combining the estimates \eqref{jknba} and \eqref{jknb} proves the result.
\end{proof}

\noindent
We will state, without proof, the following Bourgain-Brezis type result whose non-trivial proof is amply outlined in Chapter 6 of \cite{curca}. E.\ Curca was kind enough to communicate with us the proof of Proposition \ref{curca} which was almost done in his thesis (see Theorem 6.4 in \cite{curca} and the discussion following Remark 6.5 on page 129 of \cite{curca}). The original  predecessor of this result is due to Maz'ya \cite{maz}.

\begin{prop} \label{curca}
Let $\Omega \subset \mathbb{R}^4$ be a ball and let $V\in\Lambda^1\otimes (L^1\oplus \dot{W}^{-2,2})$ with $d^\star V=0.$ There exists $w\in\Lambda^2\otimes \dot{W}^{-1,2}(\Omega)$ such that
$$d^\star w=V\quad,\quad dw=0\quad,\quad \norm{w}_{\dot{W}^{-1,2}(\Omega)}\lesssim \norm{V}_{L^1\oplus \dot{W}^{-2,2}(\Omega)}.$$
\end{prop}


\noindent
Finally, we will need the following technical lemma.

\begin{lem}  \label{tech}
Let $\Omega$ be a ball in $\mathbb{R}^4$. For $k\in (0,1)$, we denote by $\Omega_k$, the ball with same center as $\Omega$ and with radius rescaled by $k$. Suppose $T\in \Lambda^1\otimes L^2(\Omega)$ satisfies
$$\norm{dT}_{L^{4/3}(\Omega)}+\norm{d^\star T}_{L^{4/3}(\Omega) }\lesssim M,$$
for some constant $M>0$. Then
$$\norm{T}_{L^2(\Omega_k)}\lesssim M+|\Omega|^{1/4}\norm{T}_{L^2(\Omega)}.$$ 
\end{lem}

\begin{proof}
Let $\mu$ be a standard smooth cut-off function with
$$\mu|_{\Omega_k}=1\quad,\quad \mu|_{\mathbb{R}^4\setminus \Omega}=0\quad, \quad 0<\mu(x) <1\,\,\, \forall\,\, x   \in\mathbb{R}^4.         $$
Clearly, 
\begin{align}
\norm{d(\mu T)}_{L^{4/3}(\Omega)} &\lesssim \norm{d\mu}_{L^\infty (\Omega)}  \norm{T}_{L^{4/3}(\Omega)} +\norm{\mu}_{L^\infty (\Omega)} \norm{dT}_{L^{4/3}(\Omega)}  \nonumber
\\&\lesssim |\Omega|^{1/4}\norm{T}_{L^2(\Omega)} +M\,,          \label{1b1}
\end{align}
where we have used Jensen's inequality. By the same token, using
$$d^\star (\mu T)=\star(d\mu\wedge_4 \star T)+\mu d^\star T\,,$$
we find
\begin{align}
\norm{d^\star (\mu T)}_{L^{4/3}(\Omega)}  \lesssim |\Omega|^{1/4} \norm{T}_{L^2(\Omega)} +M\,.         \label{1b2}
\end{align}

\noindent
Next, let $\Psi \in \Lambda^1\otimes L^2(\Omega)$ be arbitrary. Consider the problem (componentwise)

$$\Delta \psi=\Psi\quad, \quad \psi|_{\partial \Omega}=0\quad,\quad \psi\in \Lambda^1\,.$$
Per Proposition \ref{callit} and the Sobolev embedding theorem, we have
\begin{align}
\norm{d\psi}_{L^4(\Omega)} +\norm{d^\star \psi}_{L^4(\Omega)} \lesssim \norm{D^2\psi}_{L^2(\Omega)}\lesssim \norm{\Psi}_{L^2(\Omega)}.   \label{1b3}
\end{align}
Next, we compute
\begin{align}
\langle \mu T, \Psi \rangle &= \langle \mu T, \Delta \psi\rangle \nonumber
\\&= \langle \mu T, dd^\star \psi +d^\star d\psi\rangle  \nonumber
\\& =\langle d^\star (\mu T), d^\star\psi  \rangle +\langle d(\mu T), d\psi\rangle  \nonumber
\\&\overset{\eqref{1b1}, \eqref{1b2}, \eqref{1b3}} \lesssim (|\Omega|^{1/4}\norm{T}_{L^2(\Omega)} +M)  \norm{\Psi}_{L^2(\Omega)}.  \nonumber
\end{align}
This shows that
$$\norm{\mu T}_{L^2(\Omega)}\lesssim |\Omega|^{1/4} \norm{T}_{L^2(\Omega)} +M\,.$$
Since $\mu\equiv 1$ on $\Omega_k,$ the desired conclusion follows.

\end{proof}

 The energy $\mathcal{E}(\Sigma)$ cannot be assumed to be small since it is not bounded from below. Instead, we rescale our domain so as to obtain a ball $B\subset \mathbb{R}^4$ of arbitrary center and radius with
$$\vec{\norm{h}}_{L^4(B)} < \varepsilon,$$
\noindent
where $\varepsilon >0$ may be chosen as small as we please.
\noindent
A particular useful quantity  is the 2-vector-valued 2-form
$$\vec{\eta}:=d\vec{\Phi}\overset{\wedge}\wedge_4 d\vec{\Phi}.$$

\begin{prop}
It holds
\begin{align}
\norm{d^\star\vec{\eta}}_{L^4(B)} <\varepsilon  \label{ssssj}
\end{align}\end{prop}
\begin{proof}
It suffices to realise that
$$d^\star \vec{\eta}= (4\vec{H}\wedge \nabla_j\vec{\Phi}+\nabla^i\vec{\Phi}\wedge\vec{h}_{ij})\,\, dx^j$$
and $\vec{h}\in L^4(B)$, $d\vec{\Phi}\in L^\infty(B)$ and $||\vec{H}||_{L^4(B)}\lesssim ||\vec{h}||_{L^4(B)}\leq \varepsilon$.
\end{proof}

\section{The creation of three characteristic two-forms}  \label{amena}

Let $$E(\Omega):=\norm{DH}_{L^2(\Omega)} +\norm{H}_{L^4(\Omega)} <\infty.$$
We have seen in Chapter \ref{latty1} using translation invariance of $\mathcal{E}(\Sigma)$ that there exists a divergence free tensor $\vec{V}$, that is $\nabla_j\vec{V}^j=\vec{0}$. Accordingly, we need to solve
$$d\star\vec{L}_0=\star\vec{V}$$
for $\vec{L}_0\in \Lambda^2(\Omega, \mathbb{R}^5)$ and understand the regularity of $\vec{L}_0$. We will call upon Proposition \ref{curca}, but we need to first understand the regularity of $\vec{V}$. The exact expression for  $\vec{V}$  (see Theorem \ref{cons-law}) reveals that
$$\vec{V}^j=\nabla^j(\Delta_\perp\vec{H})+\vec{a}^j+\nabla_k(\vec{b}^{kj})+\vec{f}^j\,\,,$$
with
\begin{align}
\norm{\vec{a}}_{L^1(\Omega)} +\vec{\norm{b}}_{L^{4/3}(\Omega)}   \lesssim E(\Omega)    \label{adq}
\end{align}
and 
$$\vec{f}^j\simeq (\vec{h}^{ij}\cdot\Delta_\perp\vec{H})\nabla_i\vec{\Phi}.$$
We can write 
$$\vec{f}=\nabla_k(h^{ij}\nabla^kH\nabla_i\vec{\Phi})-\vec{h}^{ij}h_{ik}\nabla^kH-
(\nabla^kH\nabla_kh^{ij})\nabla_i\vec{\Phi},$$
so that
$$\vec{\norm{f}}_{L^1\oplus \dot W^{-1,4/3}(\Omega)}\lesssim E(\Omega).$$
Altogether, we find
$$\vec{\norm{V}}_{\dot W^{-2,2}\oplus L^1(\Omega)}\lesssim E(\Omega),$$
where we have used that $\Delta_\perp \vec{H}\in \dot W^{-1,2}$  and the fact that $\dot W^{-1,4/3} \subset \dot W^{-2,2}$ by the dual of the Sobolev injection. Now, we call upon Proposition \ref{curca} to obtain $\vec{w}\in \Lambda^2\otimes \dot W^{-1,2}(\Omega)$ with $d^\star \vec{w}=\vec{V}$
and
\begin{align}
\vec{\norm{w}}_{\dot W^{-1,2}(\Omega)}\lesssim \vec{\norm{V}}_{\dot W^{-2,2}\oplus L^1(\Omega)}\lesssim E(\Omega).  \label{dfd}
\end{align}
We then set $\vec{L}_0:=\vec{w}$, so that $d\star\vec{L}_0=\star \vec{V}$.  From \eqref{dfd} we have that $\vec{L}_0\in \Lambda^2\otimes \dot W^{-1,2}(\Omega)$. Per Claim \ref{claimclaim},  we know that the operator $\star$ preserves  $\dot W^{-1,2}$ so that $\star\vec{L}_0\in \Lambda^2\otimes \dot W^{-1,2}(\Omega)$ with the estimate
\begin{align}
\vec{\norm{\star L_0}}_{\dot W^{-1,2}(\Omega)} \lesssim E(\Omega).  \label{tenn}
\end{align}

\noindent
Defining as in Proposition \ref{coro} the vector-valued 2-form
$$\star\vec{L}=\vec{L}_0-dH^2\wedge_4 d\vec{\Phi},$$
we see that $\vec{L}\in \dot W^{-1,2}$ with $d\vec{L}=d\star \vec{L}_0$ and the estimate
\begin{align}\vec{\norm{L}}_{\dot W^{-1,2}(\Omega)}\lesssim \vec{\norm{\star L_0}}_{\dot W^{-1,2}(\Omega)} +\norm{H}_{L^4(\Omega)}\norm{\nabla H}_{L^2(\Omega)}\overset{\eqref{tenn}}\lesssim  E(\Omega),       \label{ax}
\end{align}
where we have used that $L^{4/3}$ injects continuously into $\dot W^{-1,2}$ which is a consequence of the Sobolev injection $W_0^{1,2}\subset L^4$. 
As in Proposition \ref{coro}, there exists 2-forms $S_0$ and $\vec{R}_0$ satisfying
$$dS_0=\vec{L}\overset{.}\wedge_4 d\vec{\Phi} \quad\mbox{and}\quad d\vec{R}_0= \vec{L}\overset{\wedge}\wedge_4 d\vec{\Phi}+\star d\vec{u}-(\vec{J}+8H^2\vec{H})\wedge \star d\vec{\Phi}$$
where $\vec{J}:=\frac{1}{2}\Delta_\perp \vec{H}+\frac{1}{2}|\vec{h}|^2\vec{H} -7|\vec{H}|^2\vec{H}$.

\noindent
In addition, we are free to demand $d^\star S_0=0$ and $d^\star\vec{R}_0=\vec{0}$.
\\
Next is to understand the regularity of $\vec{R}_0$ and $S_0$. Recalling the definition of $\vec{J}$ (see Proposition \ref{coro}) and the fact that $\Delta_\perp\vec{H}\in \dot W^{-1,2}$, we find that $\vec{J}\in  \dot W^{-1,2}$. It remains to study the regularity of $d\vec{u}$.
To this end we have defined (see Proposition \ref{coro})
$$\frac{5}{3}H^2 d^\star\vec{\eta}= d\vec{u} +d^\star\vec{v},\quad\quad \vec{u}\in\Lambda^0(\Omega, \Lambda^2(\mathbb{R}^5)).$$
Clearly $d^\star\vec{u}=\vec{0}$.  Let $\vec{\Theta}\in L^2(\Omega)$ be arbitrary. Consider the problem
$$\Delta \vec{\theta}=\vec{\Theta} \quad \mbox{on}\,\, \Omega\quad\mbox{with}\,\,\, \vec{\theta}|_{\partial \Omega}=\vec{0}.$$

\noindent
Proposition \ref{callit} confirms that
\begin{gather}
\vec{\norm{D\theta}}_{L^4(\Omega)}  \lesssim \vec{\norm{\theta}}_{W_0^{2,2}(\Omega)} \lesssim \vec{\norm{\Theta}}_{L^2(\Omega)}.  \label{klnb}
\end{gather}

\noindent
Observe that
\begin{align}
\langle \vec{\Theta}, \vec{u} \rangle&= \langle \Delta\vec{\theta}, \vec{u}\rangle =\langle d\vec{\theta}, d\vec{u} \rangle   \nonumber
\\&= \langle d\vec{\theta}, H^2 d^\star \vec{\eta}- d^\star \vec{v} \rangle   \nonumber
\\&= \langle d\vec{\theta}, H^2 d^\star\vec{\eta}\rangle  \nonumber
\\&\lesssim \norm{H^2 d^\star\vec{\eta}}_{L^{4/3}(\Omega)}  \vec{\norm{\theta}}_{L^4(\Omega)}      \nonumber
\\&\lesssim \norm{H^2 d^\star\vec{\eta}}_{L^{4/3}(\Omega)} \vec{\norm{\Theta}}_{L^2(\Omega)}.  \nonumber
\end{align}
This shows that
\begin{align}
\vec{\norm{u}}_{L^2(\Omega)} \lesssim \norm{H^2 d^\star\vec{\eta}}_{L^{4/3}(\Omega)}  \lesssim \norm{H}^2_{L^4(\Omega)} \norm{d^\star\vec{\eta}}_{L^4(\Omega)}  \overset{\eqref{ssssj}}\lesssim \varepsilon E(\Omega).  \label{zcx}
\end{align}

\noindent
We have observed from equation \eqref{hads8} that $\vec{u}\in \Lambda^0(\Omega, \mathbb{R}^5)$ satisfies
$$\Delta \vec{u}=\frac{5}{3}d^\star (H^2 d^\star \vec{\eta}),$$
which is of the type investigated in Proposition  \ref{adz}.
\\
We split $\vec{u}=\vec{u}_0+\vec{u}_1$ where 
\begin{equation*}
 \begin{cases}
           \Delta \vec{u}_0=\vec{0} & \mbox{in}\,\, \Omega \\
            \,\,\,\,\, \vec{u}_0= \vec{u} & \mbox{on} \,\, \partial \Omega,
       \end{cases} \quad
\quad\quad\mbox{and} \quad\quad\begin{cases}
            \Delta \vec{u}_1=\Delta\vec{u} & \mbox{in}\,\, \Omega \\
            \,\,\,\,\, \vec{u}_1= \vec{0} & \mbox{on} \,\, \partial \Omega,
       \end{cases}
\end{equation*}
The equation for $\vec{u}_1$ is handled as in Proposition \ref{adz}. We find
\begin{align}
\norm{D\vec{u}_1}_{L^{4/3}(\Omega)}  \lesssim \norm{H^2}_{L^2(\Omega)}\norm{d^\star\vec{\eta}}_{L^4(\Omega)} \overset{\eqref{ssssj}} \lesssim \varepsilon E(\Omega).   \label{vcb}
\end{align}
In particular, the Sobolev embedding theorem gives
\begin{align}
\norm{\vec{u}_1}_{L^2(\Omega)} \lesssim \varepsilon E(\Omega).  \nonumber
\end{align}
Hence by \eqref{zcx}
\begin{align}
\norm{\vec{u}_0}_{L^2(\Omega)} \lesssim \varepsilon E(\Omega)   . \label{pio}
\end{align}
On the other hand, the Caccioppoli inequality yields\footnote{The closure of $\Omega '$ is  compact in $\Omega$.}
\begin{align}
\norm{D\vec{u}_0}_{L^{4/3}(\Omega ')} \lesssim \norm{D\vec{u}_0}_{L^2(\Omega ')} \lesssim \norm{\vec{u}_0}_{L^2(\Omega)} \overset{\eqref{pio}}\lesssim \varepsilon E(\Omega)  \quad\quad\forall\,\,\,\Omega '\Subset \Omega 
\end{align}

\noindent
Together with \eqref{vcb}, the latter gives
\begin{align}
\norm{D\vec{u}}_{L^{4/3}(\Omega ')}  \lesssim \varepsilon E(\Omega) \quad\quad \forall\,\,\, \Omega '\Subset \Omega.   \label{host}
\end{align}

\noindent
Using Claim \eqref{claimclaim} and the estimates we have of the quantities $\vec{u}$, $\vec{J}$ and $\vec{L}$, we are now ready to state the first regularity property:
$$dS_0\,\, ,\,\, d\vec{R}_0\in\dot W^{-1,2}(\Omega)$$

\noindent
with the estimate
\begin{align}
\norm{dS_0}_{\dot W^{-1,2}(\Omega)} +||d\vec{R}_0|| _{\dot W^{-1,2}(\Omega)} \lesssim E(\Omega_1)    \quad\quad\forall\,\, \Omega\Subset \Omega_1.   \label{simm}
\end{align}

\noindent
Since $S_0$ and $\vec{R}_0$ are co-closed, by Proposition \eqref{111}, we obtain 
\begin{align}
\norm{S_0}_{L^2(\Omega)} +||\vec{R}_0||_{L^2(\Omega)} \lesssim E(\Omega_1) \quad\quad \Omega\Subset \Omega_1.  \label{ax1}
\end{align}

\begin{rmk}  \label{rama}
Note that
\begin{align}
\norm{d^\star\vec{v}}_{L^{4/3}(\Omega ')} \lesssim \norm{d\vec{u}}_{L^{4/3}(\Omega ')} + \norm{H}^2_{L^4(\Omega)} \norm{d^\star\vec{\eta}}_{L^4(\Omega)} \overset{\eqref{ssssj}, \eqref{host}} \lesssim \varepsilon E(\Omega_1).
\end{align}
\end{rmk}

\noindent
Our next goal is to obtain an estimate similar to \eqref{simm} but with $L^{4/3}$ norm in place of $\dot W^{-1,2}$. In order to do this, we first introduce an operator $\vec{\mathcal{P}}$ which helps to redefine $\vec{L}$.
Next, we redefine the two-forms $S_0$ and $\vec{R}_0$ in order to create $S$ and $\vec{R}$ respectively. It turns out that $dS$ and $d\vec{R}$ can then be estimated in $L^{4/3}$ as desired (this is done in Section \ref{toute}). The remaining part of the present section is devoted to introducing $\vec{\mathcal{P}}$ and redefining $\vec{L}$, $S_0$ and $\vec{R}_0$.

\begin{lem} \label{conty}
The operator 
$$\mathcal{\vec{P}}: \Lambda^1(\Sigma, \Lambda^1(\mathbb{R}^5))\longrightarrow \Lambda^2(\Sigma,\Lambda^2(\mathbb{R}^5))\quad\quad\mbox{defined by}\quad\quad \mathcal{\vec{P}}:\vec{\ell}\longmapsto \vec{\ell}\overset{\wedge}\wedge_4 d\vec{\Phi}$$
is algebraically invertible. Moreover \,$\mathcal{\vec{P}}^{-1}$ injects $L^2$ into itself continuously.
\end{lem}

\begin{proof}
We have 
$$\mathcal{\vec{P}}_{pq}:=\big(\mathcal{\vec{P}}(\vec{\ell})  \big)_{pq} =\vec{\ell}_p\wedge \nabla_q\vec{\Phi}-\vec{\ell}_q\wedge\nabla_p\vec{\Phi}.$$
\noindent
Let $\vec{n}$ be the normal vector to the Willmore-type hypersurface $\Sigma$, defined on a small patch $\Omega \subset \mathbb{R}^4$ 
$$\vec{n}:=\star_{\mathbb{R}^5}\star (d\vec{\Phi}\overset{\wedge}\wedge_4 d\vec{\Phi}\overset{\wedge}\wedge_4 d\vec{\Phi}\overset{\wedge}\wedge_4 d\vec{\Phi})\in \Lambda^0(\Lambda^1)$$
where $\star_{\mathbb{R}^5}$ acts on the 4-vectors in the ambient space while $\star$ acts on the 4-forms in the parameter space.

\noindent
We compute
\begin{align}
\mathcal{\vec{P}}_{pq}\cdot (\vec{n}\wedge \nabla^q\vec{\Phi})&= (\vec{\ell}_p\wedge \nabla_q\vec{\Phi}-\vec{\ell}_q\wedge\nabla_p\vec{\Phi})\cdot (\vec{n}\wedge\nabla^q\vec{\Phi})             \nonumber
\\&=4\vec{n}\cdot\vec{\ell}_p-\vec{n}\cdot\vec{\ell}_p          \nonumber
\\&= 3\vec{n}\cdot\vec{\ell}_p\,, \nonumber
\end{align}
so that
\begin{gather}
\vec{\eta}\cdot\vec{\ell}_p =\frac{1}{3}\mathcal{\vec{P}}_{pq}\cdot (\vec{n}\wedge\nabla^q\vec{\Phi}).   \label{ta}
\end{gather}

\noindent
On the other hand, we have
\begin{align}
\mathcal{\vec{P}}_{pq}\cdot (\nabla^p\vec{\Phi}\wedge \nabla^q\vec{\Phi})&= (\vec{\ell}_p\wedge \nabla_q\vec{\Phi}-\vec{\ell}_q\wedge \nabla_p\vec{\Phi})\cdot (\nabla^p\vec{\Phi}\wedge \nabla^q\vec{\Phi})   \nonumber
\\&= 2(4\vec{\ell}_p\cdot \nabla^p\vec{\Phi}-\nabla^p\vec{\Phi}\cdot\vec{\ell}_p)   \nonumber
\\&= 6\vec{\ell}_p\cdot\nabla^p\vec{\Phi}\,,   \nonumber
\end{align}
so that
\begin{gather}
\vec{\ell}_p\cdot\nabla^p\vec{\Phi} =\frac{1}{6} \mathcal{\vec{P}}_{pq} \cdot (\nabla^p\vec{\Phi}\wedge \nabla^q\vec{\Phi}).     \label{sothat}
\end{gather}

\noindent
Finally, we check that
\begin{align}
\mathcal{\vec{P}}_{pq}\cdot (\nabla^i\vec{\Phi}\wedge \nabla^q\vec{\Phi}) &= (\vec{\ell}_p\wedge \nabla_q\vec{\Phi}-\vec{\ell}_q\wedge \nabla_p\vec{\Phi})\cdot (\nabla^i\vec{\Phi}\wedge \nabla^q\vec{\Phi})   \nonumber
\\&= 2\vec{\ell}_p\cdot \nabla^i\vec{\Phi} +\delta^i_p \vec{\ell}_q\cdot \nabla^q\vec{\Phi}   \nonumber
\\&\overset{\eqref{sothat}}=2\vec{\ell}_p\cdot \nabla^i\vec{\Phi} +\frac{1}{6}\delta^i_p \mathcal{\vec{P}}_{sq} \cdot (\nabla^s\vec{\Phi}\wedge \nabla^q\vec{\Phi}),  \nonumber
\end{align}
so that
\begin{align}
\vec{\ell}_p\cdot \nabla^i\vec{\Phi} =\frac{1}{2}\mathcal{\vec{P}}_{pq}\cdot (\nabla^i\vec{\Phi}\wedge \nabla^q\vec{\Phi})  -\frac{1}{12}\delta^i_p \mathcal{\vec{P}}_{sq}\cdot (\nabla^s\vec{\Phi}\wedge \nabla^q\vec{\Phi}).    \label{ta1}
\end{align}

\noindent
According to \eqref{ta} and \eqref{ta1}, the 1-form $\vec{\ell}$ can be totally recovered from $\mathcal{\vec{P}}$ via
$$\vec{\ell}_p=\left[\frac{1}{2}\mathcal{\vec{P}}_{pq}\cdot (\nabla^i\vec{\Phi}\wedge \nabla^q\vec{\Phi})  -\frac{1}{12}\delta^i_p \mathcal{\vec{P}}_{sq}\cdot (\nabla^s\vec{\Phi}\wedge \nabla^q\vec{\Phi})  \right]\nabla_i\vec{\Phi} +\left[\frac{1}{3}\mathcal{\vec{P}}_{pq}\cdot (\vec{n}\wedge\nabla^q\vec{\Phi})  \right]\vec{n}.$$
That $\mathcal{\vec{P}}^{-1}$ maps $L^2$ onto itself continuously is trivial.

\end{proof}

\noindent
Observe that given any $\vec{\mu}\in\Lambda^2(\Omega, \Lambda^2(\mathbb{R}^5))\otimes L^2(\Omega)$, we can define
\begin{align}
\vec{\ell}:= \mathcal{\vec{P}}^{-1}(\vec{\mu}).   \label{jert}
\end{align}

\noindent
We now create a new 2-form 
$$\vec{L}_1:=\vec{L}+ d\vec{\ell}$$
satisfying
\begin{gather}
\vec{\norm{L_1}}_{\dot W^{-1,2}(\Omega)}  \lesssim \vec{\norm{L}}_{\dot W^{-1,2}(\Omega)} + \vec{\norm{\ell}}_{L^2(\Omega)}  \overset{\eqref{ax}} \lesssim E(\Omega_1) +||\vec{\mu}||_{L^2(\Omega)}.               \nonumber
\end{gather}

\noindent
Consider now 
\begin{align}
S:=S_0+\vec{\ell}\overset{.}\wedge_4 d\vec{\Phi} \quad\quad\mbox{and}\quad\quad \vec{R}:=\vec{R}_0+\vec{\ell}\overset {\wedge}\wedge_4 d\vec{\Phi}. \label{merio}
\end{align}
We immediately see that
$$dS=\vec{L}_1\overset{.}\wedge_4 d\vec{\Phi}\quad\mbox{and}\quad d\vec{R}= \vec{L}_1\overset{\wedge}\wedge_4 d\vec{\Phi} +\star d\vec{u} -(\vec{J} +8H^2\vec{H})\wedge\star d\vec{\Phi}.$$

\section{Estimating $S$ and $\vec{R}$ in the correct space}  \label{toute}

From Chapter \ref{latty1}, Proposition \ref{coro}, the two-forms $S_0$ and $\vec{R}_0$ satisfy 
$$dS_0= \vec{L}\overset{\cdot}\wedge_4 d\vec{\Phi}\quad\quad\mbox{and}\quad\quad d\vec{R}_0= \vec{L}\overset{\wedge}\wedge_4 d\vec{\Phi} +\star d\vec{u}-\vec{J}_0\wedge\star d\vec{\Phi}$$
where for convenience, we have set $\vec{J}_0:= \vec{J}+8H^2\vec{H}.$

\noindent
We compute
\begin{align}
\quad\star\vec{\eta}\overset{\cdot}\wedge_4\star(d\vec{R}_0-\star d\vec{u} +\vec{J}_0\wedge\star d\vec{\Phi})  \nonumber
& =\epsilon_{ijk[r} \epsilon_{ab]pq} \vec{L}^{ij}\cdot\nabla^p\vec{\Phi} \, g^{kq}\,\,\, dx^{a}\wedge_4 dx^b\wedge_4 dx^r  \nonumber
\\&= \vec{L}_{[ab}\cdot\nabla_{r]}\vec{\Phi}  \,\,\,dx^a\wedge_4 dx^b\wedge_4 dx^r   \nonumber  
\\&=dS_0.  \nonumber
\end{align}

\noindent
As $\vec{J}_0$ is a normal vector, we have 
$$\star\vec{\eta}\overset{\cdot}\wedge_4 (\vec{J}_0\wedge \star d\vec{\Phi})=0.$$

\noindent
Hence 
\begin{align}
dS_0 =\star\vec{\eta} \overset{\cdot} \wedge_4 \star d\vec{R}_0 -\star\vec{\eta} \overset{\cdot}\wedge_4 d\vec{u}.  \label{rela}
\end{align}

\noindent
By a similar token, we find (see Appendix \ref{grace5} for the full computation)
\begin{align}
&\quad\star\vec{\eta} \overset{\bullet} \wedge_4 \star (d\vec{R}_0-\star d\vec{u} +\vec{J}_0\wedge \star d\vec{\Phi})    \nonumber
\\&=\epsilon_{ijk[r} \epsilon_{ab]pq} (\nabla^p\vec{\Phi}\wedge \nabla^q\vec{\Phi})  \bullet (\vec{L}^{ij}\wedge \nabla^k\vec{\Phi})  \,\,\, dx^a\wedge_4 dx^b\wedge_4 dx^r   \nonumber
\\&=\epsilon_{ijk[r} \epsilon_{ab]pq} \left[(\nabla^p\vec{\Phi}\cdot \vec{L}^{ij}) \nabla^q\vec{\Phi}\wedge \nabla^k\vec{\Phi} + g^{qk} \nabla^p\vec{\Phi}\wedge \vec{L}^{ij} \right]    \nonumber
\\&= \star\vec{\eta}\wedge_4 \star dS_0 -(d\vec{R}_0-\star d\vec{u} +\vec{J}_0\wedge \star d\vec{\Phi})  +\vec{Q}_0   \label{kli}
\end{align}

\noindent
where $\vec{Q}_0\in \Lambda^3(\Omega, \Lambda^2(\mathbb{R}^5))$ is given by
$$\vec{Q}_0:=\frac{1}{12}\delta^{\alpha\beta\gamma}_{abr}\left[(\vec{L}_{k\alpha}\cdot\nabla^k\vec{\Phi}) 
\nabla_\beta\vec{\Phi} -(\vec{L}_{k\alpha}\cdot\nabla_\beta\vec{\Phi})\nabla^k\vec{\Phi}  \right]\wedge \nabla_\gamma \vec{\Phi}\,\,\,\, dx^a\wedge_4 dx^b\wedge_4 dx^r.$$

\noindent
As done in Appendix \ref{grace5}, we find
$$\star\vec{\eta}\overset {\bullet}\wedge_4 (\vec{J}_0\wedge d\vec{\Phi})=-\vec{J}_0\wedge \star d\vec{\Phi}.$$

\noindent
Hence from equation \eqref{kli}, we have
\begin{align}
d\vec{R}_0=\star\vec{\eta}\overset{\bullet}\wedge_4 \star d\vec{R}_0+ \star\vec{\eta}\wedge_4 \star dS_0 -\star\vec{\eta}\overset{\bullet}\wedge_4 d\vec{u} -\star d\vec{u} +\vec{Q}_0.  \label{kakio}
\end{align}

\noindent
We continue with one trivial identity which holds by the antisymmetry of $(\star\vec{R}_0)$ and $(\star S_0)$,
\begin{align}
&\quad\quad(\nabla_l\vec{\eta}_{ij}+\nabla_j\vec{\eta}_{il})\bullet(\star\vec{R}_0)^{lj}   + (\nabla_l\vec{\eta}_{ij}+\nabla_j\vec{\eta}_{il})(\star S_0)^{lj} \nonumber
\\&=(\nabla_l\vec{\eta}_{ij}-\nabla_l\vec{\eta}_{ij})\bullet(\star\vec{R}_0)^{lj}   + (\nabla_l\vec{\eta}_{ij}-\nabla_l\vec{\eta}_{ij})(\star S_0)^{lj} \nonumber
\\&=\vec{0}.  \nonumber
\end{align}

\noindent
As $d\vec{\eta}=\vec{0}$, that is, $\nabla_i\vec{\eta}_{jl} +\nabla_j\vec{\eta}_{li} +\nabla_l\vec{\eta}_{ij}=\vec{0}$, the latter yields
\begin{align}
 (-\nabla_i\vec{\eta}_{jl}+2\nabla_j\vec{\eta}_{il})\bullet(\star\vec{R}_0)^{lj} +(-\nabla_i\vec{\eta}_{jl}+2\nabla_j\vec{\eta}_{il})(\star S_0)^{lj} =\vec{0}.  \label{lassz}
\end{align}

\noindent
We next compute
\begin{align}
&\quad\nabla^j(\vec{\eta}_{jl}\bullet (\star\vec{R}_0)^{il}+\vec{\eta}_{jl}(\star S_0)^{il})   \nonumber
\\&= \vec{\eta}_{jl}  \bullet \nabla^j(\star\vec{R}_0)^{il} + (d^\star\vec{\eta})_l \bullet(\star\vec{R}_0)^{il}    \nonumber
+\vec{\eta}_{jl}  \nabla^j(\star{S}_0)^{il} + (d^\star\vec{\eta})_l (\star{S}_0)^{il}     \nonumber
\\&= -\vec{\eta}_{jl}\bullet \big(\nabla^l(\star\vec{R}_0)^{ji} +\nabla^i(\star\vec{R}_0)^{lj}  \big)   +(d^\star\vec{\eta})_l\bullet (\star\vec{R}_0)^{il}    \nonumber
\\&\quad \quad  -\vec{\eta}_{jl} \big(\nabla^l(\star{S}_0)^{ji} +\nabla^i(\star{S}_0)^{lj}  \big)   + (d^\star\vec{\eta})_l  (\star{S}_0)^{il}  \label{okka}
\\&= -\nabla^l\big( \vec{\eta}_{jl} \bullet (\star \vec{R}_0)^{ji}  \big) +\nabla^l\vec{\eta}_{jl}\bullet (\star\vec{R}_0)^{ji} +(d^\star\vec{\eta})_l\bullet (\star\vec{R}_0)^{il}    \nonumber
\\&\quad\quad+\nabla^i\big(\vec{\eta}_{lj}\bullet(\star\vec{R}_0)^{lj}  \big)  - \nabla^i\vec{\eta}_{lj}\bullet(\star \vec{R}_0)^{lj}  \nonumber
\\&\quad\quad\quad
-\nabla^l\big( \vec{\eta}_{jl}  (\star {S}_0)^{ji}  \big) +\nabla^l\vec{\eta}_{jl}(\star S_0)^{ji} +(d^\star\vec{\eta})_{l}(\star S_0)^{il}  \nonumber
\\&\quad\quad\quad\quad+\nabla^i\big(\vec{\eta}_{lj}(\star{S}_0)^{lj}  \big)  -  \nabla^i\vec{\eta}_{lj} (\star {S}_0)^{lj}   \nonumber
\\&= -\nabla^j\big( \vec{\eta}_{jl}\bullet (\star\vec{R}_0)^{il} +\vec{\eta}_{jl} (\star S_0)^{il}  \big) + \nabla^i\big( \vec{\eta}_{jl}\bullet (\star\vec{R}_0)^{jl} +\vec{\eta}_{jl} (\star S_0)^{jl}  \big)  \nonumber
\\& \quad\quad - \nabla^i\vec{\eta}_{lj}\bullet (\star \vec{R}_0)^{lj}  -\nabla^i\vec{\eta}_{lj}(\star S_0)^{lj}  +2(d^\star\vec{\eta})_l\bullet (\star \vec{R}_0)^{il} +2(d^\star\vec{\eta})_l(\star {S}_0)^{il} \nonumber
\end{align}
where we have used that $S_0$ and $\vec{R}_0$ are co-closed to obtain \eqref{okka}.

\noindent
Hence
\begin{align}
\nabla^j\big( \vec{\eta}_{jl}\bullet (\star\vec{R}_0)^{il} +\vec{\eta}_{jl} (\star S_0)^{il}  \big)& = \frac{1}{2}\nabla^i\big( \vec{\eta}_{jl}\bullet (\star\vec{R}_0)^{jl} +\vec{\eta}_{jl} (\star S_0)^{jl}  \big)  \nonumber
\\&\quad -\frac{1}{2}\left(  \nabla^i\vec{\eta}_{lj}\bullet(\star \vec{R}_0)^{lj}  +\nabla^i\vec{\eta}_{lj}(\star S_0)^{lj}\right) \nonumber
\\&\quad\quad +(d^\star\vec{\eta})_l\bullet (\star \vec{R}_0)^{il} +(d^\star\vec{\eta})_l(\star {S}_0)^{il} .  \label{tyq}
\end{align}

\noindent
This yields
\begin{align}
&\quad\quad\left( \star(\star\vec{\eta}\overset{\bullet}\wedge_4 \star d\vec{R}_0 +\star\vec{\eta}\wedge_4 \star dS_0  )  \right)^i   \nonumber
\\&= \epsilon^{jkli}\frac{1}{2}\epsilon_{jkpq}\big( \vec{\eta}^{pq}\bullet \nabla^r(\star\vec{R}_0)_{rl} +\vec{\eta}^{pq} \nabla^r(\star S_0)_{rl}   \big)  \nonumber
   \\&  =2\vec{\eta}^{li}\bullet \nabla^j(\star \vec{R}_0)_{jl}  +2\vec{\eta}^{li} \nabla^j(\star S_0)_{jl}\nonumber
\\&=2 \nabla^j\left (\vec{\eta}^{li}\bullet (\star{R}_0)_{jl}  +\vec{\eta}^{li} (\star{S}_0)_{jl}      \right)   -2\nabla^j\vec{\eta}^{li} \bullet(\star\vec{R}_0)_{jl}   -2\nabla^j\vec{\eta}^{li}(\star S_0)_{jl}   \nonumber
\\&\overset{\eqref{tyq}}=2 \nabla^j\left (\vec{\eta}_{jl}\bullet (\star\vec{R}_0)^{il}-\vec{\eta}^{il}\bullet (\star{R}_0)_{jl}  
+\vec{\eta}_{jl} (\star {S}_0)^{il} -\vec{\eta}^{il} (\star{S}_0)_{jl}     \right)  \nonumber
\\&\quad\quad - \nabla^i\left( \vec{\eta}_{jl}\bullet (\star\vec{R}_0)^{jl}   + \vec{\eta}_{jl} (\star{S}_0)^{jl}\right) +\left(\nabla^i\vec{\eta}_{lj}\bullet (\star\vec{R}_0)^{lj} +\nabla^i\vec{\eta}_{lj} (\star{S}_0)^{lj}   \right) \nonumber
\\&\quad\quad\quad -2\nabla^j\vec{\eta}^{li}\bullet(\star\vec{R}_0)_{jl}       -2\nabla^j\vec{\eta}^{li}(\star{S}_0)_{jl}       -2(d^\star\vec{\eta})_l\bullet (\star \vec{R}_0)^{il} -2(d^\star\vec{\eta})_l (\star {S}_0)^{il}  \nonumber
\\&\overset{\eqref{lassz}}= \nabla^j\tensor{\vec{P}}{_j^i}  +\nabla^i\vec{G}  -4\nabla^j\vec{\eta}^{li}\bullet(\star\vec{R}_0)_{jl}       -4\nabla^j\vec{\eta}^{li}(\star{S}_0)_{jl}  -2(d^\star\vec{\eta})_l\bullet (\star \vec{R}_0)^{il} -2(d^\star\vec{\eta})_l (\star {S}_0)^{il}   \nonumber
\\&=\nabla^j\tensor{\vec{P}}{_j^i}  +\nabla^i\vec{G} +\star\vec{Q}_1,   \nonumber
\end{align}
where
\begin{align}
\vec{P}_{ij}:= \tensor{\vec{\eta}}{_{j}^l}\bullet \tensor{(\star\vec{R}_0)}{_{i}_l}-\tensor{\vec{\eta}}{_{i}^l}\bullet (\star{R}_0)_{jl}  
+\tensor{\vec{\eta}}{_{j}^l} (\star {S}_0)_{il} -\tensor{\vec{\eta}}{_{i}^l} (\star{S}_0)_{jl}
\end{align}
is the component of the two-vector valued 2-form $\vec{P}:=\vec{P}_{ij}\,\, dx^i\wedge_4 dx^j$ and
$$\vec{G}:= \vec{\eta}_{lj}\bullet (\star\vec{R}_0)^{jl}   + \vec{\eta}_{lj} (\star{S}_0)^{jl}   \,\,\,\in\Lambda^0(\mathbb{R}^4, \Lambda^2(\mathbb{R}^5)). $$

\noindent
Introducing \eqref{kakio}, we find
\begin{align}
\star d\vec{R}_0=d^\star \vec{P} +d\vec{G} +\vec{U} +\star\vec{Q},    \label{zads}
\end{align}
where for notational convenience we have set
\begin{align}
\vec{U}:=\star(\star\vec{\eta}\overset{\bullet}\wedge_4 d\vec{u} +\star d\vec{u}),  \quad\quad\vec{Q}:=\vec{Q}_0+\vec{Q}_1.  \label{notational}
\end{align}

\noindent
By estimates \eqref{ax}, \eqref{ax1} and Claim \eqref{claimclaim}, we have
\begin{align}
||\vec{Q}||_{\dot W^{-1,2}(\Omega)}  \lesssim ||\vec{L}||_{\dot W^{-1,2}(\Omega)} +||\vec{R}_0||_{L^2(\Omega)} +||S_0||_{L^2(\Omega)}   \overset{\eqref{ax},\eqref{ax1}}\lesssim E(\Omega_1).\label{lives}
\end{align}

\noindent
Standard Hodge theory gives the decomposition
$$\vec{Q}=d\vec{A}+d^\star \vec{B}$$
$$\mbox{with}\quad d^\star\vec{A}=\vec{0}\,,\,d\vec{B}=\vec{0},  \quad\quad\quad \vec{A}\in\Lambda^2(\mathbb{R}^4, \Lambda^2(\mathbb{R}^5)), \vec{B}\in\Lambda^4(\mathbb{R}^4, \Lambda^2(\mathbb{R}^5)).$$

\noindent
We will now write $\vec{R}$ in terms of $\vec{G}$, $\star{B}$ and $\vec{U}$ and according to the definition \eqref{jert} of $\vec{\ell}$. 
Owing to Lemma  \ref{conty}, we define
$$\vec{\ell}:= \mathcal{P}^{-1}(-\vec{A}-\star\vec{P}).$$
\noindent
By Proposition \ref{nadaa}, we know that $\vec{A}$ belongs to $L^2$; by estimate \eqref{ax1} and equation \eqref{zads}, we know that $\star\vec{P}$ also belongs to $L^2$. Thus the sum $\vec{A}+\star\vec{P}$ belongs to $ L^2$ and now $\vec{\ell}$ is a suitable candidate to redefine $\vec{L}$.

\noindent
Accordingly, from \eqref{zads} we have
\begin{align}
d(\vec{G}+\star\vec{B}) +\vec{U} = \star d(\vec{R}_0-\vec{A}-\star\vec{P}) =\star d(\vec{R}_0+\vec{\ell}\overset{\wedge}\wedge_4 d\vec{\Phi}) \overset{\eqref{merio}}=\star d\vec{R} .   \label{hags}
\end{align}

\noindent
We finish this section with estimates. Let $\Omega_1$ appearing in \eqref{ax1}   be a ball $B_r$ of fixed radius $r$. Let $\vec{F}:= \vec{G}+\star\vec{B}$, then we have

\begin{lem}\label{delaa}
It holds
\begin{align}
||d\vec{F}||_{L^{4/3}(B_{kr})}\lesssim (k+\varepsilon)E(B_r)\quad\quad\forall\,\,0<k<1/2.    \nonumber
\end{align}
\end{lem}

\begin{proof}
Observe first that
$$\Delta \vec{F}=-d^\star \vec{U}.$$
To study this equation, note that
\begin{align}
||\vec{G}||_{L^2(B_{r/2})}\quad& \lesssim \quad||\vec{\eta}||_{L^\infty(B_{r/2})} \left(||\vec{R}_0||_{L^2(B_{r/2})} +||S_0||_{L^2(B_{r/2})}   \right)     \nonumber
\\\quad\quad&\overset{ \eqref{ax1}} \lesssim E(B_r) .  \label{gag1}
\end{align}

\noindent
By Proposition \eqref{nadaa}, we have
\begin{align}
||\star\vec{B}||_{L^2(B_{r/2})}\lesssim ||\vec{Q}||_{\dot W^{-1,2}(B_{r/2})} \overset{\eqref{lives}}\lesssim E(B_{r}).  \label{gaga}  \end{align}

\noindent
Combining the estimates \eqref{gag1} and \eqref{gaga}  yield
\begin{align}
||\vec{F}||_{L^2(B_{r/2})}  \lesssim E(B_{r}).  \label{amer1}
\end{align}

\noindent
We split $\vec{F}=\vec{F}_0+\vec{F}_1$   such that

\begin{equation*}
 \begin{cases}
           \mathcal{L}[ \vec{F}_0]=\vec{0} & \mbox{in}\,\, B_{r/2} \\
            \,\,\,\,\, \vec{F}_0= \vec{F} & \mbox{on} \,\, \partial B_{r/2},
       \end{cases} \quad
\quad\quad\mbox{and} \quad\quad\begin{cases}
            \mathcal{L} [\vec{F}_1]=|g|^{1/2}d^\star\vec{U} & \mbox{in}\,\, B_{r/2}\\
            \,\,\,\,\, \vec{F}_1= \vec{0} & \mbox{on} \,\, \partial B_{r/2}.
       \end{cases}
\end{equation*}

\noindent
Elliptic estimates (see Proposition \ref{adz} and recall  $\vec{F}_1$ is a 0-form) yield
\begin{align}
||\vec{F}_1||_{L^2(B_{r/2})}+ ||d\vec{F}_1||_{L^{4/3}(B_{r/2})}  \lesssim ||\vec{U}||_{L^{4/3}(B_{r/2})}.  \label{amer}
\end{align}

\noindent
On the other hand the Jensen and Caccioppoli inequalities give
\begin{align}
||d\vec{F}_0||_{L^{4/3}(B_{kr/2})}& \lesssim kr||d\vec{F}_0||_{L^2(B_{kr/2})}  \nonumber
\\&\lesssim kr||d\vec{F}_0||_{L^2(B_{r/2})}  \nonumber
\\&\lesssim k||\vec{F}_0||_{L^2(B_{r/2})}  \nonumber
\\&\lesssim k||\vec{F}||_{L^2(B_{r/2})} +k||\vec{F}_1||_{L^2(B_{r/2})}   \nonumber
\\&\overset{\eqref{amer}  \eqref{amer1}} \lesssim ||\vec{U}||_{L^{4/3}(B_{r/2})} +kE(B_{r}).   \nonumber
\end{align}

\noindent
Combining the latter with \eqref{amer} yields the estimate
\begin{align}
||d\vec{F}||_{L^{4/3}(B_{kr/2})} & \lesssim ||d\vec{F}_0||_{L^{4/3}(B_{kr/2})} +||d\vec{F}_1||_{L^{4/3}(B_{r/2})} \nonumber
\\&\lesssim ||\vec{U}||_{L^{4/3}(B_{r/2})} +kE(B_{r}).
\end{align}
It remains to observe from \eqref{notational} that
\begin{align}
||\vec{U}||_{L^{4/3}(B_{r/2})}\lesssim ||d\vec{u}||_{L^{4/3}(B_{r/2})}\overset{\eqref{host}} \lesssim \varepsilon E(B_r).   \label{starss}
\end{align}
Relabelling the domains and choosing $0<k<1/2$ rather than $0<k<1$, we obtain
$$||d\vec{F}||_{L^{4/3}(B_{kr})} \lesssim (k+\varepsilon) E(B_r).$$
This completes the proof.

\end{proof}

\begin{cor}\label{takeova}
We have
$$\norm{dS}_{L^{4/3}(B_{kr})} +||d\vec{R}||_{L^{4/3}(B_{kr})} \lesssim (k+\varepsilon) E(B_r) \quad\quad \forall\,\,k\in(0,1/2).$$
\end{cor}
\begin{proof}
It follows immediately from Lemma \ref{delaa}, equation \eqref{hags} and the estimate
\begin{align}
||\vec{U}||_{L^{4/3}(B_{kr})}\lesssim ||d\vec{u}||_{L^{4/3}(B_{kr})}\overset{\eqref{host}} \lesssim \varepsilon E(B_r).   
\end{align}
that
\begin{align}
||d\vec{R}||_{L^{4/3}(B_{kr})}  \lesssim (k+\varepsilon) E(B_r).  \label{messia}
\end{align}

\noindent
Recall that $dS$ and $d\vec{R}_0$ are defined from the modified $\vec{L}+d\vec{\ell}$. Thus $S$ and $\vec{R}$ form an admissible pair and are linked together by the same relation  \eqref{rela}
linking $S_0$ and $\vec{R}_0$, namely
$$dS=\star\vec{\eta}\overset{\cdot}\wedge_4\star d\vec{R} -\star\vec{\eta}\overset{\cdot}\wedge_4 d\vec{u}.$$

\noindent
Calling upon \eqref{messia} and  \eqref{host} gives the desired estimate
$$\norm{dS}_{L^{4/3}(B_{kr})}\lesssim  (k+\varepsilon) E(B_r).$$

\end{proof}
\noindent
The advantage of $S$ and $\vec{R}$ over $S_0$ and $\vec{R}_0$ is that Corollary \ref{takeova} involves $L^{4/3}$ norm rather than $\dot W^{-1,2}$ norm in the estimate \eqref{simm}.

\section{The return equation: controlling the geometry} \label{linkage}
We recall the return equation from Chapter \ref{latty1} (precisely \eqref{recca})
$$12d^\star \vec{T}=\star (d\vec{R} \overset{\bullet}\wedge_4 d\vec{\Phi}) +\star (dS \wedge_4 d\vec{\Phi}) +6\star(\star d^\star \vec{v}\overset{ \bullet}\wedge_4 d\vec{\Phi})$$
where 
\begin{align}
\vec{T}:=(\nabla_j\vec{H}+2Hh_{jk}\nabla^k\vec{\Phi}-2H^2\nabla_j\vec{\Phi})dx^j   \label{nameT}
\end{align}
and $\vec{v}$ is as in Remark \ref{rama}. Note first that
\begin{align}
\vec{\norm{T}}_{L^2(B_{kr})} \lesssim \norm{dH}_{L^2(B_{kr})} +\norm{H}_{L^4(B_{kr})} \norm{h}_{L^4(B_{kr)}} \lesssim E(B_{kr})\lesssim E(B_{r}).    \label{jk}
\end{align}
On the other hand, we find from \eqref{nameT}
\begin{align}d^\star \vec{T}= \nabla^j \vec{T}_j=\Delta\vec{H} +2\nabla^j(
Hh_{jk}\nabla^k\vec{\Phi}-H^2\nabla_j\vec{\Phi}) . \label{nied}
\end{align}

\noindent
Using \eqref{nied} and introducing Remark \ref{rama}  and Corollary \ref{takeova} into the return equation yields for all $k\in (0,1/2)$ 
\begin{align}
\norm{\Delta H}_{L^{4/3}(B_{kr})}\lesssim ||d^\star \vec{T}||_{L^{4/3}(B_{kr})} +\varepsilon E(B_r) \lesssim (k+\varepsilon) E(B_r).   \label{jk1}
\end{align}

\noindent
On the other hand, a direct computation reveals
$$(d\vec{T})_{ij}=2(h_{ik}\nabla_jH- h_{jk}\nabla_iH)\nabla^k\vec{\Phi} +2\nabla_jH^2\nabla_i\vec{\Phi}-2\nabla_iH^2\nabla_j\vec{\Phi},$$

\noindent
so that
\begin{gather}
||d\vec{T}||_{L^{4/3}(B_{kr})}  \lesssim \norm{h}_{L^4(B_{kr})}  \norm{dH}_{L^2(B_{kr})}   \lesssim \varepsilon E(B_r).   \label{jk2}
\end{gather}

\noindent
Using \eqref{jk}, \eqref{jk1} and \eqref{jk2} into Lemma \ref{tech} yields
\begin{align}
||\vec{T}||_{L^2(B_{kr})} \lesssim (\varepsilon +k)E(B_r). \label{hg1}
\end{align}

\noindent
We now find
\begin{align}
\norm{dH}_{L^2(B_{kr})} &\overset{\eqref{nameT}}\lesssim ||\vec{T}||_{L^2(B_{kr})} +\norm{H}_{L^4(B_{kr})} \norm{h}_{L^4(B_{kr})}       \nonumber
\\&\overset{\eqref{hg1}}\lesssim (\varepsilon +k)E(B_r)   +\varepsilon E(B_r)  \nonumber
\\&\lesssim (\varepsilon+k) E(B_r).  \label{ja1b}
\end{align}

\noindent
By the Sobolev embedding theorem and the Jensen's inequality, it follows that
\begin{align}
\norm{H}_{L^4(B_{kr})} \lesssim \norm{H}_{L^2(B_{kr})} +\norm{dH}_{L^2(B_{kr})}  \overset {\eqref{ja1b}} \lesssim (\varepsilon +k) E(B_r).  \label{233}
\end{align}
 Combining \eqref{ja1b} and \eqref{233}, we have the crucial estimate
$$E(B_{kr})\equiv \norm{dH}_{L^2(B_{kr})}+\norm{H}_{L^4(B_{kr})}  \lesssim (\varepsilon +k)E(B_r).$$

\noindent
Since $\varepsilon$ and $k$ may be chosen as small as we please, by a standard controlled growth argument (cf. Lemma 5.13 in \cite{gia}), we  find the Morrey decay
\begin{align}
E(B_r)\equiv \norm{DH}_{L^2(B_r)} +\norm{H}_{L^4(B_r)} \lesssim E(B_1)r^\beta\quad\quad\forall\,\, r<1.    \label{morrey}
\end{align}
We emphasise that this is true for any $\beta \in (0,1)$.

Consider the maximal function

$$\mathcal{M}_{3-\beta}[f]:= \sup_{\rho>0}\rho^{-1-\beta}\norm{f}_{L^1(B_\rho)}.$$

\noindent
By Jensen's inequality, we have 
$$\mathcal{M}_{3-\beta}[f]\lesssim \sup_{\rho>0} \rho^{-\beta}\norm{f}_{L^{4/3}(B_\rho)}.$$
Using the Morrey decay  \eqref{morrey} and \eqref{jk}, we obtain
\begin{align}
\norm{\mathcal{M}_{3-\beta}[\Delta\vec{H}]}_{L^\infty(B_r)} \lesssim E(B_1)\quad\quad\forall\,\, r<1.  
\end{align}

\noindent
Let $f$ be a locally integrable function on $\mathbb{R}^n$.
Recall that for a number $\alpha$ satisfying $\alpha\in (0, n)$, the {\it Riesz potential $\mathcal{I}_\alpha$ of order $\alpha$ } of $f$ is defined by the convolution
$$(\mathcal{I}_\alpha *f)(x):=\frac{1}{C(\alpha)}\int_{\mathbb{R}^n} f(y)|x-y|^{\alpha-n} dy$$
where $C(\alpha)=\pi^{n/2}2^\alpha\,\, \Gamma(\frac{\alpha}{2})/\Gamma(\frac{n-\alpha}{2})$ is a constant.

 \noindent
We will now use the following result from \cite{ada}.  
\begin{prop}  \label{nbd}
If $\alpha>0$, $0<\lambda\leq n$, $1<p<\lambda/\alpha$, $1\leq q\leq \infty$, and $f\in L^p(\mathbb{R}^n)$ with $\mathcal{M}_{\lambda/p} [f]\in L^q(\Omega)$, $\Omega\subset \mathbb{R}^n$, then
$$\norm{\mathcal{I}_\alpha[f]}_{L^r(\Omega)} \lesssim \norm{M_{\lambda/p }[f]}_{L^q(\Omega)}^{\alpha p/\lambda} \norm{f}_p^{1-\alpha p/\lambda} $$
where $1/r=1/p-\alpha/\lambda+(\alpha p)/(\lambda q)$.
\end{prop}

\vskip3mm
\noindent
Putting $\alpha=1$, $q=\infty$, $p=4/3$, $\lambda=4(3-\beta)/3$ in Proposition \ref{nbd}, we find
\begin{align}
\norm{\mathcal{I}_1[\Delta \vec{H}]}_{L^s(B_r)}& \lesssim \norm{\mathcal{M}_{3-\beta}[\Delta\vec{H}]}_{L^{\infty}(B_r)}^{1/(3-\beta)}  \norm{\Delta \vec{H}}_{L^{4/3}(B_r)}^{(2-\beta)/(3-\beta)}   \nonumber
\\&\lesssim E(B_1)^{1-\frac{4}{3s}} E(B_r)^{\frac{4}{3s}}
\end{align}

\noindent
where
$$s:= \frac{4}{3}\left( \frac{3-\beta}{2-\beta} \right)\in \left(2, \frac{8}{3} \right)$$

\noindent
and $\mathcal{I}_1$ denotes the Riesz potential of order 1.

\vskip 3mm
\noindent
Let now $s\in\left(2,\frac{8}{3} \right)$. Using standard elliptic estimates and the Sobolev embedding theorem, we have for all $r<1$
\begin{align}
\norm{D\vec{H}}_{L^s(B_r)}& \lesssim \norm{\mathcal{I}_1[\Delta\vec{H}]}_{L^s(B_r)} +r^{2-\frac{4}{s}}\norm{D\vec{H}}_{L^2(B_1)}  \nonumber
\\&  \lesssim E(B_1)^{1-\frac{4}{3s}} E(B_r)^{\frac{4}{3s}} +r^{2-\frac{4}{s}}E(B_1).  \label{gha}
\end{align}

\noindent
As \eqref{gha} holds for all $s\in \left(2,\frac{8}{3}   \right)$, we see that the integrability of $D\vec{H}$ has been improved.

\vskip 3mm
\noindent
With this new information on the integrability of $d\vec{H}$, the above procedure may be repeated untill we obtain that $\vec{H}$ is Lipschitz. Once this is known, we see that $\Delta \vec{\Phi}$ is as well Lipschitz. Since $|g|^{1/2}\in W^{1,4}$ by hypothesis, we see that $|g|^{1/2}\Delta \vec{\Phi}$ lies in $W^{1,4}$. We have
$$\mathcal{L}[\vec{\Phi}]=|g|^{1/2}\Delta\vec{\Phi}.$$
Calling upon Theorem 1.1 in \cite{gru}, the Green kernel $\mathcal{G}$ of $\mathcal{L}$ satisfies $D\mathcal{G}\in L^{{4/3},\infty}$ where $L^{{4/3},\infty}$ is the \footnote{The weak-$L^{4/3}$ Marcinkiewicz space $L^{{4/3},\infty} (B_1(0))$ is defined as those functions $f$ which satisfy $\sup_{\alpha >0} \alpha^{4/3}|\{x\in B_1(0): |f(x)|\geq \alpha \}|<\infty$ . The space $L^{{4/3},\infty}$ is also a Lorentz space.}weak Marcinkiewicz space.

\noindent
Formally, the solution $\vec{\Phi}$
is given by the convolution
$$\vec{\Phi}=\mathcal{G}* \mathcal{L}[\vec{\Phi}].$$

\noindent
Hence by the convolution rule for Lorentz spaces (cf. \cite{hunt})
$$D^2\vec{\Phi}= D\mathcal{G}*D\mathcal{L}[\vec{\Phi}]\in L^{{4/3},\infty}*L^4\subset L^p\quad\quad\forall\,\, p<\infty.$$

\noindent
Accordingly, by the Sobolev embedding theorem, $D\vec{\Phi}\in \bigcap_{p<\infty}W^{1,p} \subset \bigcap_{\alpha<1} C^{0,\alpha}$. The regularity of $\vec{\Phi}$ has thus also improved. In particular, the metric coefficients are H\"{o}lder continuous of all orders, and it follows that the standard analysis of second-order uniformly elliptic  operators is now at hand (cf. \cite{gil}). Eventually, by standard elliptic arguments, we reach the conclusion that $\vec{\Phi}$ is smooth. 

\vskip 3mm
Finally, by standard elliptic arguments we obtain that
\begin{align}
\norm{DH}_{L^\infty(B_r)} +\frac{1}{r} \norm{H}_{L^\infty(B_r)}  \lesssim \frac{1}{r^2} \left(\norm{DH}_{L^2(B_1)} +\norm{H}_{L^4(B_1)}   \right) \quad\quad\forall \, r<1.
\end{align}
This concludes the proof.

\chapter {Final Thoughts and Openings}\label{chap5}
In the previous Chapters, we presented a conformally invariant four dimensional generalisation of the Willmore energy and proved, via conservation laws, that critical points of this energy are regular. There are other important analysis aspects of this energy that are yet to be explored. Our goal here is not to outline four dimensional analogues of every result in literature concerning two dimensional Willmore surfaces, as there are some results that cannot be formulated for the energy \eqref{jkas}. For instance, we do not have Willmore-type conjectures\footnote{The Willmore conjecture \cite{willmoretom}  simply says that in the class of immersed tori, the Clifford torus has the minimal Willmore energy of $2\pi^2$.} for the energy \eqref{jkas} as the energy is not bounded below (see Proposition 1.1 in \cite{robingraham} and the discussion beneath it). In the present Chapter, we outline some interesting analysis aspects of the energy \eqref{jkas} that are directly related to our work. 
\vskip16mm

\section{Robert Bryant's formalism} 

In his seminal article \cite{bry}, Robert Bryant introduced the following formalism for studying the Willmore immersions of a surface via its conformal Gauss map. Let $\vec{\Phi}:\Sigma\rightarrow \mathbb{R}^3$ be an immersion with unit normal vector $\vec{n}$. Let $\mathbb{S}^{4,1}$ be the deSitter space. The conformal Gauss map $\vec{Y}:\Sigma\rightarrow \mathbb{S}^{4,1}$ is defined via
\begin{align}
\vec{Y}:= H \vec{X}
+ \vec{N}               \nonumber
\end{align}
where
\begin{align}
\vec{X}:= \begin{pmatrix}
\vec{\Phi} \\
\frac{1}{2}(|\vec{\Phi}|^2-1)\\
\frac{1}{2}(|\vec{\Phi}|^2+1)
\end{pmatrix}  \quad\quad \mbox{and}\quad\quad \vec{N}:=  \begin{pmatrix}
\vec{n} \\
\vec{n}\cdot\vec{\Phi}\\
\vec{n}\cdot\vec{\Phi}
\end{pmatrix}     \nonumber
\end{align}
 One can easily show\footnote{See Appendix \ref{desitt}} that
\begin{align}
d\vec{Y}= \vec{X}dH-(h_0)d\vec{X} \quad\quad\mbox{and}\quad\quad \langle \nabla_i\vec{Y},\nabla_j\vec{Y}\rangle_{\mathbb{R}^{4,1}}=(h_0)_{ik}(h_0)^k_j=\frac{1}{2}|h_0|^2g_{ij} \nonumber
\end{align}
where $\mathbb{R}^{4,1}$ is the Minkowski space and $\mathbb{S}^{4,1}\subset\mathbb{R}^{4,1}$. Thus the Dirichlet energy of $\vec{Y}$ in the deSitter space is the same as the Willmore energy of $\vec{\Phi}$, that is 
$$\Vert d\vec{Y}\rVert_{\mathbb{S}^{4,1}}^2=\int_\Sigma|h_0|^2\,\,d\textnormal{vol}_g.$$
One can see that the immersion $\vec{\Phi}$ is Willmore if and only if $\vec{Y}$ is minimal (see \cite{esch}). 
In addition, by defining the quartic form
$$Q:= \langle\vec{Y}_{zz}, \vec{Y}_{zz}\rangle_{\mathbb{S}^{4,1}}\,\,dz^4,$$
where $z\in \mathbb{C}$ is a local coordinate, Bryant was able to prove that $\vec{\Phi}$ is Willmore if and only if $Q$ is holomorphic. This leads us to ask if the Bryant formalism can be reproduced in four dimensions for the energy \eqref{jkas}. 
Preliminary computations\footnote{See Appendix \ref{desitt}} show that 

\noindent
$$\lVert  \Delta\vec{Y}  \rVert^2_{\mathbb{S}^{6,1}}=\int_{\Sigma}-2|\nabla H |^2+|h_0|^4\,\,d\textnormal{vol}_g.$$

\noindent
Although the quantity  $\lVert  \Delta\vec{Y}  \rVert^2_{\mathbb{S}^{6,1}}$ expressed above significantly differs from the energy \eqref{jkas} (in codimension one), it reproduces the leading term $|\nabla H|^2$ and a conformally invariant term $|h_0|^4$. This observation leads us to believe that there is a link between $\vec{Y}$ and $\vec{\Phi}$.

\section{Fr\'ed\'eric H\'elein's Coulomb frames}

\noindent
The well-known uniformisation theorem classifies Riemann surfaces into three different types. It says that every Riemann surface is conformally equivalent to either the unit disk, the complex plane, or the Riemann sphere. 
The induced metric $g$ on any  $\Sigma\subset \mathbb{R}^3$ satisfies $g=e^{2\lambda}\delta$ for some conformal parameter $\lambda$. There is a link between the Willmore energy of $\Sigma$ and the regularity of $\lambda$. Since the parameter $\lambda$ can be viewed as a ``distortion'' ratio between a flat disk and its image on $\Sigma$, the geometry of the surface $\Sigma$ can be understood via its Willmore energy. In Chapter 5 of \cite{helein}, H\'{e}lein presented an approach to the study of Willmore surfaces via the introduction of {\it Coulomb frames.}\footnote{See also \cite{hubert}}
\vskip 3mm

\noindent
Let $\vec{\Phi}:\Sigma\rightarrow\mathbb{R}^3$ be an immersion and let $\Gamma(\mathcal{T}\vec{\Phi}(\Sigma))$ be a section of the tangent bundle to $\vec{\Phi}(\Sigma)\subset\mathbb{R}^3$.  A {\it moving frame} on $\vec{\Phi}$ is a pair $\vec{e}:=(\vec{e}_1,\vec{e}_2)\in \Gamma(\mathcal{T}\vec{\Phi}(\Sigma))\times \Gamma(\mathcal{T}\vec{\Phi}(\Sigma))$ if for every point $x:=(x_1,x_2)\in \Sigma$, the pair $\vec{e}\,(x_1,x_2)$ is a positive\footnote{That is, the frame $\vec{e}$ agrees with a fixed orientation of $\vec{\Phi}(\Sigma)$.} orthormal basis of $\mathcal{T}_x\vec{\Phi}(\Sigma)$. A {\it Coulomb frame} is a frame $\vec{e}$ that satisfies the following {\it Coulomb condition}:
$$d^\star(\vec{e}_1\cdot d\vec{e}_2)=0.$$

\noindent
Owing to the Gauss equation and the Gauss-Bonnet theorem, the Willmore energy can be written as
\begin{align}
\frac{1}{4}\int_\Sigma |h|^2 +\pi\chi(\Sigma) d\textnormal{vol}_g
\end{align}
where $\chi(\Sigma)$ is a topological invariant. 
By controlling the Willmore energy of $\Sigma$, H\'elein (Lemma 5.1.4 of \cite{helein}) was able to construct a Coulomb frame with good regularity properies:
\begin{theo} Suppose that the Willmore energy satisfies  
$$\int_{D_1(0)}|h|^2d\textnormal{vol}_g<\frac{8\pi}{3}$$
then there exists a frame $\vec{e}\in W^{1,2}(D_1(0))$ such that
$$\lVert d\vec{e}_1  \rVert_{L^2(D_{1/2}(0))}  +\lVert d\vec{e}_2  \rVert_{L^2(D_{1/2}(0))}  \lesssim \lVert h  \rVert_{L^2(D_1(0))}.$$
\end{theo}

\noindent
A simple computation gives a link 
relating the conformal parameter $\lambda$ and the frame $\vec{e}$ 
$$-\star d\lambda= \vec{e}_1\cdot d\vec{e}_2 $$ 
so that we have the Wente structure
\begin{align}
-\Delta_0 \lambda= \star(d\vec{e}_1\cdot d\vec{e}_2)
\label{klas245}
\end{align}
where $\Delta_0$ denotes the negative flat Laplacian.
With the help of techiques of integration by compensation, one eventually concludes that $\lambda$ is bounded from above and below,  thus giving a control of the geometry of $\Sigma$. 

\vskip3mm
\noindent
Can the geometry of the Willmore-type hypersurface $\Sigma$ associated with the energy \eqref{jkas} be studied by following a similar procedure to the one above?

\vskip3mm
\noindent
The situation is different in four dimensions. First, we do not have an analogue of the uniformisation theorem due to the diversity of the conformal structures. Only partial answers are available in literature (see for instance \cite{{chang}, {chen}, {djadii}}). 
This makes it difficult to process the question above. However, if we assume that $\Sigma$ is locally conformally flat then the induced metric satisfies
 $g=e^{4\lambda}\delta$. A similar procedure to the one above shows that the conformal parameter $\lambda$ satisfies
\begin{align}\Delta_0^2\lambda=e^{4\lambda}Q     \label{klas1}\end{align}

\noindent
where $Q:=\frac{1}{6}\Delta R+\frac{1}{16}\star \mbox{Pf} $\,\, is Tom Branson's celebrated {\it Q-curvature}  \cite{ceoy}\footnote{The actual expression for $Q$ is $\frac{1}{6}\Delta R+ \frac{1}{16}\star \mbox{Pf}-\frac{1}{4}|W|^{2}$ where $W$ is the Weyl curvature, $R$ is the scalar curvature and Pf is the Pfaffian 4-form. It is known that $W$ vanishes for a conformally flat manifold.}.
Clearly, the left hand side of \eqref{klas1} is a divergence. Therefore, the quantity $\frac{1}{16}\star \mbox{Pf}$ must also be a divergence. 
It is expected that a  link should exist between a good Coulomb frame and the energy  \eqref{jkas} .

\section{Analysis of branched Willmore surfaces}
The local analysis of point singularities of the two dimensional Willmore surfaces is motivated by the notion of concentration compactness:  sequences of Willmore immersions with uniformly bounded energy converges everywhere except at finitely many points where the Willmore energy is concentrated. Such point singularities, which naturally occur as blow-ups of the Willmore flow, were first analytically investigated in codimension one by Kuwert and Sch\"atzle (see \cite{{stru},{struku}}). 
In \cite{stru}, they proved the regularity of Willmore surfaces having unit density singularity; and later considered a more general case of Willmore surfaces with higher order singularity \cite{struku}. Their results were generalised to higher codimension by Bernard and Rivi\`ere in  \cite{yber2}  via the reformulation of the Willmore equation in divergence form and techniques of integration by compensation (see also \cite{riv}). 

\vskip 3mm
\noindent
Let $\vec{\Phi}:D_1(0)\backslash\{0\}\rightarrow \mathbb{R}^m$ be an  immersion.
A point $p$ on $\Sigma$ is called a branch point\footnote{The terms {\it branch point} and {\it singularity} are similar but quite different. The immersion $\vec{\Phi}$ and its derivative $d\vec{\Phi}$ are both defined at branch points but $\vec{\Phi}$ ceases to be an immersion there.} if $\vec{\Phi}$ degenerates at $p$ or if $d\vec{\Phi}$ vanishes at $p$. For simplicity, we have assumed that there is a branch point at the origin $0$ and $\vec{\Phi}$ is smooth away from the origin but continuous there. Using Noether's theorem, $\vec{\Phi}$ satisfies an equation of the type 
\begin{align}
d^\star\vec{V}=\vec{0}   \label{lak2}
\end{align}
for some tensor $\vec{V}$.
The equation \eqref{lak2} hold, in particular, away from the origin. 
A constant $\vec{s}_0\in\mathbb{R}^m$, called the first residue, can be defined by
\begin{align}
\vec{s}_0:=\int_{\partial D_1(0)}\vec{\nu}\cdot \vec{V}  \nonumber
\end{align}
where $\vec{\nu}$ is the unit outward normal vector to $\partial D_1(0)$. One can show that the immersion is smooth across the singularity at the origin when the residue vanishes. This condition on the residue is known as the point removability condition.  Bernard and Rivi\`ere   \cite{yber2} showed that branched Willmore surfaces are smooth if some point removability conditions are satisfied.

\vskip3mm
\noindent
It is interesting to ask if similar point removability conditions can be obtained in order to study
the four dimensional branched Willmore-type hypersurface associated with the energy \eqref{jkas}.
Since our regularity proof of Theorem \ref{heart} is not based on integration by compensation but on a Bourgain-Brezis type result, we expect enormous difficulty in answering the question above.

\clearpage
\appendix
\chapter {Brief Background Notes}\label{chap2}

\section{Geometric background}
In this section, we briefly recall definitions and fundamental facts that are well known in Geometry. A comprehensive discussion can be found in most standard texts such as \cite{jost}, \cite{kobayashi1}, \cite{kobayashi2} and \cite{willmoretj}.

\subsection{Manifolds}
A topological space $\mathcal{M}$ is {\it  locally Euclidean} if every point in $\mathcal{M}$ has a neighbourhood $U$ that is homeomorphic to an open subset $\Omega$ of a Euclidean space $\mathbb{R}^n$. The homeomorphism $f: U\rightarrow \Omega$ gives rise to the pair $(U,f)$ which is called a {\it chart}. The chart $(U,f)$ is {\it centered at} $x\in U$ if $f(x)=0$. A collection $\{(U_\alpha, f_\alpha)\}$ of charts that form an open covering of $\mathcal{M}$ is known as {\it atlas.}

A {\it manifold} $\mathcal{M}$ of dimension $m$ is a locally Euclidean topological space of dimension $m$. 
For the sake of precision, we give the following definition. More details can be found in \cite{jost}.

An $m$-dimensional topological manifold $\mathcal{M}$ is an Hausdorff, second countable (has a countable basis), $n$-dimensional locally Euclidean space.
A {\it differential manifold}  of class $C^k$ $(1\leq k\leq \infty)$ is a topological manifold $\mathcal{M}$ such that
\begin{enumerate}[(1)]
\item $\bigcup_{\alpha} U_{\alpha} =\mathcal{M}$.
\item For all $\alpha, \beta,$ $ f_\alpha\circ f_\beta^{-1}$ is $C^k$.
\item The atlas $\{U_\alpha, f_\alpha\}$ is maximal.
\end{enumerate}
On a differentiable manifold $\mathcal{M}$, we can define a rule that takes of each $p\in\mathcal{M}$  to the tangent space $T_p\mathcal{M}$. This rule is called a {\it vector field.} Thus vector fields are maps from  $\mathcal{M}$ onto the tangent bundle $T\mathcal{M}$. The vector field $X$ is called smooth if the associated mapping is smooth. Let $U\subset \mathcal{M}$ and $V\subset \mathbb{R}^m$ be open subsets and let $x:U\rightarrow V$ be a local coordinate chart. If $X:\mathcal{M}\rightarrow T\mathcal{M}$ is a vector field then the restriction $X|_U$ can  be represented via 
$$X(p)=\Sigma _{i=1}^m a_i(p)\frac{\partial}{\partial x_i}$$
where each function $a_i: U\rightarrow \mathbb{R}$ is smooth.

\subsection{Submanifolds}
A differentiable map $f:\mathcal{M}\rightarrow \mathcal{N}$ is called an immersion if for any $p\in\mathcal{M}$, the differential of $f$
$$df:T_p\mathcal{M}\rightarrow T_{f(p)}\mathcal{N}$$
is injective ($T_p\mathcal{M}$ denotes the tangent space of $\mathcal{M}$ at the point $p$).

Roughly speaking submanifolds are images of injective immersions. Just as a manifold looks locally like a Euclidean space, a submanifold can also be viewed as a subset of a manifold resembling a subspace of a Euclidean space.


\subsection{Differential forms}
\noindent
 We fix a coordinate basis $\{dx^i\}_{i=1,\cdots,4}$ and an orientation.  In order to ease the notation, we will denote the differential elements $dx^{i_1}\wedge_4...\wedge_4 dx^{i_{k}}  $ by $dx^{i_1...i_{k}}$. A $k$-form is defined locally as 
\begin{gather}
A=\frac{1}{k!}\sum_{i_1<\cdots<i_k} A_{[i_{1}...i_k]}\,\,dx^{i_1...i_{k}}   \nonumber
\end{gather}
which we abbreviate as
$$A=\frac{1}{k!}A_{[i_1...i_k]} dx^{i_1...i_{k}}$$
where 
$$A_{[i_1...i_k]}=\frac{1}{k!}\delta^{j_1...j_k}_{i_1...i_k} A_{j_1...j_k}.$$
The wedge product is defined as follows. For $A\in\Lambda^p$ and $B\in\Lambda^q$, we have

$$A\wedge B=\frac{1}{(p+q)!}A_{[a_1...a_p}B_{b_1...b_q]} dx^{a_1...a_pb_1...b_p}\in\Lambda^{p+q}.$$
Defining the Levi-Civita tensor $\epsilon_{abcd}:=|g|^{1/2}\mbox{sign} \begin{pmatrix}
1 & 2 & 3 & 4\\
a & b & c & d
\end{pmatrix}$ enables us to define the Hodge star operator. For $A\in\Lambda ^p$,
$$\star A=\frac{1}{(4-p)!}\epsilon_{i_1...i_4}\left( \frac{1}{p!}A^{[i_1...i_p]}\right)dx^{i_{p+1}...i_4}.$$

\noindent
With this definition of $\star$, we find 
$$\star\star A=(-1)^{p(4-p)}A.$$

\noindent
Also, for $A\in\Lambda^p$ and $B\in\Lambda^q$, we find that
\begin{align}
(\star A)\wedge B &=\frac{1}{(4-p+q)!}(\star A)_{[i_{p+1}...i_4} \wedge B_{j_1...j_q]} dx^{i_{p+1}...i_4 j_1...j_q}            \nonumber
\\&=\frac{1}{(4-p+q)!} \epsilon_{i_1...i_4}\left(\frac{1}{p!} A^{[i_1...i_p]}\right)\wedge B_{j_1...j_q} dx^{i_{p+1}...i_4j_1...j_q}.   \nonumber
\end{align}

\noindent
If $A$ is a $p$-form, then
$$dA=\frac{1}{(p+1)!}\nabla_{[i}A_{i_1...i_p]} dx^i\wedge_4 dx^{i_1...i_p},$$
For $A\in \Lambda ^p$ and $B\in\Lambda^q$, Leibnitz rule reads
$$d(A\wedge B)=dA\wedge B+(-1)^pA\wedge dB.$$
We also define $d^\star:=\star d\star.$ With this definition, we find that
$$d^\star A=\frac{1}{(p-1)!}\nabla^{i_1}A_{[i_1...i_p]} dx^{i_2...i_p}.$$
In particular, for a $p$-form $A$,  
$$\Delta A=d^\star dA +dd^\star A= \frac{1}{p!}\Delta_g A_{[i_1...i_p]} dx^{i_1...i_p}.$$

\begin{exa}  
 Let $A\in\Lambda^2$ and $B\in\Lambda^1$, then
\begin{align}
(\star A)\wedge B= \frac{1}{3!}(\star A)_{[kl}\wedge B_{m]} dx^{klm}=\frac{1}{6}\delta^{\alpha\beta\gamma}_{klm} \frac{1}{3!}\epsilon_{ij\alpha\beta}\frac{1}{2!} A^{ij}\wedge B_\gamma dx^{klm}.  \nonumber
\end{align}
If we apply the Hodge star operator to the above, we have
\begin{align}
\star((\star A)\wedge B)&=  \frac{1}{1!}\left[\epsilon^{klmr}\frac{1}{3!}\left(\frac{1}{6}\delta^{\alpha\beta\gamma}_{klm} \epsilon_{ij\alpha\beta}\left(\frac{1}{2!}A^{ij}  \right)\wedge B_\gamma  \right)   \right]    g_{rs} dx^s.      \nonumber
\end{align}


\end{exa}


\chapter{Some Computations I} \label{grace5}
\noindent
Let $\vec{A}$ and  $\vec{B}$ be 2-vectors. From the definition of the first-order contraction operator (see page \pageref{se123} for definition), we have  the following multiplication rule.
\begin{align}
\vec{A}\cdot \vec{B}&:= (\vec{A}_1\wedge\vec{A}_2) \cdot (\vec{B}_1\wedge\vec{B}_2)   \nonumber
\\&= (\vec{A}_1\cdot\vec{B}_1)(\vec{A}_2\cdot\vec{B}_2) -(\vec{A}_1\cdot\vec{B}_2)(\vec{A}_2\cdot\vec{B}_1)  \nonumber
\end{align}
\begin{align}
\vec{A}\bullet \vec{B}&:=(\vec{A}_1\wedge\vec{A}_2) \bullet (\vec{B}_1\wedge\vec{B}_2)   \nonumber
\\&=(\vec{A}_2\cdot \vec{B}_2)\vec{A}_1\wedge \vec{B}_1- (\vec{A}_2\cdot \vec{B}_1)\vec{A}_1\wedge \vec{B}_2-(\vec{A}_1\cdot \vec{B}_2)\vec{A}_2\wedge \vec{B}_1\nonumber
\\&\quad + (\vec{A}_1\cdot \vec{B}_1)\vec{A}_2\wedge \vec{B}_4.\nonumber
\end{align}

\noindent
With the rules above, we are ready for the following computation. We know that $S_0$ and $\vec{R}_0$ satisfy

$$dS_0=\vec{L}\overset{\cdot} \wedge_4 d\vec{\Phi}\quad\quad\mbox{and}\quad\quad d\vec{R}_0=\vec{L}\overset{\wedge}\wedge_4 d\vec{\Phi}+\star d\vec{u} -\vec{J}_0\wedge \star d\vec{\Phi} .$$

\noindent
We compute
\begin{align}
&\quad\star \vec{\eta}\overset{\bullet}\wedge_4 \star\left(d\vec{R}_0-\star d\vec{u} +\vec{J}_0\wedge\star d\vec{\Phi}  \right)  \nonumber
\\&=\star\vec{\eta}\overset{\bullet}\wedge_4 \star(\vec{L}\overset{\wedge}\wedge_4 d\vec{\Phi})  \nonumber
\\&=    \frac{1}{3!} (\star\vec{\eta})_{[ab}\bullet (\star(\vec{L}\overset{\wedge}\wedge_4 d\vec{\Phi}))_{r]} \,\, dx^{a}\wedge_4 dx^b\wedge_4 dx^r      \nonumber
\\&= \frac{1}{3!}\frac{1}{6}\delta_{abr}^{\alpha\beta\gamma}\left (\epsilon_{\alpha\beta pq}\frac{1}{2!}\nabla^p\vec{\Phi}\wedge\nabla^q\vec{\Phi} \right) \bullet \left( \epsilon_{\gamma ijk} \frac{1}{3!} \vec{L}^{ij}\wedge \nabla^k \vec{\Phi} \right) \,\, dx^{a}\wedge_4 dx^b\wedge_4 dx^r   \nonumber
\\&=\frac{1}{216}\delta^{\alpha\beta\gamma}_{abr}\epsilon_{\alpha\beta pq}\epsilon_{\gamma ijk}\left( -g^{qk} \vec{L}^{ij}\wedge\nabla^p\vec{\Phi} -(\vec{L}^{ij}\cdot\nabla^{p}\vec{\Phi})\nabla^k\vec{\Phi}\wedge\nabla^q\vec{\Phi}     \right) \,\,   dx^{a}\wedge_4 dx^b\wedge_4 dx^r    \nonumber
\\&:= \vec{A} -\frac{1}{126}\delta^{\alpha\beta\gamma}_{abr}\epsilon_{\alpha\beta pq}\epsilon_{\gamma ijk} \vec{B}^{ijpk}\wedge\nabla^q\vec{\Phi}\,\, dx^{a}\wedge_4 dx^b\wedge_4 dx^r. \label{derc}
\end{align}

\noindent
We have for the first term of \eqref{derc} 
\begin{align}
\vec{A}&= -\frac{1}{216} \delta^{\alpha\beta\gamma}_{abr} \epsilon_{\alpha\beta pq}\tensor{\epsilon}{_{\gamma ij}^{q}}\vec{L}^{ij}\wedge_4 \nabla^p\vec{\Phi}\,\, dx^{a}\wedge_4 dx^b\wedge_4 dx^r    \nonumber
\\&=-\frac{1}{216} \delta^{\alpha\beta\gamma}_{abr}\epsilon_{\alpha\beta pq}\epsilon^{cijq}g_{c\gamma}\vec{L}_{ij}\wedge \nabla^p\vec{\Phi} \,\, dx^{a}\wedge_4 dx^b\wedge_4 dx^r     \nonumber
\\&= -\frac{1}{216}\delta^{\alpha\beta\gamma}_{abr} \delta_{\alpha\beta pq}^{cijq}g_{c\gamma} \vec{L}_{ij}\wedge \nabla^p\vec{\Phi} \,\, dx^{a}\wedge_4 dx^b\wedge_4 dx^r   \nonumber
\\&= -\frac{1}{12} \vec{L}_{[ab}\wedge \nabla_{r]}\vec{\Phi}\,\, dx^{a}\wedge_4 dx^b\wedge_4 dx^r   \nonumber
\\&=-\vec{L}\overset{\wedge} \wedge_4 d\vec{\Phi}.  \label{plpl}
\end{align}

\noindent
Next is to compute the second term of \eqref{derc}. Note that we have set $\vec{B}^{ijkp}:= (\vec{L}^{ij}\cdot\nabla^p\vec{\Phi})\nabla^k\vec{\Phi}$ where $\vec{B}$ is antisymmetric in indices $(i,j,k)$.  
\\
We have
\begin{align}
&\quad\frac{1}{216}\delta^{\alpha\beta\gamma}_{abr} \epsilon_{\alpha\beta pq}\epsilon_{\gamma ijk} \vec{B}^{ijpk}\wedge\nabla^q\vec{\Phi}\,\, dx^{a}\wedge_4 dx^b\wedge_4 dx^r   \nonumber
\\&=  \frac{1}{216}\delta^{\alpha\beta \gamma}_{abr}\epsilon^{lm pq}g_{l \alpha}g_{m \beta}\epsilon_{\gamma ijk} \tensor {\vec{B}}{^{ij}_{p}^k}\wedge \nabla_q\vec{\Phi}   \,\, dx^{a}\wedge_4 dx^b\wedge_4 dx^r    \nonumber
\\&=\frac{1}{216}\delta^{\alpha\beta\gamma}_{abr}\delta^{lm pq}_{\gamma ijk}g_{l \alpha}g_{m \beta} \tensor {\vec{B}}{^{ij}_{p}^k} \wedge\nabla_q\vec{\Phi}  \,\, dx^{a}\wedge_4 dx^b\wedge_4 dx^r   \nonumber
\\&=\frac{1}{216}  \delta^{\alpha\beta\gamma}_{abr}   \left( \delta^p_\gamma \delta^{lm q}_{ijk} -\delta^q_\gamma \delta^{lm p}_{ijk} \right)      g_{l \alpha}g_{m \beta} \tensor {\vec{B}}{^{ij}_{p}^k} \wedge\nabla_q\vec{\Phi}  \,\,                     dx^{a}\wedge_4 dx^b\wedge_4 dx^r   \nonumber
\\&= \frac{1}{216}\delta^{\alpha\beta\gamma}_{abr} \left[\delta^p_\gamma\left(\delta^l_i \delta^{m q}_{jk} -\delta^l_j \delta^{m q}_{ik} + \delta^l_k \delta^{m q}_{ij} \right)  -\delta^q_\gamma \left( \delta^l_i \delta^{m p}_{jk} -\delta^l_j \delta^{m p}_{ik} \right.\right.  \nonumber
\\&\left.\left.\quad\quad\quad\quad\quad\quad
+\delta^l_k \delta^{m p}_{ij}  \right)\right]   g_{l \alpha}g_{m \beta} \tensor {\vec{B}}{^{ij}_{p}^k} \wedge\nabla_q\vec{\Phi}  \,\,                     dx^{a}\wedge_4 dx^b\wedge_4 dx^r    \nonumber
\\&=  \frac{1}{216}\delta^{\alpha\beta\gamma}_{abr}\left[2\delta^p_\gamma \delta^l_i\delta^{m q}_{jk}+ \delta^p_\gamma \delta^l_k\delta^{m q}_{ij} -2\delta^q_\gamma \delta^l_i\delta^{m p}_{jk}-\delta^q_\gamma \delta^l_k\delta^{m p}_{ij}\right]     g_{l \alpha}g_{m \beta} \tensor {\vec{B}}{^{ij}_{p}^k} \wedge\nabla_q\vec{\Phi}  \,\,                     dx^{a}\wedge_4 dx^b\wedge_4 dx^r   \nonumber
\\&=  \frac{1}{216}\delta^{\alpha\beta\gamma}_{abr}\left[2\delta^p_\gamma \delta^l_i \delta^m_j \delta^q_k -  2\delta^p_\gamma \delta^l_i \delta^m_k \delta^q_j +\delta^p_\gamma \delta^l_k \delta^m_i \delta^q_j -\delta^p_\gamma \delta^l_k \delta^m_j \delta^q_i \right]           g_{l \alpha}g_{m \beta} \tensor {\vec{B}}{^{ij}_{p}^k} \wedge\nabla_q\vec{\Phi}  \,\,                   dx^{a}\wedge_4 dx^b\wedge_4 dx^r   \nonumber
\\&\quad+\frac{1}{216}\left[-2\delta^q_\gamma \delta^l_i \delta^m_j \delta^p_k +  2\delta^q_\gamma \delta^l_i \delta^m_k \delta^p_j -\delta^q_\gamma \delta^l_k \delta^m_i \delta^p_j +\delta^q_\gamma \delta^l_k \delta^m_j \delta^p_i \right]           g_{l \alpha}g_{m \beta} \tensor {\vec{B}}{^{ij}_{p}^k} \wedge\nabla_q\vec{\Phi}  \,\,                    dx^{a}\wedge_4 dx^b\wedge_4 dx^r         \nonumber
\\&=   \frac{1}{216}\delta^{\alpha\beta\gamma}_{abr}\left[2\delta^q_\gamma\delta^l_k \delta^m_j \delta^p_i  -2\delta^q_\gamma\delta^l_i \delta^m_j \delta^p_k  +2\delta^q_\gamma\delta^l_i \delta^m_k \delta^p_j   \right]                         g_{l \alpha}g_{m \beta} \tensor {\vec{B}}{^{ij}_{p}^k} \wedge\nabla_q\vec{\Phi}  \,\,                     dx^{a}\wedge_4 dx^b\wedge_4 dx^r          \nonumber
\\&= \frac{1}{36} \delta^{\alpha\beta\gamma}_{abr} \delta^q_\gamma\delta^l_i \delta^m_j \delta^p_k      g_{l \alpha}g_{m \beta} \tensor {\vec{B}}{^{ij}_{p}^k} \wedge\nabla_q\vec{\Phi}  \,\,                     dx^{a}\wedge_4 dx^b\wedge_4 dx^r      \nonumber
\\&= \frac{1}{36} \delta^{\alpha\beta\gamma}_{abr} \delta^k_p\,\,\,\,  \tensor {\vec{B}}{_{\alpha\beta}^{p}_k} \wedge\nabla_\gamma\vec{\Phi}  \,\,                     dx^{a}\wedge_4 dx^b\wedge_4 dx^r   . \label{plpl1}
\end{align}
where $\tensor {\vec{B}}{_{\alpha\beta}^{p}_k} $ is antisymmetric in the indices $\alpha,\beta$ and $k$.

\noindent
Also, we have
\begin{align}
\star\vec{\eta} \overset{\bullet}\wedge_4 \star(\vec{J}_0\wedge \star d\vec{\Phi})&=-\frac{1}{3!2!}\frac{1}{6}\delta_{abr}^{\alpha\beta\gamma}\epsilon_{\alpha\beta pq} (\nabla^p\vec{\Phi}\wedge \nabla^q\vec{\Phi}) \bullet (\vec{J}_0\wedge\nabla_\gamma\vec{\Phi})\,\, \,\,                     dx^{a}\wedge_4 dx^b\wedge_4 dx^r  \nonumber
\\&=  \frac{1}{36}\delta^{\alpha\beta\gamma}_{abr}\epsilon_{\alpha\beta pq} \delta^{q}_\gamma \nabla_p\vec{\Phi} \wedge\vec{J}_0 \,\, dx^{a}\wedge_4 dx^b\wedge_4 dx^r \nonumber
\\&= -\vec{J}_0\wedge_4\star d\vec{\Phi} . \label{plpl2}
\end{align}

\noindent
Using \eqref{plpl}, \eqref{plpl1} and \eqref{plpl2} in \eqref{derc}, we have
$$-d\vec{R}_0=\star\vec{\eta}\overset{\bullet}\wedge_4 \star d\vec{R}_0 +\star d\vec{u}+ \star\vec{\eta}\overset{\bullet}\wedge_4 d\vec{u} + \frac{1}{36} \delta^{\alpha\beta\gamma}_{abr} \delta^k_p\,\,\,\,  \tensor {\vec{B}}{_{\alpha\beta}^{p}_k} \wedge\nabla_\gamma\vec{\Phi}  \,\,                     dx^{a}\wedge_4 dx^b\wedge_4 dx^r     .   $$
\noindent
On the other hand, we have

\begin{align}
\star \vec{\eta} \wedge_4 \star dS&= \frac{1}{432} \delta^{\alpha\beta\gamma}_{abr}\epsilon_{\alpha\beta pq}\epsilon_{\gamma ijk} (\vec{L}^{ij}\cdot\nabla^k\vec{\Phi})(\nabla^p\vec{\Phi}\wedge \nabla^q\vec{\Phi})\,\, \,\,                     dx^{a}\wedge_4 dx^b\wedge_4 dx^r          \nonumber
\\&= \frac{1}{432} \delta^{\alpha\beta\gamma}_{abr}\delta^{l    m    pq}_{\gamma ijk} g_{l     \alpha}g_{m    \beta}  \tensor {\vec{B}}{^{ijk}_{p}} \wedge\nabla_q\vec{\Phi}  \,\,                   \,\,                     dx^{a}\wedge_4 dx^b\wedge_4 dx^r   \nonumber
\\&=   \frac{1}{432}\delta^{\alpha\beta\gamma}_{abr} \left(\delta^p_\gamma \delta^{l    m    q}_{ijk}-\delta^q_\alpha \delta^{l    m    p}_{ijk}  \right) g_{l     \alpha}g_{m    \beta}  \tensor {\vec{B}}{^{ijk}_{p}} \wedge\nabla_q\vec{\Phi}  \,\,                    \,\,                     dx^{a}\wedge_4 dx^b\wedge_4 dx^r     \nonumber
\\&= \frac{1}{72}\delta^{\alpha\beta\gamma}_{abr} \left( g_{i \alpha}g_{j \beta}  \tensor {\vec{B}}{^{ijk}_{\gamma}} \wedge\nabla_k\vec{\Phi} -  g_{i \alpha}g_{j \beta}\tensor {\vec{B}}{^{ijk}_{k}} \wedge\nabla_\gamma\vec{\Phi} \right)\,\,                     \,\,                     dx^{a}\wedge_4 dx^b\wedge_4 dx^r     \nonumber
\\&=-\frac{1}{36}\delta^{\alpha\beta\gamma}_{abr} g_{i \alpha}g_{j \beta} \tensor {\vec{B}}{^{ijk}_{k}} \wedge\nabla_\gamma\vec{\Phi} \,\,                     \,\,                     dx^{a}\wedge_4 dx^b\wedge_4 dx^r  \nonumber
\\&=-\frac{1}{36}\delta^{\alpha\beta\gamma}_{abr} \delta^k_p \tensor {\vec{B}}{_{\alpha\beta k}^{p}} \wedge\nabla_\gamma \vec{\Phi} \,\,                     \,\,                     dx^{a}\wedge_4 dx^b\wedge_4 dx^r   \nonumber
\end{align}
where $ \tensor {\vec{B}}{_{\alpha\beta k}^{p}}$ is antisymmetric in the indices $\alpha, \beta$ and $k$.

\noindent
Thus we arrive at
$$d\vec{R}_0=\star\vec{\eta}\overset{\bullet}\wedge_4 \star d\vec{R}_0 -\star d\vec{u}- \star\vec{\eta}\overset{\bullet}\wedge_4 d\vec{u} +\star\vec{\eta}\wedge_4\star dS_0\,\,   +\vec{Q}                    $$

\noindent
where
$$\vec{Q}:= \frac{1}{12}\delta_{abr}^{\alpha\beta\gamma}( \tensor{\vec{B}}{_{k\alpha}^k_\beta}-\tensor{\vec{B}}{_{k\alpha\beta}^k})\wedge\nabla_\gamma \vec{\Phi}      \,\,                     dx^{a}\wedge_4 dx^b\wedge_4 dx^r. $$

\chapter{Some Computations II}\label{desitt}
\begin{lem}
Let $\vec{\Phi}:\Sigma\rightarrow \mathbb{R}^3$ be an immersion. Let $g$ and $h$ be the first and second fundamental form respectively of $\Sigma$. The trace free second fundamental form satisfies the identity
$$(h_0)_{ik}(h_0)_j^k=\frac{1}{2}|h_0|^2 g_{ij}.$$
\end{lem}
\begin{proof}
We have
\begin{align}
(h_0)_{ik}(h_0)_j^k&= (h_0)_{ik}(h_0)^{sk}g_{sj} \nonumber
\\& =(h_0)_{i1}(h_0)^{11}g_{1j} +(h_0)_{i1}(h_0)^{21}g_{2j} +(h_0)_{i2}(h_0)^{12}g_{1j} +(h_0)_{i2}(h_0)^{22}g_{2j}  \nonumber
\\&= \begin{pmatrix}   
A     &B   \\
   C &D 
 \end{pmatrix}  \nonumber
\end{align}
where 
\begin{align}
A&=     g_{11}\left( (h_0)_{11}^2+(h_0)_{12}^2    \right) +g_{12}(h_0)_{12}\left( (h_0)_{11}+(h_0)_{22}   \right)    \nonumber
\\     
B&=g_{12}\left( (h_0)_{11}^2+(h_0)_{12}^2    \right) +g_{22}(h_0)_{12}\left( (h_0)_{11}+(h_0)_{22}   \right)  \nonumber
\\
C&= g_{11}(h_0)_{12}\left( (h_0)_{11}+(h_0)_{22}    \right) +g_{21}\left( (h_0)_{12}^2+(h_0)_{22}^2   \right)   \nonumber
\\
D&=     g_{12}(h_0)_{12}\left( (h_0)_{11}+(h_0)_{22}    \right) +g_{22}\left( (h_0)_{12}^2+(h_0)_{22}^2   \right) .   \nonumber
\end{align}
Since $h_0$ is trace free, we have $(h_0)_{11}=-(h_0)_{22}$. This gives
\begin{align}
(h_0)_{ik}(h_0)_j^k&= \begin{pmatrix}   
A     &B   \\
   C &D 
 \end{pmatrix}  \nonumber
\\&=\begin{pmatrix}   
g_{11}\left((h_0)_{11}^2+ (h_0)_{12}^2  \right)     &g_{12}\left((h_0)_{11}^2+ (h_0)_{12}^2  \right)    \\
   g_{21}\left((h_0)_{12}^2+ (h_0)_{11}^2  \right)  &g_{22}\left((h_0)_{12}^2+ (h_0)_{11}^2  \right)  
 \end{pmatrix}  \nonumber
\\&=\left((h_0)_{11}^2+ (h_0)_{12}^2  \right) g_{ij}  \nonumber
\\&=\frac{1}{2} \left((h_0)_{11}^2+ 2(h_0)_{12}^2 +(h_0)_{22}^2 \right) g_{ij}  \nonumber
\\&=\frac{1}{2} |h_0|^2 g_{ij}\label{des23}
\end{align}
as announced.

\end{proof}

Let $\vec{\Phi}:\Sigma\rightarrow\mathbb{R}^3$ be an immersion. Let $\vec{Y}:\Sigma\rightarrow \mathbb{S}^{4,1}$ be the conformal Gauss map of $\vec{\Phi}$ into the deSitter space, defined via 
\begin{align}
 \vec{Y}:= \vec{X}H+\vec{N}  \nonumber
\end{align}
\begin{align}
\mbox{where}\quad\quad\quad\quad\quad\vec{X}=\begin{pmatrix}
    \vec{\Phi}        \\
    \frac{1}{2}(|\vec{\Phi}|^2-1)      \\
    \frac{1}{2}(|\vec{\Phi}|^2+1)      
\end{pmatrix} \quad\mbox{and}\,\,\,\,\vec{N}=\begin{pmatrix}
    \vec{n}        \\
    \vec{n}\cdot\vec{\Phi}      \\
           \vec{n}\cdot\vec{\Phi}
\end{pmatrix} . 
    \nonumber
\end{align}

\noindent
We have
\begin{align}
\nabla_i \vec{Y}= \vec{X}\nabla_i H+H\nabla_i \vec{X}+\nabla_i \vec{N}.                \nonumber
\end{align}
Observe that
\begin{align}H\nabla_i\vec{\Phi}+\nabla_i\vec{n}&=H\nabla^k\vec{\Phi}g_{ik}-h_{ik}\nabla^k\vec{\Phi}\nonumber\\
&=-(h_0)\nabla^k\vec{\Phi}
\end{align}
where $h_0$ is the trace free second fundamental form.
Similarly, 
\begin{align}
H(\vec{\Phi}\cdot\nabla_i\vec{\Phi})+ (\vec{\Phi}\cdot\nabla_i\vec{n})= H(\vec{\Phi}\cdot\nabla^k\vec{\Phi})g_{ik}-h_{ik}(\vec{\Phi}\cdot\nabla^k\vec{\Phi})  =-(h_0)_{ij}(\vec{\Phi}\cdot\nabla^j\vec{\Phi}) .\nonumber
\end{align}
\noindent
Thus we have found that
\begin{align}
\nabla_i \vec{Y}= \vec{X}\nabla_i H+ (h_0)_{ij}\vec{Z}^j \quad \mbox{where}\,\,\, \vec{Z}^j=-\begin{pmatrix}   \nabla^j\vec{\Phi}\\ (\vec{\Phi}\cdot\nabla^j\vec{\Phi})\\(\vec{\Phi}\cdot\nabla^j\vec{\Phi})          \end{pmatrix}=-\nabla^j \vec{X}. \nonumber
\end{align}
We see that
\begin{align}
\langle  \vec{X},\vec{X}   \rangle_{\mathbb{R}^{4,1}}&=|\vec{\Phi}|^2 +\frac{1}{4}(|\vec{\Phi}|^2-1)^2-\frac{1}{4}(|\vec{\Phi}|^2+1)^2  \nonumber
\\&=0 . \nonumber
\end{align}
Similarly
\begin{align}
-\langle  \vec{X}, \vec{Z}  \rangle_{\mathbb{R}^{4,1}}  &=\vec{\Phi}\cdot\nabla^j\vec{\Phi} +\frac{1}{2}(\vec{\Phi}\cdot\nabla^j\vec{\Phi}) (|\vec{\Phi}|^2-1) - \frac{1}{2}(\vec{\Phi}\cdot\nabla^j\vec{\Phi}) (|\vec{\Phi}|^2+1)  \nonumber
\\&=0.   \nonumber
\end{align}
Also
\begin{align}
\langle \vec{Z}^k, \vec{Z}^m \rangle_{\mathbb{R}^{4,1}}= (\nabla^k\vec{\Phi}\cdot\nabla^m\vec{\Phi}) +(\vec{\Phi}\cdot\nabla^k\vec{\Phi})(\vec{\Phi}\cdot\nabla^m\vec{\Phi})-(\vec{\Phi}\cdot\nabla^k\vec{\Phi})(\vec{\Phi}\cdot\nabla^m\vec{\Phi})= g^{km}            \nonumber
\end{align}

\noindent
Thus
\begin{align}
\left\langle  \nabla_i \vec{Y}, \nabla_j \vec{Y}  \right\rangle_{\mathbb{R}^{4,1}} & = \left\langle    \vec{X}\nabla_i H+(h_0)_{ik} \vec{Z}^k\,\,, \,\,\vec{X}\nabla_j H+(h_0)_{jm} \vec{Z}^m       \right\rangle_{\mathbb{R}^{4,1}}     \nonumber
\\&   =(h_0)_{ik} (h_0)_{jm} g^{km}=(h_0)_{ik}(h_0)_j^k\overset{\eqref{des23}}=\frac{1}{2} |h_0|^2g_{ij}. \nonumber
\end{align}
Thus we have
$$\left\langle  \nabla_i \vec{Y}, \nabla^i \vec{Y}  \right\rangle_{\mathbb{R}^{4,1}}=|h_0|^2.$$
We have just shown that the Willmore energy correspond to the Dirichlet energy of $Y$ in the Minkowski space, that is $\lVert d\vec{Y} \rVert_{\mathbb{R}^{4,1}}^2=\int_{\Sigma}|h_0|^2$.
How about four dimensions? Does the energy \eqref{jkas} correspond to some energy involving $\vec{Y}$? A reasonable guess is to consider the biharmonic energy of $\vec{Y}$.\\
If we write $\Delta \vec{Y}:= \nabla^i\nabla_i \vec{Y}$, then
\begin{align}
\Delta \vec{Y}= \nabla^i \vec{X}\nabla_i H + \vec{X} \Delta H+ \nabla^i(h_0)_{ij} \vec{Z}^j + (h_0)_{ij}\nabla^i \vec{Z}^j       .     \nonumber
\end{align}

\noindent
We have
\begin{align}
&\langle  \nabla^i \vec{X}, \nabla^k \vec{X}  \rangle_{\mathbb{R}^{6,1}}=\langle  -\vec{Z}^i, -\vec{Z}^k \rangle_{\mathbb{R}^{6,1}} =g^{ik}\quad,  \nonumber\hspace{1cm} \langle \nabla^i \vec{X}, \vec{X}  \rangle_{\mathbb{R}^{6,1}}= -\langle \vec{Z}^i, \vec{X} \rangle_{\mathbb{R}^{6,1}}=0   
\\ &    \langle  \nabla^i \vec{X}, \vec{Z}^m  \rangle_{\mathbb{R}^{6,1}} =-\langle \vec{Z}^i, \vec{Z}^m \rangle_{\mathbb{R}^{6,1}}     =-g^{im} \quad,               \hspace{1cm}   \langle \nabla^i \vec{X}, \nabla^k \vec{Z}^m   \rangle_{\mathbb{R}^{6,1}}=-\langle \vec{Z}^i, \nabla^k \vec{Z}^m  \rangle =0\nonumber
\\ &\langle \vec{Z}^j, \nabla^k \vec{Z}^m \rangle_{\mathbb{R}^{6,1}}= \nabla^j\vec{\Phi}\cdot\vec{h}^{km}=0\quad\,  ,\nonumber  \hspace{1.5cm} \langle \nabla^i \vec{Z}^j, \nabla^k \vec{Z}^m  \rangle_{\mathbb{R}^{6,1}}= \vec{h}^{ij}\cdot \vec{h}^{km}\nonumber
\end{align}
and
\begin{align}
& \langle  \vec{X},\nabla^k \vec{Z}^m  \rangle_{\mathbb{R}^{6,1}}  \nonumber
\\&=-\left[ (\vec{\Phi}\cdot\vec{h}^{km})+\frac{1}{2}(|\vec{\Phi}|^2-1)(g^{km}+\vec{\Phi}\cdot\vec{h}^{km})  
-\frac{1}{2}(|\vec{\Phi}|^2+1)(g^{km}+\vec{\Phi}\cdot\vec{h}^{km})  \right] \nonumber
 =g^{km} \nonumber
\end{align}
Now,
\begin{align}
||\Delta \vec{Y}||^2 &=\langle \Delta \vec{Y}, \Delta \vec{Y}   \rangle_{\mathbb{R}^{6,1}}          \nonumber
\\&= \left\langle \nabla^i \vec{X}\nabla_i H + \vec{X} \Delta H+ \nabla^i((h_0)_{ij} \vec{Z}^j)\,\, ,\,\,\,\nabla^i \vec{X}\nabla_i H + \vec{X} \Delta H+ \nabla^i((h_0)_{ij} \vec{Z}^j) \right\rangle_{\mathbb{R}^{6,1}}        \nonumber
\\&=\nabla_i H\nabla_k g^{ik}+\nabla_i H\nabla^k(h_0)_{km} g^{im} + \Delta H (h_0)_{km} g^{km}+ \nabla^i (h_0)_{ij} \nabla_k H g^{kj}   \nonumber
\\&\quad - \nabla ^i (h_0)_{ij} \nabla^k (h_0)_{km} g^{jm}+ (h_0)_{ij} \Delta H g^{ij} +(h_0)_{ij}(h_0)_{km}(\vec{h}^{ij}\cdot\vec{h}^{km})                      \nonumber
\\&=\nabla_i H\nabla^i H+\nabla_i H\nabla^k(h_0)^i_k +\Delta H (h_0)_{i}^{i}+\nabla^i(h_0)_{ij}\nabla ^j H    \nonumber
\\&\quad-\nabla^i(h_0)_{ij}\nabla^k(h_0)^j_k + (h_0)^i_i\Delta H+(h_0)_{ij}(h_0)_{km}(\vec{h}^{ij}\cdot\vec{h}^{km})                             \nonumber
\end{align}

\noindent
Noting that $(h_0)_i^i=0$ and using Codazzi equation, we find that 
\begin{align}
\nabla^i(h_0)_{ij}\nabla^j H=3\nabla_j H\nabla^j H \quad \mbox{and}\quad\nabla^i(h_0)_{ij}\nabla^k(h_0)_k^j=9 \nabla_j H\nabla^j H   .  \nonumber
\end{align}

Hence,
\begin{align}
||\Delta \vec{Y}||_{\mathbb{R}^{6,1}}^2&=-2\nabla_i H\nabla^i H + (h_0)_{ij}(h_0)_{km} (\vec{h}^{ij}\cdot\vec{h}^{km})     \nonumber
\\&= -2|\nabla H|^2 + (h_0)_{ij}(h_0)_{km} h^{ij}h^{km}         \nonumber
\\&= -2|\nabla H|^2+ (h_0)_{ij}(h_0)_{km} \left((h_0)^{ij}+Hg^{ij} \right)\left((h_0)^{km} +H g^{km}\right)\nonumber
\\&= -2|\nabla H|^2+ |h_0|^4.                  \nonumber
\end{align}
Clearly the biharmonic energy of $\vec{Y}$ is not the energy \eqref{jkas}. However, it correctly reproduces the leading order term of \eqref{jkas}.


\end{document}